\documentclass[toc=nobib,toc=idx,abstract=false,parskip=half]{scrartcl}
\usepackage[utf8]{inputenc}
\usepackage[T1]{fontenc}
\usepackage[english]{babel}
\usepackage{tocbasic}
\usepackage{csquotes}
\usepackage{mathrsfs}
\usepackage{amsmath}
\usepackage{amsthm}
\usepackage{thmtools}
\usepackage{amsfonts}
\usepackage{amssymb}
\usepackage{stmaryrd}
\usepackage{mathtools}
\usepackage{mathdots}
\usepackage{enumitem}
\usepackage[normalem]{ulem} 
\usepackage{booktabs}
\usepackage{calc}
\usepackage{nicefrac}
\usepackage{graphicx}
\usepackage[dvipsnames]{xcolor}
\usepackage[style=alphabetic,citestyle=alphabetic,maxcitenames=4,maxalphanames=4,maxbibnames=10,backref=false,hyperref=true,backend=biber,doi=false,isbn=false,giveninits=true,sorting=nyt,backref=true,backrefstyle=none]{biblatex}
\usepackage[colorlinks=true,linkcolor=black,urlcolor=black,citecolor=black,bookmarksnumbered,linktocpage=true]{hyperref}
\usepackage{tikz}

\usetikzlibrary{arrows,decorations.markings,chains,calc,matrix,patterns,cd}
\usetikzlibrary{decorations.pathmorphing}
\tikzset{fuzzy/.style={decorate, decoration={random steps,segment length=3pt,amplitude=2pt}}}

\bibliography{triangle}
\DeclareFieldFormat{postnote}{#1}	
\DeclareFieldFormat{multipostnote}{#1}
\DeclareFieldFormat[article,eprint,unpublished,incollection]{title}{\mkbibemph{#1}}
\DeclareFieldFormat[article,eprint,unpublished,incollection]{journaltitle}{#1}
\DeclareFieldFormat[incollection]{booktitle}{#1}
\DefineBibliographyStrings{english}{%
    backrefpage  = {$\uparrow$},
    backrefpages = {$\uparrow$}
}

\DeclareTOCStyleEntry{dottedtocline}{section}

\AtBeginDocument{%

}

\declaretheoremstyle[spaceabove=\topsep,spacebelow=0pt,bodyfont=\normalfont]{scdef}
\declaretheoremstyle[spaceabove=\topsep,spacebelow=0pt,bodyfont=\itshape]{scthm}
\declaretheoremstyle[spaceabove=\topsep,spacebelow=0pt,headfont=\normalfont\itshape,notefont=\normalfont\itshape,notebraces={}{},headformat={\NAME\NOTE},postheadspace=1em,qed=\qedsymbol]{scprf}

\declaretheorem[style=scthm,numberwithin=section,name=Theorem,    refname={Theorem,Theorems},        Refname={Theorem,Theorems}]        {Thm}
\declaretheorem[style=scthm,sharenumber=Thm,     name=Lemma,      refname={Lemma,Lemmas},            Refname={Lemma,Lemmas}]            {Lem}
\declaretheorem[style=scthm,sharenumber=Thm,     name=Corollary,  refname={Corollary,Corollaries},   Refname={Corollary,Corollaries}]   {Cor}
\declaretheorem[style=scthm,sharenumber=Thm,     name=Proposition,refname={Proposition,Propositions},Refname={Proposition,Propositions}]{Prop}

\declaretheorem[style=scthm,sharenumber=Thm,     name=Question,   refname={Question,Questions},      Refname={Question,Questions}]      {Ques}
\declaretheorem[style=scdef,sharenumber=Thm,     name=Definition, refname={Definition,Definitions},  Refname={Definition,Definitions}]  {Def}
\declaretheorem[style=scdef,sharenumber=Thm,     name=Remark,     refname={Remark,Remarks},          Refname={Remark,Remarks}]          {Rem}

\declaretheorem[style=scthm,sharenumber=Thm,     name=Fact,       refname={Fact,Facts},              Refname={Fact,Facts}]              {Fact}
\declaretheorem[style=scprf,unnumbered,          name=Proof]{Prf}

\makeatletter
\def\blfootnote{\gdef\@thefnmark{}\@footnotetext}
\makeatother

\setlist[enumerate,1]{label={(\roman*)}}

\DeclareTextFontCommand{\red}{\bfseries\color{red}}
\DeclareTextFontCommand{\green}{\bfseries\color{ForestGreen}}

\makeatletter\def\reallynopagebreak{\par\nopagebreak\@nobreaktrue}\makeatother

\def\svdots{\vbox{\baselineskip=4pt \lineskiplimit=0pt \kern2pt \hbox{.}\hbox{.}\hbox{.}\vspace{1pt}}}

\DeclareMathOperator{\id}{id}

\DeclareMathOperator{\tr}{tr}

\DeclareMathOperator{\SL}{SL}

\DeclareMathOperator{\PGL}{PGL}
\DeclareMathOperator{\SO}{SO}
\DeclareMathOperator{\rO}{O}

\DeclareMathOperator{\Hom}{Hom}
\DeclareMathOperator{\Span}{span}

\DeclareMathOperator{\Sym}{Sym}

\DeclareMathOperator{\diam}{diam}
\DeclareMathOperator{\dist}{dist}
\DeclareMathOperator{\Haus}{Haus}

\usepackage{pinlabel}

\DeclareMathOperator{\Isom}{Isom}

\DeclareMathOperator{\PU}{PU}

\newcommand{\F}{\mathcal F}
\newcommand{\dd}{\mathrm d}

\newcommand{\RP}{\mathbb{R}\mathrm{P}}

\newcommand{\CHyp}{\mathbb{C}\mathbb{H}}

\newcommand{\bR}{\mathbb{R}}
\newcommand{\bC}{\mathbb{C}}
\newcommand{\bZ}{\mathbb{Z}}
\newcommand{\bN}{\mathbb{N}}
\newcommand{\bH}{\mathbb{H}}

\newcommand{\tred}{t_{\mathrm{red}}}
\newcommand{\tcrit}{t_{\mathrm{crit}}}

\newcommand{\bx}{\square}
\newcommand{\bxp}{\square'}
\newcommand{\bxpp}{\square''}

\newcommand{\CH}[2]{\mathrm{CH}_{#1}(#2)}

\begin{document}

\title{Anosov triangle reflection groups in \texorpdfstring{$\SL(3,\bR)$}{SL(3,R)}}

\author{Gye--Seon Lee, Jaejeong Lee \& Florian Stecker}
\date{}

\maketitle
\blfootnote{\textup{2020} \textit{Mathematics Subject Classification.} 22E40, 51F15, 57S30}
\blfootnote{\textit{Key words and phrases.} reflection groups, discrete subgroups of Lie groups, Anosov representations}

\begin{abstract}
  We identify all Anosov representations of compact hyperbolic triangle reflection groups into the higher rank Lie group $\SL(3,\bR)$. Specifically, we prove that such a representation is Anosov if and only if either it lies in the Hitchin component of the representation space, or it lies in the ``Barbot component'' and the product of the three generators of the triangle group has distinct real eigenvalues.
  Unlike representations in the Hitchin component, Anosov representations in the Barbot component have non--convex boundary maps.
\end{abstract}

\tableofcontents


\section{Introduction}

Given a finitely generated group $\Gamma$ and a Lie group $G$, it is a natural problem to find all discrete subgroups of $G$ isomorphic to $\Gamma$, or more precisely, all discrete and faithful representations of $\Gamma$ into $G$. When $\Gamma$ is a fundamental group of a manifold or orbifold $M$, this problem is closely linked to the study of geometric structures on $M$ \cite{GoldmanBook}.

If $\Gamma$ is the fundamental group of a closed surface $S$ of genus $g \geq 2$ and $G = \mathrm{PGL}(2,\mathbb{R})$, which is isomorphic to $\mathrm{Isom}(\mathbb{H}^2)$, the isometry group of the hyperbolic plane, then the discrete and faithful representations $\Gamma \rightarrow G$ are fully understood: they form a union of two connected components of the representation space $\mathrm{Hom}(\Gamma,G)$, and each  representation up to conjugation corresponds to a hyperbolic structure on $S$.
Each of these components modulo conjugation is called the \emph{Teichmüller space} of $S$.

In this paper, we are interested in the problem where $\Gamma$ is a compact hyperbolic triangle reflection group and $G = \PGL(3,\bR) \cong \SL(3,\bR)$, the group of automorphisms of the projective plane $\RP^2$. Even for such ``small'' $\Gamma$ and $G$, it is still an open problem to understand all discrete and faithful representations of $\Gamma$ into $G$. Nevertheless, our main theorem characterizes all representations that are \emph{Anosov}, a strengthening of discrete (see \autoref{thm:main}).
This raises the question:

\begin{Ques}
Let $\Gamma$ be a compact hyperbolic triangle reflection group. Is a representation $\rho \colon \Gamma \rightarrow \PGL(3,\bR)$ discrete and faithful if and only if $\rho$ lies in the closure of the space of Anosov representations in $\mathrm{Hom}(\Gamma, \PGL(3,\bR))$?
\end{Ques}

\subsection{Anosov representations}

Anosov representations are discrete representations of a word--hyperbolic group $\Gamma$ into a Lie group $G$ with good dynamical properties. They have received a lot of attention and have been actively studied in recent years; see for example \cite{Labourie,GuichardWienhard,KLP1,GGKW,BPS}. Anosov representations have two key properties that set them apart from general discrete ones. The first is the existence of {\em boundary maps} (see \autoref{def:Anosov}), which in fact characterizes Anosov representations with Zariski dense image. The second key property is {\em openness}: small deformations of Anosov representations are also Anosov. Hence if an Anosov representation is not isolated, it provides a family of new discrete representations.

Examples of Anosov representations of surface groups include \emph{Hitchin representations} in a real split simple Lie group like $\PGL(3,\bR)$, or \emph{maximal representations} in a simple Lie group of Hermitian type.
Similarly to surface group representations in $\PGL(2,\bR)$, these representations form a closed and open subset of the representation space, hence a union of connected components.
Such a component is called a {\em higher Teichm\"{u}ller space}; see \cite{GuichardWienhard2,Wienhard}.

In general, however, the set of Anosov representations is not closed in the representation space. An example is the component of surface group representations in $\mathrm{PGL}(3,\mathbb{R})$ that contains a discrete and faithful representation whose action on the projective plane fixes a point and preserves a line disjoint from it; see \cite{Barbot}. It is known to contain Anosov representations as well as non--Anosov representations. The shape of the space of Anosov representations, or even the number of its connected components, is not known in this case.

\subsection{Results}

The space of representations of a surface group into $\SL(3,\mathbb{R})$ is too high--dimensional for a classification of Anosov representations to be feasible. For instance, the space of Hitchin representations for a closed surface of genus $g \geq 2$ (modulo conjugation) has dimension $16g - 16$.
One can decrease the dimension by specializing to surface groups with more symmetries (see \autoref{rem:surface_groups}).
So we focus on the compact hyperbolic reflection group
\[\Gamma = \Gamma_{p_1,p_2,p_3} = \langle s_1,s_2,s_3 \mid s_1^2 = s_2^2 = s_3^2 = (s_2s_3)^{p_1} = (s_3s_1)^{p_2} = (s_1s_2)^{p_3} = 1 \rangle\]
where $2 \leq p_1 \leq p_2 \leq p_3 < \infty$ and $\frac{1}{p_1} + \frac{1}{p_2} + \frac{1}{p_3} < 1$. It is isomorphic to the group generated by reflections along the sides of the triangle with dihedral angles $\frac{\pi}{p_1},\frac{\pi}{p_2},\frac{\pi}{p_3}$ in the hyperbolic plane. Then the space of characters $\chi(\Gamma,\SL(3,\bR))$, which is the space of semisimple representations modulo conjugation, has dimension $0$ or $1$; see \autoref{lem:chi_coxeter} and \autoref{lem:homeoToR}.

  As in the surface group case, there is a unique \emph{Hitchin component} in $\chi(\Gamma,\SL(3,\bR))$, consisting of those representations which can be continuously deformed to a discrete and faithful representation into $\SO(2,1)$ \cite{ChoiGoldman, ALS}.
  If $p_1,p_2,p_3$ are all odd, there is also a unique component containing discrete and faithful representations whose action in the projective plane fixes a point and preserves a line disjoint from it (see \autoref{sec:reducible}).
  We call it the \emph{Barbot component} as it is analogous to the component studied in \cite{Barbot}.
  Our main theorem is
  
\begin{Thm}\label{thm:main}
  Let $\rho\colon \Gamma_{p_1,p_2,p_3} \rightarrow \mathrm{SL}(3,\mathbb{R})$ be a representation.
  Then $\rho$ is Anosov if and only if
  \begin{enumerate}
  \item either $\rho$ is in the Hitchin component,
  \item or $p_1,p_2,p_3$ are odd, $\rho$ is in the Barbot component, and $\rho(s_1 s_2 s_3)$ has distinct real eigenvalues.
  \end{enumerate}

  In case (ii), the set of Anosov characters in the Barbot component is the complement of a compact interval, as sketched in \autoref{fig:barbot_component}.
\end{Thm}

It is known that every representation in the Hitchin component is Anosov and has a convex boundary map \cites{ChoiGoldman,Labourie}.
In contrast, the boundary maps of Anosov representations in the Barbot component are not convex.
Our main contribution in \autoref{thm:main} is to construct this non--convex boundary map assuming that $\rho(s_1s_2s_3)$ has distinct eigenvalues, from which we can deduce the Anosov property.
This construction uses the geometry of conics in $\RP^2$ and ideas inspired by the works of Schwartz \cite{Schwartz} and Sullivan \cite[Section 9]{Sullivan}.
To complete the proof of \autoref{thm:main}, we use topological constraints imposed by the boundary map to show that there are no Anosov examples in the other components.

  As far as the authors know, \autoref{thm:main} is the first instance where a \emph{non--closed} Anosov space in $\Hom(\Gamma, G)$ is completely identified for a non--elementary hyperbolic group $\Gamma$ and a \emph{higher rank} simple Lie group $G$.

\begin{center}
  \begin{tikzpicture}
    \begin{scope}[yshift=-3.2cm]
      \begin{scope}[xshift=1.35cm]
        \node at (0,0) {\includegraphics[width=4.5cm,bb=0cm 2cm 33.5cm 22cm,clip]{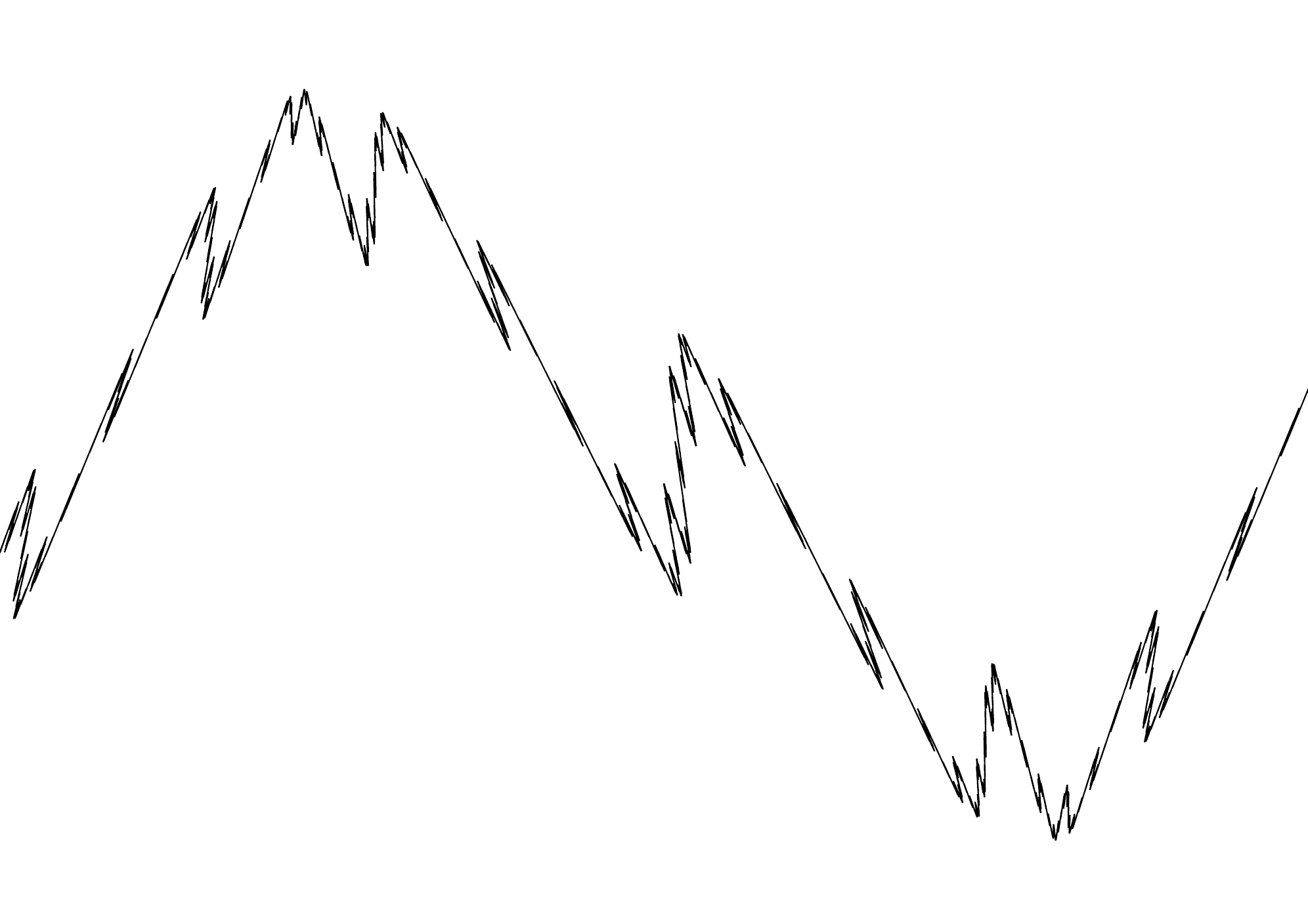}};
        \draw (-2.25,-1.5) -- (-2.25,1.5) -- (2.25,1.5) -- (2.25,-1.5) -- cycle;
        \node[orange] at (2,1.2) {A};
      \end{scope}
      \begin{scope}[xshift=6cm]
        \node at (0,0) {\includegraphics[width=4.5cm,bb=0cm 2cm 33.5cm 22cm,clip]{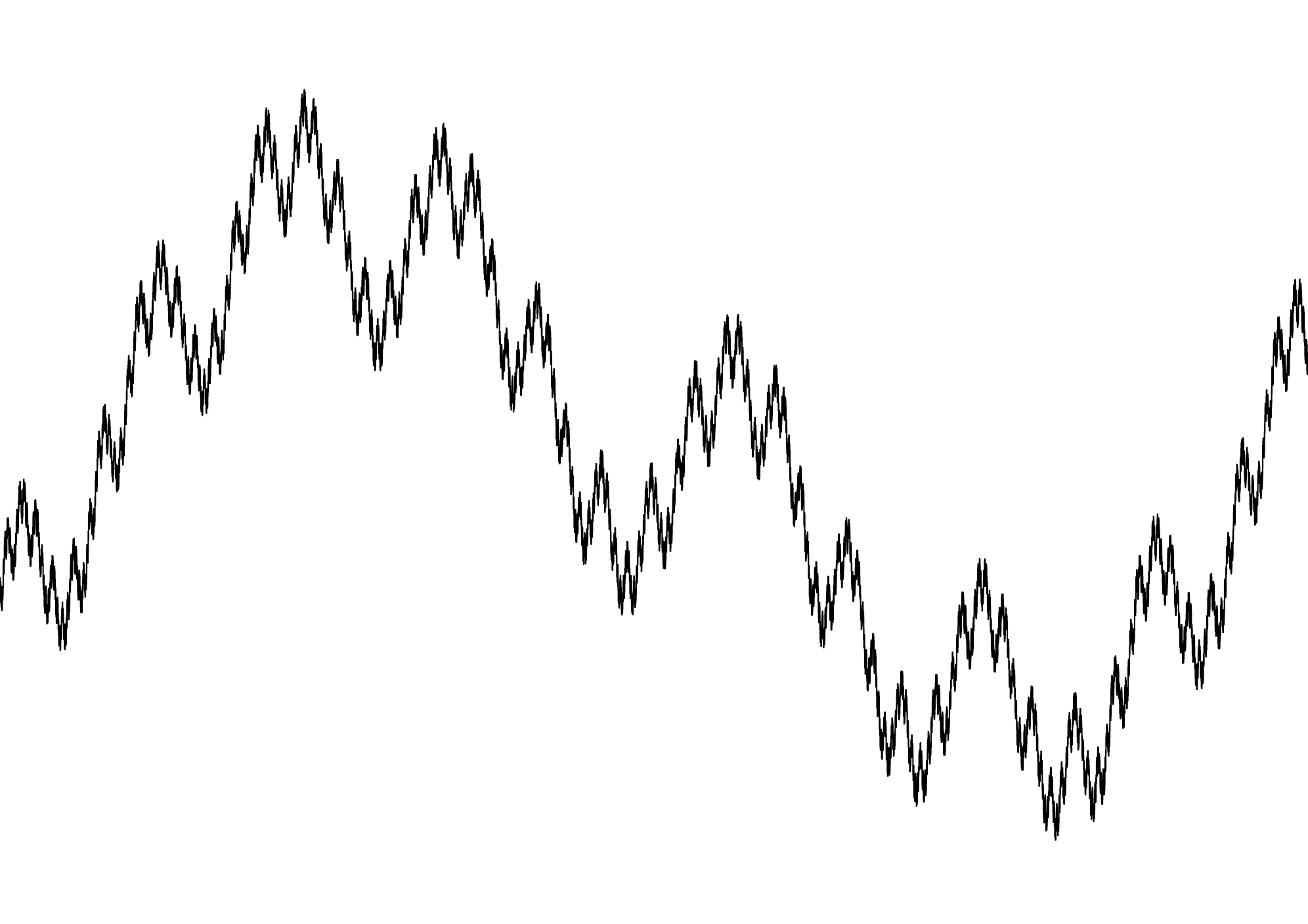}};
        \draw (-2.25,-1.5) -- (-2.25,1.5) -- (2.25,1.5) -- (2.25,-1.5) -- cycle;
        \node[purple] at (2,1.2) {B};
      \end{scope}
      \begin{scope}[xshift=10.65cm]
        \node at (0,0) {\includegraphics[width=4.5cm,bb=0cm 2cm 33.5cm 22cm,clip]{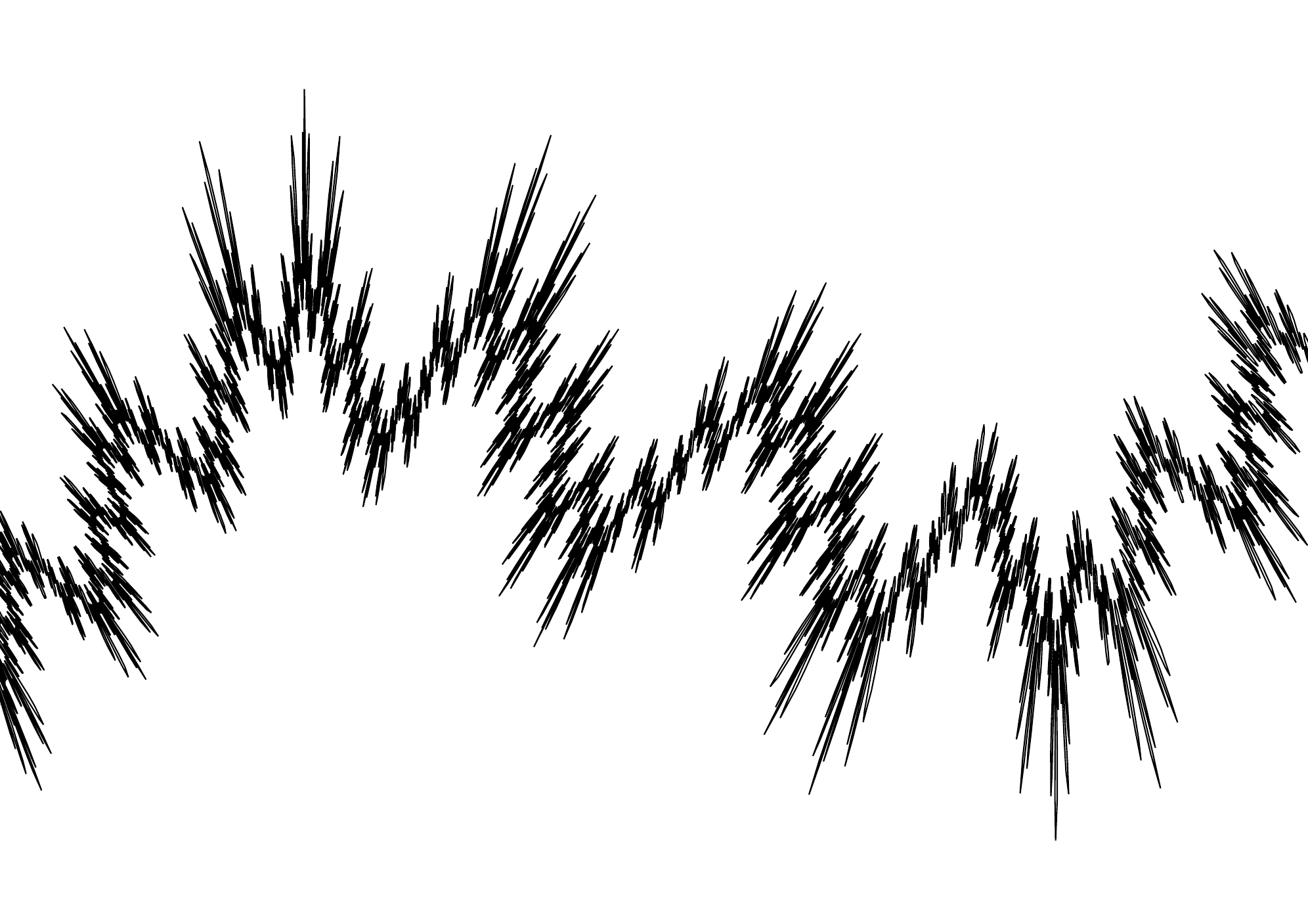}};
        \draw (-2.25,-1.5) -- (-2.25,1.5) -- (2.25,1.5) -- (2.25,-1.5) -- cycle;
        \node[blue] at (2,1.2) {C};
      \end{scope}
    \end{scope}

    \begin{scope}[yshift=-0.5cm]
      \draw[gray,thick] (0,0) -- (12,0);
      \draw[(->,very thick,ForestGreen] (8,0) -- (12,0);
      \draw[(->,very thick,ForestGreen] (4,0) -- (0,0);
      \fill[ForestGreen] (2,0) circle (0.08);
      \fill[ForestGreen] (10,0) circle (0.08);
      
      \draw[->,thick] (1,-0.7) node[anchor=north] {\scriptsize reducible} -- (1.9,-0.1);
      \draw[->,thick] (11,-0.7) node[anchor=north] {\scriptsize reducible} -- (10.1,-0.1);
      
      \node[anchor=north] at (6,-0.4) {\scriptsize discrete faithful};
      \node[anchor=north] at (6,-0.7) {\scriptsize but not Anosov};
      \draw[->,thick] (5,-0.4) -- (4.1,-0.12);
      \draw[->,thick] (7,-0.4) -- (7.9,-0.12);
      
      \node[ForestGreen] at (2,1.0) {\scriptsize $\rho$ Anosov};
      \node[ForestGreen] at (2,0.7) {\scriptsize $\rho(s_1s_2s_3)$ has};
      \node[ForestGreen] at (2,0.4) {\scriptsize distinct real eigenvalues};
      \node[ForestGreen] at (10,1.0) {\scriptsize $\rho$ Anosov};
      \node[ForestGreen] at (10,0.7) {\scriptsize $\rho(s_1s_2s_3)$ has};
      \node[ForestGreen] at (10,0.4) {\scriptsize distinct real eigenvalues};
            
      \node[gray] at (6,1.0) {\scriptsize $\rho$ not Anosov};
      \node[gray] at (6,0.7) {\scriptsize $\rho(s_1s_2s_3)$ has};
      \node[gray] at (6,0.4) {\scriptsize non--real eigenvalues};
      
      \node[orange] at (0.5,-0.35) {A};
      \node[purple] at (2,-0.35) {B};
      \node[blue] at (3.95,-0.35) {C};      
    \end{scope}
  \end{tikzpicture}
  
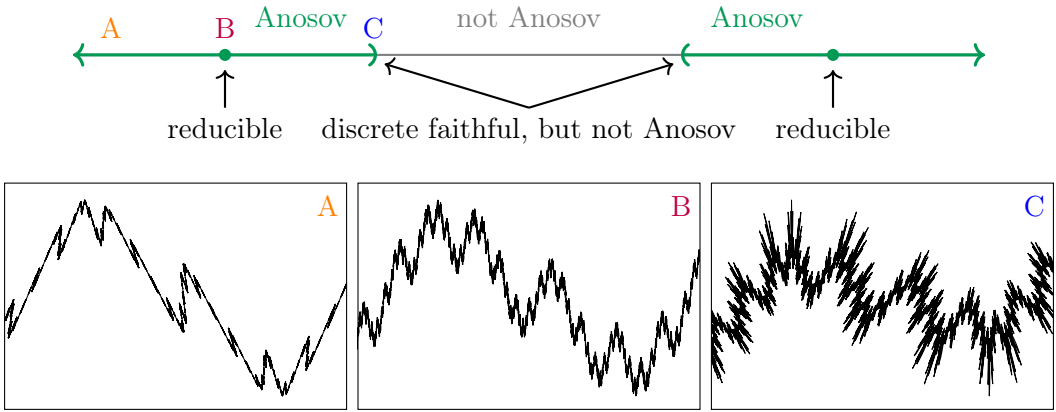
\captionof{figure}{(Top) A sketch of the Barbot component showing the two open intervals of Anosov representations, each of them containing a single reducible representation. (Bottom) Images of boundary maps into $\RP^2$ for three different representations of the $(3,3,5)$ triangle group. In case (B) the reducible representation shown is not semisimple, but line--irreducible (see \autoref{def:irreducible}).}
  \label{fig:barbot_component}
\end{center}

It is known that the representations on the boundary of the Anosov set are still discrete and faithful; see \autoref{rem:discrete_faithful_closed}. We show in \autoref{thm:barbot_transverse} that they also admit continuous injective boundary maps.
However, they fail to be Anosov as their boundary maps are not \emph{transverse}; see \autoref{sec:transversality}.

An interesting consequence of \autoref{thm:main}, which demonstrates the explicitness of its criterion, is that the Anosov property can be checked with a few equalities and inequalities involving only the traces of group elements up to word length $3$.
Writing this out, we obtain the following (see also \autoref{lem:coxeter_eigenvalues} and \autoref{fig:goldman}).

\begin{Cor}\label{cor:traces}
  Let $\rho \colon \Gamma_{p_1,p_2,p_3} \to \SL(3,\bR)$ be a representation. Assume $p_1,p_2,p_3 \geq 3$ and define $c_k = 2\cos\frac{\pi}{p_k}$ for $k \in \{1,2,3\}$, and
  \[t_1 = \tr \rho(s_2s_3), \ \  t_2 = \tr \rho(s_3s_1), \ \  t_3 = \tr \rho(s_1s_2), \ \  x = \tr \rho(s_1s_2s_3), \ \  y = \tr \rho(s_3s_2s_1).\]
  Then $\rho$ is Anosov if and only if one of the following holds:
  \begin{enumerate}
  \item $t_k = c_k^2 - 1$ for all $k$ and $x + t_1 + t_2 + t_3 < 0$, or
  \item $p_1,p_2,p_3$ are odd, $t_k = 1 - c_k$ for all $k$, and $x^2y^2 - 4x^3 - 4y^3 + 18xy - 27 > 0$.
  \end{enumerate}
\end{Cor}

\begin{Rem}\label{rem:surface_groups}
Our result also gives some information about Anosov representations of surface groups: when $p_1$, $p_2$ and $p_3$ are odd, the fundamental group $\pi_1(S_g)$ of the orientable surface $S_g$ of genus $g$ is a subgroup of $\Gamma_{p_1,p_2,p_3}$ of finite index if and only if
$$g = \frac{k}{2} \left( 1 - \frac{1}{p_1} - \frac{1}{p_2} - \frac{1}{p_3}  \right) \mathrm{lcm}(p_1,p_2,p_3) + 1 \quad \textrm{for any } k \in \mathbb{N} $$
where $\mathrm{lcm}(p_1,p_2,p_3)$ denotes the least common multiple of $p_1$, $p_2$ and $p_3$; see \cite{EEK}. In that case, the representations of $\Gamma_{p_1,p_2,p_3}$ provide families of representations of the surface group $\pi_1(S_g)$, and among such surface group representations our main theorem characterizes the Anosov ones.
\end{Rem}

\begin{Rem}
  Instead of compact hyperbolic triangle reflection groups, one may consider the ``ideal hyperbolic triangle'' reflection group $W_3 = \mathbb{Z}/2 * \mathbb{Z}/2 * \mathbb{Z}/2$ along with certain unipotent conditions in order to obtain representations that preserve a circular limit curve.
    Such representations are not Anosov but may be regarded as ``relatively Anosov'' in a suitable sense \cite{KLRelativeAnosov,Zhu,Weisman,ZhuZimmer1}.

    When the target Lie group $G$ is a non--compact real form of $\SL(3,\bC)$ the corresponding (relative) character space is also $1$--dimensional and similar results have been obtained.
    More specifically, when $G = \Isom(\CHyp^2)\cong\PU(2,1)$, the isometry group of the {\em complex} hyperbolic plane, Goldman and Parker \cite{GoldmanParker} considered the representations $W_3\to G$ that map the standard generators of $W_3$ to distinct, order two, complex reflections and satisfy the condition that any product of two distinct generators is parabolic.
    Among those representations, they conjectured exactly which ones are discrete and faithful.
    The conjecture was proved by Schwartz in \cite{Schwartz_Annals,Schwartz_GT}.
    Analogously, when $G=\SL(3,\bR)$, Kim and Lee \cite{KimLee} identified representations $W_3\to G$ with an invariant circular limit curve in the flag manifold, among the representations that map the standard generators of $W_3$ to distinct involutions and satisfy the condition that any product of two distinct generators is ``quasi--unipotent''.
\end{Rem}

\begin{Rem}
  Let $\Gamma_{N,\infty,\infty}$ be the reflection group obtained from a non--compact hyperbolic triangle with one vertex of angle $\pi/N$ and two ideal vertices. Recently, Filip and Fougeron \cite{FilipFougeron} constructed a relative Anosov representation $\rho_N : \Gamma_{N,\infty,\infty} \rightarrow \mathrm{PSp}(4,\mathbb{R})$. They introduced ``cones'' in $\mathbb{R}^4$ on which the group $\Gamma_{N,\infty,\infty}$ plays ping--pong, and related these cones with ``crooked surfaces'' to produce a non--empty domain of discontinuity in the Lagrangian Grassmannian $\mathrm{LGr}(\mathbb{R}^4)$.
\end{Rem}

\subsection{Overview}

In \autoref{sec:prelims} we parametrize the space of characters, and review some properties of triangle reflection groups and Anosov representations.
We also give a proof of \autoref{cor:traces} in \autoref{sec:reducible}, assuming \autoref{thm:main}.
Then, in \autoref{sec:non_anosov}, we show that only the Hitchin and Barbot components can contain Anosov representations.

The remainder of the paper is devoted to proving that, if $\rho$ is in the Barbot component and $\rho(s_1s_2s_3)$ has distinct real eigenvalues, then $\rho$ is Anosov.
We do this by approximating its boundary map with a collection of ``boxes'' in $\RP^2$, similarly to the approach in \cite{Schwartz}.
First, we show in \autoref{sec:limit_curve} that if a subset of $\RP^2$ is mapped into itself by certain elements of $\Gamma$, iterating this makes the resulting nested sets converge to a continuous boundary map.
We then construct boxes which have this property in \autoref{sec:nested_boxes}, using the eigenvectors of $\rho(s_1s_2s_3)$.
This yields a continuous boundary map into $\RP^2$.
Note that the arguments in \autoref{sec:limit_curve} are more general than those in other sections.
That is, they might also apply to other cocompact discrete subgroups of the isometry group of the hyperbolic plane, not only to hyperbolic triangle reflection groups.
In \autoref{sec:duality} we use duality to extend the boundary map to a map into the flag manifold.
Finally, we show in \autoref{sec:transversality} that the resulting map is transverse, and therefore $\rho$ is Anosov.
Combining everything, we prove \autoref{thm:main} at the end of \autoref{sec:transversality}.

\subsection{Acknowledgements}

We are thankful for helpful conversations with Jean--Philippe Burelle, Jeffrey Danciger, and Richard Evan Schwartz.
We also thank the referees for many valuable comments which improved the paper significantly.

G.-S.~Lee was supported by the European Research Council under ERC--Consolidator Grant 614733 and by the National Research Foundation of Korea (NRF) grant funded by the Korean government (MSIT) (No 2020R1C1C1A01013667).
J.~Lee is supported by the grants NRF--2019R1F1A1047703 and NRF--2020R1C1C1A01013667.
F.~Stecker received support from the Klaus Tschira Foundation, the RTG 2229 grant of the German Research Foundation and ERC grants 614733 and 715982.

\section{Triangle reflection groups in \texorpdfstring{$\SL(3,\bR)$}{SL(3,R)}}\label{sec:prelims}

\subsection{Parametrizing Coxeter representations}\label{sec:Coxeter_representations}

Our goal is to find all Anosov representations of the compact hyperbolic triangle reflection group
\[\Gamma = \Gamma_{p_1,p_2,p_3} = \langle s_1,s_2,s_3 \mid s_1^2 = s_2^2 = s_3^2 = (s_2s_3)^{p_1} = (s_3s_1)^{p_2} = (s_1s_2)^{p_3} = 1 \rangle\]
into $\SL(3,\bR)$, where $2 \leq p_1,p_2,p_3 < \infty$ and $\frac{1}{p_1} + \frac{1}{p_2} + \frac{1}{p_3} < 1$.
We call the product $s_1s_2s_3$ the \emph{Coxeter element}.

In this section, we shall parametrize those representations $\rho \colon \Gamma \to \SL(3,\bR)$ which send the generators $s_1,s_2,s_3$ to pairwise distinct non--trivial involutions in $\SL(3,\bR)$.
We call such a representation a \emph{Coxeter representation}, and denote the space of them by $\Hom^{\mathrm{Cox}}(\Gamma,\SL(3,\bR))$.
It is easy to see that there are only finitely many conjugacy classes of non--Coxeter representations, and that their images are always trivial, $\bZ/2\bZ$, or a finite dihedral group.
$\Hom^{\mathrm{Cox}}(\Gamma,\SL(3,\bR))$ is embedded into $\SL(3,\bR)^3$ as the images of the generators $s_1,s_2,s_3$ and inherits its topology from this embedding.

The \emph{space of Coxeter characters} is
\[\chi^{\mathrm{Cox}}(\Gamma, \SL(3,\bR)) = \Haus(\Hom^{\mathrm{Cox}}(\Gamma, \SL(3,\bR))/\SL(3,\bR))\]
where $\SL(3,\bR)$ acts on the space of representations by conjugation, and $\Haus(X)$ is the \emph{Hausdorff quotient} or \emph{Hausdorffification} of a topological space $X$. That is the quotient of $X$ by the equivalence relation $\sim$ defined by
\[x \sim y \Longleftrightarrow \text{$x \approx y$ for every equivalence relation $\approx$ such that $X / \mathord{\approx}$ is Hausdorff.}\]
In particular, two points $x,y \in X$ with intersecting closures $\overline{\{x\}}$ and $\overline{\{y\}}$ represent the same element of $\Haus(X)$.
So two representations have the same character if their conjugacy classes have intersecting closures. The space of semisimple representations modulo conjugation is Hausdorff and it may be identified with the Hausdorff quotient of the representation space; see \cite{Luna1,Luna2,RichardsonSlodowy}.

An involution $\sigma$ in $\SL(3,\bR)$ can be written as
\[ \sigma = b \otimes \alpha - 1,  \quad \textrm{i.e.} \quad \sigma(v) = \alpha(v)b - v,\quad \forall v\in \bR^3,\]
where $\alpha$ is a linear functional and $b$ is a vector of $\bR^3$ such that $\alpha(b)=2$. It uniquely determines the pair $(\alpha, b)$ up to the action of $\bR^*$ by $\lambda \cdot (\alpha, b) = (\lambda\alpha,\lambda^{-1}b)$.

\begin{Lem}[{\cite[Chapter III]{GoldmanThesis}}]\label{lem:two_involutions}
Let $\sigma_1 = b_1 \otimes \alpha_1 - 1$, $\sigma_2 = b_2 \otimes \alpha_2 - 1$ be two distinct involutions in $\SL(3,\bR)$ and $p \geq 2$ an integer. Then $(\sigma_1\sigma_2)^p = 1$ if and only if
\begin{itemize}
\item $\alpha_1(b_2)\,\alpha_2(b_1) = 4 \cos^2\big(\tfrac{q}{p}\pi\big)$, where $1 \leq q \leq \tfrac{p}{2}$, and
\item $\alpha_1(b_2)$ and $\alpha_2(b_1)$ are either both zero or both non--zero.
\end{itemize}
\end{Lem}

\begin{Prf}
  The subspace $\ker \alpha_1 \cap \ker \alpha_2$ of $\bR^3$ is at least 1--dimensional, so $1$ is an eigenvalue of $\sigma_1\sigma_2$.
  A computation shows that
  \begin{equation}\label{eq:computation}
    \sigma_i\sigma_j = \alpha_i(b_j)\, b_i \otimes \alpha_j - b_i \otimes \alpha_i - b_j \otimes \alpha_j + 1 \quad \textrm{ for } \{ i, j\} = \{ 1, 2 \},
  \end{equation}
  \[\tr \sigma_1\sigma_2 = \alpha_1(b_2) \alpha_2(b_1) - 1.\]
  Now we assume $(\sigma_1\sigma_2)^p = 1$.
  Then $\sigma_1\sigma_2$ must be complex diagonalizable with eigenvalues $1,e^{2\pi i q/p},e^{-2\pi i q/p}$, where $1 \leq q < p$.
  Possibly replacing $q$ by $p-q$, we can assume $q \leq \frac{p}{2}$.
  So $\alpha_1(b_2)\alpha_2(b_1) = \tr\sigma_1\sigma_2 + 1 = 4 \cos^2(\frac{q}{p}\pi)$.
  If $q = \frac{p}{2}$, we also have $\sigma_1\sigma_2 = \sigma_2\sigma_1$.
  Since $b_1 \otimes \alpha_2$ and $b_2 \otimes \alpha_1$ are linearly independent in the space of $3 \times 3$ matrices, this implies $\alpha_1(b_2) = \alpha_2(b_1) = 0$.

  Conversely, if $\alpha_1(b_2)\alpha_2(b_1) = 4\cos^2(\frac{q}{p}\pi)$ with $1 \leq q < \frac{p}{2}$, then $\sigma_1\sigma_2$ has eigenvalues $1,e^{2\pi i q/p},e^{-2\pi i q/p}$, so $(\sigma_1\sigma_2)^p = 1$.
  If $\alpha_1(b_2) = \alpha_2(b_1) = 0$, then $p$ is even and $\sigma_1\sigma_2 = \sigma_2\sigma_1$ by \eqref{eq:computation}, so $(\sigma_1\sigma_2)^p = 1$.
\end{Prf}

This motivates the following definition. A real matrix $A = (a_{ij})_{1 \leq i,j \leq 3}$ is called a \emph{Cartan matrix} if
\begin{enumerate}
\item $a_{ii}=2$ for all $i = 1,2,3$,
\item $a_{ij}a_{ji} = 4 \cos^2\big(\tfrac{q_k}{p_k}\pi\big)$ for integers $1 \leq q_k \leq \tfrac{p_k}{2}$ and $\{i,j,k\} = \{1,2,3\}$,
\item if $a_{ij} = 0$ then $a_{ji} = 0$ for all $i,j = 1,2,3$.
\end{enumerate}

Two Cartan matrices are \emph{equivalent} if they are conjugated by a diagonal matrix.
We denote by $\mathscr C$ the space of Cartan matrices modulo equivalence.
It parametrizes the Coxeter characters as follows.

\begin{Prop}\label{lem:chi_coxeter}
  The map
  \[\Psi \colon \chi^{\mathrm{Cox}}(\Gamma,\SL(3,\bR)) \to \mathscr C, \quad [\rho] \mapsto (\alpha_i(b_j))_{1 \leq i,j \leq 3},\]
  where $\rho(s_i) = b_i \otimes \alpha_i - 1$ and $\alpha_i(b_i) = 2$ for all $i = 1,2,3$, is a homeomorphism.
\end{Prop}

\begin{Prf}
  First note that since $b_i \in \bR^3$ and $\alpha_i \in (\bR^3)^*$ are only determined up to the action of $\bR^*$, this gives us the matrix $(\alpha_i(b_j))_{1 \leq i,j \leq 3}$ up to equivalence.
  It is a Cartan matrix by \autoref{lem:two_involutions}.
  So $\Psi$ is well--defined as a map from $\Hom^{\mathrm{Cox}}(\Gamma,\SL(3,\bR))$, and it is continuous.
  Every continuous map from a topological space $X$ to a Hausdorff space $Y$ induces a unique continuous map from $\Haus(X)$ to $Y$.
  So since $\mathscr C$ is Hausdorff and $\Psi$ is conjugation invariant, it descends to a map from $\chi^{\mathrm{Cox}}(\Gamma,\SL(3,\bR))$.

  We construct a continuous map $\Psi'$, which will be the inverse of $\Psi$, as follows:
  For a Cartan matrix $C$ we set $\rho_C(s_i) = e_i \otimes \gamma_i - 1$ where $\{e_i\}_{i=1}^3$ is the standard basis of $\bR^3$, and $\gamma_i \in (\bR^3)^*$ is the $i$--th row of $C$.
  \autoref{lem:two_involutions} ensures that this indeed defines a representation of $\Gamma$.
  If $\Lambda$ is a diagonal matrix with entries $\lambda_1,\lambda_2,\lambda_3$ then the $i$--th row of $\Lambda C \Lambda^{-1}$ is $\lambda_i \gamma_i \Lambda^{-1}$, so $\rho_{\Lambda C \Lambda^{-1}}(s_i) = \Lambda \rho_C(s_i) \Lambda^{-1}$.
  Hence we can define $\Psi'$ by setting $\Psi'([C]) = [\rho_C]$.
  It is easy to see that $\Psi \circ \Psi'$ is the identity map.

  To see that $\Psi' \circ \Psi$ is also the identity, let $\rho$ be a representation defining $b_i$ and $\alpha_i$ as before, and let $A$ be the matrix with $i$--th row $\alpha_i$ and $B$ the matrix with $i$--th column $b_i$, for all $i$. Then the corresponding Cartan matrix is $C = AB$.
  We want to show that $[\rho_C] = [\rho]$.
  The matrix $B$ need not be invertible, but we can write $B = \widetilde B P$ for an invertible matrix $\widetilde B$ and a projection $P$ (that is $P^2 = P$).
  If we set $P_n = P + \frac{1}{n}(1 - P)$ then $P_nP = PP_n = P$, so $CP_n^{-1} = A\widetilde B P P_n^{-1} = C$ and $P_n^{-1}\widetilde B^{-1}B = P_n^{-1}P = P$.
  Further $A\widetilde B P_n$ converges to $C$, so we have
\begin{align*}
    P_n \rho_C(s_i) P_n^{-1} &= (P_ne_i) \otimes (\gamma_i P_n^{-1}) - 1 & \xrightarrow{n\to\infty} \quad (Pe_i) \otimes \gamma_i - 1, \\
    P_n^{-1}\widetilde B^{-1}\rho(s_i)\widetilde B P_n &= (P_n^{-1}\widetilde B^{-1} b_i) \otimes (\alpha_i \widetilde B P_n) - 1 & \xrightarrow{n\to\infty} \quad (Pe_i) \otimes \gamma_i - 1.
\end{align*}
  So the conjugacy classes of $\rho$ and $\rho_C$ have intersecting closures, which implies that they represent the same character in $\chi^{\mathrm{Cox}}(\Gamma,\SL(3,\bR))$.
\end{Prf}

\begin{Def}\label{def:irreducible}
  We call a representation $\rho \colon \Gamma \to \SL(3,\bR)$
  \begin{itemize}
  \item \emph{point--irreducible} if it does not preserve any one--dimensional subspace of $\bR^3$,
  \item \emph{line--irreducible} if it does not preserve any two--dimensional subspace of $\bR^3$,
  \item \emph{semisimple} if it is a product of irreducible representations.
  \end{itemize}  
\end{Def}

Every Coxeter character $[\rho]$ has a point--irreducible, a line--irreducible, and a semisimple representative.
Note that if $\rho$ is irreducible, then it is point--irreducible, line--irreducible, and semisimple at once.
An example for a line--irreducible one is $\rho_C$ as constructed in the proof of \autoref{lem:chi_coxeter}, and a point--irreducible representation can be obtained by a dual construction.

\subsection{The space of Cartan matrices}\label{sec:Cartan_matrices}

\autoref{lem:chi_coxeter} showed that the Coxeter characters are parametrized by Cartan matrices.
Luckily, the space of Cartan matrices $\mathscr C$ is quite simple.
It consists of a number of connected components homeomorphic to $\bR$ and possibly some isolated points.

\begin{Def}\label{def:cartan_matrix}
  Let $A = (a_{ij})_{i,j}$ be the Cartan matrix corresponding to a Coxeter representation $\rho$.
  We say $A$ \emph{is of type $(q_1,q_2,q_3)$} if $1 \leq q_k \leq \tfrac{p_k}{2}$ and
\[a_{ij}a_{ji} = c_k^2, \qquad c_k \coloneqq 2\cos\big(\tfrac{q_k}{p_k}\pi\big)\]
for all $\{i,j,k\} = \{1,2,3\}$.
These \emph{2--cyclic products} $a_{ij}a_{ji}$, as well as \emph{3--cyclic products} $a_{ij}a_{jk}a_{ki}$, are well--defined by the equivalence class $[A]$ of $A$.
We also say a Coxeter representation $\rho \colon \Gamma_{p_1,p_2,p_3} \to \SL(3,\bR)$ \emph{is of type $(q_1,q_2,q_3)$} if its Cartan matrix is.
\end{Def}

Since $a_{ij}a_{ji}$ can take only a discrete set of values, the space $\mathscr C_{q_1,q_2,q_3} \subset \mathscr C$ of Cartan matrices of type $(q_1,q_2,q_3)$ is a union of connected components.

\begin{Lem}\label{lem:homeoToR}
  If $q_k = \tfrac{p_k}{2}$ for some $k \in \{1,2,3\}$ then $\mathscr{C}_{q_1,q_2,q_3}$ is a single point.
  Otherwise it has two connected components, each homeomorphic to $\bR$.
\end{Lem}

\begin{Prf}
  If $q_k = \tfrac{p_k}{2}$ for some $k$, then $a_{ij} = a_{ji} = 0$ for $\{i,j,k\} = \{1,2,3\}$. The Cartan matrix $A$ is therefore equivalent to a symmetric matrix which is determined by $(q_1,q_2,q_3)$ alone.
  For example, if $q_1 = \frac{p_1}{2}$ then $A$ is
\[\begin{pmatrix}2 & a_{12} & a_{13} \\ a_{21} & 2 & 0 \\ a_{31} & 0 & 2\end{pmatrix} \sim \begin{pmatrix} 2 & c_3 & c_2 \\ c_3 & 2 & 0 \\ c_2 & 0 & 2 \end{pmatrix}.\]
So $\mathscr C_{q_1,q_2,q_3}$ is just a single point in this case.

If $q_k < \tfrac{p_k}{2}$ for all $k \in \{1,2,3\}$ then $A$ is equivalent to a matrix of the form
\[\begin{pmatrix}2 & -c_3 & -c_2 \\ -c_3 & 2 & -tc_1 \\ -c_2 & -t^{-1}c_1 & 2\end{pmatrix}\]
where $t \in \bR \setminus \{0\}$ (the minus signs are just a convention).
This representative is unique since $t = - a_{12}a_{23}a_{31} / c_1c_2c_3$, which only depends on the equivalence class. So $\mathscr C_{q_1,q_2,q_3} \cong \bR \setminus \{0\}$.
\end{Prf}

Let $\rho \colon \Gamma_{p_1,p_2,p_3} \to \SL(3,\bR)$ be a Coxeter representation of type $(q_1,q_2,q_3)$ with $q_k < \frac{p_k}{2}$ for all $k \in \{1,2,3\}$.
Let $(a_{ij})_{i,j}$ be the Cartan matrix of $\rho$.
We define its \emph{parameter} by
\begin{equation}
  t_\rho = -\frac{a_{12}\,a_{23}\,a_{31}}{c_1c_2c_3}.\label{eq:trho_definition}
\end{equation}
It can take any non--zero real value and parametrizes the Coxeter characters of type $(q_1,q_2,q_3)$.

We can express the 2--cyclic products and 3--cyclic products by traces:
\begin{gather}
  \tr \rho(s_1s_2) = a_{12}a_{21} - 1, \qquad \tr \rho(s_2s_3) = a_{23}a_{32} - 1, \qquad \tr \rho(s_3s_1) = a_{31}a_{13} - 1, \label{eq:trrot} \\
  \tr \rho(s_1s_2s_3) = a_{12}a_{23}a_{31} - a_{12}a_{21} - a_{23}a_{32} - a_{31}a_{13} + 3, \label{eq:trcox_general} \\
  \tr \rho(s_3s_2s_1) = a_{21}a_{32}a_{13} - a_{12}a_{21} - a_{23}a_{32} - a_{31}a_{13} + 3 . \label{eq:trcox_inverse_general}
\end{gather}

We immediately see from this that the determinant of the Cartan matrix is
\begin{equation}
  \label{eq:determinant_trace}
  \det\begin{pmatrix}2 & a_{12} & a_{13} \\ a_{21} & 2 & a_{23} \\ a_{31} & a_{32} & 2 \end{pmatrix} = \tr\rho(s_1s_2s_3) + \tr\rho(s_3s_2s_1) + 2.
\end{equation}

\begin{Lem}\label{lem:reducible_cartan_matrix}
  Let $\rho \colon \Gamma \to \SL(3,\bR)$ be a Coxeter representation of type $(q_1,q_2,q_3)$ with $q_i < \frac{p_i}{2}$ for all $i$.
  Then $\rho$ is reducible if and only if the determinant of its Cartan matrix is zero.
\end{Lem}

\begin{Prf}
  Writing $\rho(s_i) = b_i \otimes \alpha_i - 1$ for all $i$ as above, the Cartan matrix having zero determinant means that either $(b_1,b_2,b_3)$ or $(\alpha_1,\alpha_2,\alpha_3)$ are linearly dependent.
  Then the span of the $b_i$ or the intersection of the kernels of the $\alpha_i$ are a proper invariant subspace, so $\rho$ is reducible.

  Conversely, let $\rho$ be reducible, so it preserves a proper subspace $W \subset \bR^3$.
  Then, for all $i$, either $b_i \in W$ or $\alpha_i|_W = 0$.
  By the assumption $q_i < \frac{p_i}{2}$ and \autoref{lem:two_involutions} we have $\alpha_i(b_j) \neq 0$ for all $i$ and $j$, so either $b_i \in W$ for all $i$ or $\alpha_i|_W = 0$ for all $i$.
  In the first case the $b_i$ are linearly dependent and in the second case the $\alpha_i$ are linearly dependent, so in either case the determinant of the Cartan matrix is zero.
\end{Prf}

\begin{Lem}\label{lem:coxeter_eigenvalues}
  Let $\rho \colon \Gamma \to \SL(3,\bR)$ be a Coxeter representation. Assume that $p_1,p_2,p_3$ are odd and $\rho$ is of type $\big(\tfrac{p_1-1}{2},\tfrac{p_2-1}{2},\tfrac{p_3-1}{2}\big)$. Then there is a real number $\tcrit > 1$ such that $\rho(s_1s_2s_3)$
    \begin{itemize}
    \item has two non--real eigenvalues if $t_\rho < 0$ or $t_\rho \in (\tcrit^{-1},\tcrit)$,
    \item is not diagonalizable and has a negative eigenvalue $\lambda$ of algebraic multiplicity 2 if $t_\rho \in \{\tcrit^{-1},\tcrit\}$ (with $\lambda < -1$ if $t_\rho = \tcrit$ and $\lambda > -1$ if $t_\rho = \tcrit^{-1}$),
    \item has real eigenvalues with distinct absolute values if $t_\rho \in (0,\tcrit^{-1}) \cup (\tcrit,\infty)$.
    \end{itemize}
\end{Lem}

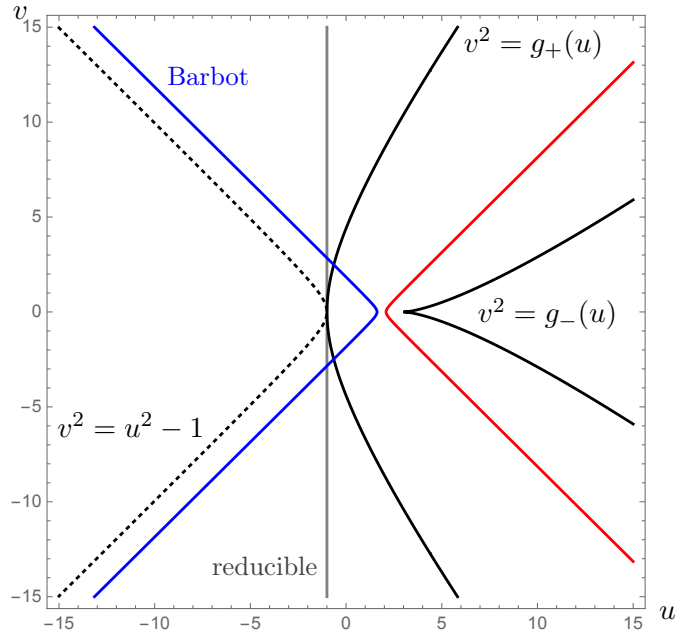
\begin{figure}[ht!]
\labellist
\pinlabel {$u$} at 465 17
\pinlabel {$v$} at 7 445

\pinlabel {$v^2 = g_{+}(u)$} at 370 423
\pinlabel {$v^2 = g_{-}(u)$} at 380 230
\pinlabel {\textcolor{darkgray}{\small reducible}} at 180 50

\pinlabel {\textcolor{blue}{\small Barbot}} at 140 400
\pinlabel {$v^2 = u^2 - 1$} at 85 150

\endlabellist
\centering
\includegraphics[width=0.57\textwidth]{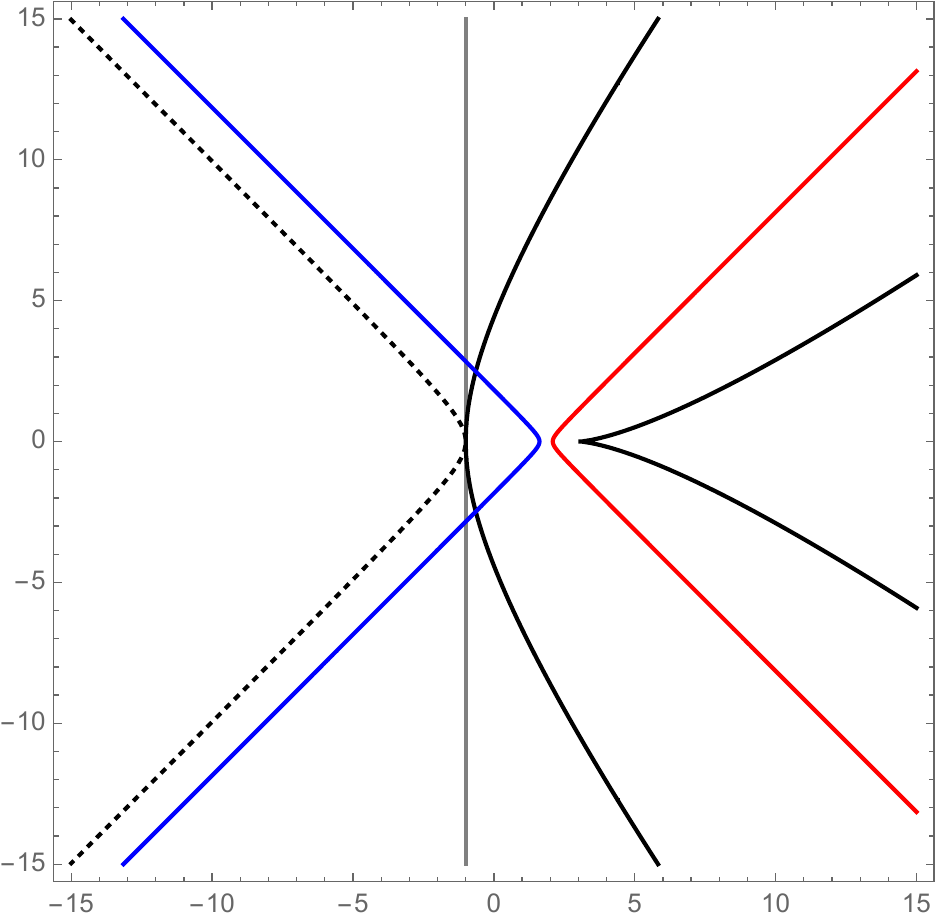}
\caption{The possible values of $u = \frac{\tr\rho(s_1s_2s_3) + \tr\rho(s_3s_2s_1)}{2}$ and $v = \frac{\tr\rho(s_1s_2s_3) - \tr\rho(s_3s_2s_1)}{2}$ for representations of type $(\frac{p_1-1}{2},\frac{p_2-1}{2},\frac{p_3-1}{2})$ in blue or in red. The curves are drawn in the case of $(p_1,p_2,p_3) = (5,5,5)$. The blue curve represents the Barbot component to be defined in \autoref{def:hitchin_and_barbot}.}
\label{fig:goldman}
\end{figure}

\begin{proof}
By the proof of \autoref{lem:homeoToR}, the Cartan matrix of $\rho$ is uniquely equivalent to
\begin{equation*}
A = \begin{pmatrix}
  2    & -c_3              & -c_2 \\
  -c_3 & 2                 & -t_\rho c_1 \\
  -c_2 & -t_\rho^{-1} c_1  & 2
\end{pmatrix},
\end{equation*}
where $t_\rho \neq 0$ and $c_i = 2\cos\big(\tfrac{p_i-1}{2 p_i}\pi\big)$. Then \eqref{eq:trcox_general} and \eqref{eq:trcox_inverse_general} above show that:
\begin{align}
x & = \tr \rho(s_1s_2s_3) = -t_\rho c_1c_2c_3 - c_1^2 - c_2^2 - c_3^2 + 3 \label{eq:trcox} \\
y & = \tr \rho(s_3s_2s_1) = -t_\rho^{-1} c_1c_2c_3 - c_1^2 - c_2^2 - c_3^2 + 3 \label{eq:trcoxinverse}
\end{align}
Thus, the variables $x$ and $y$ satisfy
\[ p(x,y) := (x - 3 + c_1^2 + c_2^2 + c_3^2 )(y - 3 + c_1^2 + c_2^2 + c_3^2 ) = c_1^2 c_2^2 c_3^2. \]
From the discriminant of the characteristic polynomial (see \cite[Section 1.7]{Goldman}) we find that $\rho(s_1s_2s_3)$ has real eigenvalues if and only if $\delta(x,y) \geq 0$ (which are distinct if and only if $\delta(x,y) > 0$), where
\begin{equation}
  \delta(x,y) := x^2 y^2 - 4 (x^3+y^3) + 18 xy - 27.\label{eq:discriminant}
\end{equation}

Using a change of variables ($u = \frac{x + y}{2}$, $v = \frac{x - y}{2}$), we obtain that:
\begin{align*}
p(x,y) = c_1^2 c_2^2 c_3^2 \quad & \Leftrightarrow \quad v^2 = (u-3 + c_1^2 + c_2^2 + c_3^2)^2 - (c_1c_2c_3)^2 \eqqcolon f(u) \\
\delta(x,y)  = 0 \quad & \Leftrightarrow \quad v^2 = u^2 + 12 u + 9 \pm 2(2u+3)^{\frac{3}{2}} \eqqcolon g_{\pm}(u)
\end{align*}
In fact, $\delta(x,y) < 0$ if and only if $g_-(u) < v^2 < g_+(u)$.
We set $u_{\pm} := 3 - c_1^2 - c_2^2 - c_3^2 \pm c_1c_2c_3$, which are the two solutions of the equation $f(u)=0$.
The points $(u,v) = (u_\pm,0)$ correspond to the Coxeter characters with $t_\rho = \mp 1$.
Since $p_i \geq 3$ we have $0 < c_i \leq 1$, and $c_i \leq 2\cos(\frac{2\pi}{5})$ for at least one $i \in \{1,2,3\}$.
Then a computation shows that $0 \leq u_- < u_+ < 3$.

Note that $g_+(u) \geq 0 \Leftrightarrow u \geq -1$ and $g_-(u) \geq 0 \Leftrightarrow u \geq 3$.
We claim that $f(u) \neq g_-(u)$ for all $u \geq 3$, $f(u) \neq g_+(u)$ for all $u \geq u_+$, and that there is exactly one $u \in (-1, u_-)$ with $f(u) = g_+(u)$, which we call $u_{\mathrm{crit}}$.
Note that $f(u_{\mathrm{crit}}) > 0$.
This will show that the red curve in \autoref{fig:goldman} (the component of $v^2 = f(u)$ containing $u_+$) does not intersect the black curve $v^2 = g_\pm(u)$ and the blue curve (containing $u_-$) intersects the black curve in exactly two points $(u_{\mathrm{crit}}, \pm v_{\mathrm{crit}})$, where $v_{\mathrm{crit}} = \sqrt{f(u_{\mathrm{crit}})}$.

We prove the claim by computing derivatives.
Since
\[ \frac{\dd f}{\dd u} = 2(u - 3 + c_1^2 + c_2^2 + c_3^2) \quad \textrm{and} \quad
\frac{\dd g_{\pm}}{\dd u} = 2u + 12 \pm 6 \sqrt{2u+3},\]
we have
\begin{align*}
  \frac{\dd f}{\dd u} - \frac{\dd g_-}{\dd u} = 6\sqrt{2u+3} - 18 + 2(c_1^2 + c_2^2 + c_3^2) \geq 0 &\qquad \text{for all $u \geq 3$}, \\
  \frac{\dd g_+}{\dd u} - \frac{\dd f}{\dd u} = 6\sqrt{2u+3} + 18 - 2(c_1^2 + c_2^2 + c_3^2) \geq 0 &\qquad \text{for all $u \geq -1$.}
\end{align*}
This proves the claim, combined with the initial values
\[(f - g_-)(3) = f(3) > 0, \quad (g_+ - f)(-1) = -f(-1) < 0, \quad (g_+ - f)(u_-) = g_+(u_-) > 0.\]
As a result, for all $t_\rho < 0$ the discriminant has the same sign as for $t_\rho = -1$, i.e. $\rho(s_1 s_2 s_3)$ has non--real eigenvalues.
For $t_\rho > 0$, the discriminant changes sign exactly at the two points $(u_{\mathrm{crit}}, \pm\sqrt{f(u_{\mathrm{crit}})})$, corresponding to $t_\rho \in \{\tcrit^{-1}, \tcrit\}$.

In the case $t_\rho \in (0,\tcrit^{-1}) \cup (\tcrit,\infty)$ if the eigenvalues of $\rho(s_1s_2s_3)$ were of the form $\lambda,-\lambda,-\lambda^{-2}$, this would imply $v^2 = u^2 - 1$ with $u \leq -1$.
But since $f(u) - u^2 + 1$ is linear, decreasing, and has a positive value at $u = -1$, the curves $v^2 = u^2 - 1$ and $v^2 = f(u)$ don't intersect at $u \leq -1$.

Finally, in the case $t_\rho \in \{\tcrit^{-1},\tcrit\}$, $\rho(s_1s_2s_3)$ is not diagonalizable:
assume it were, then it fixes a projective line pointwise.
The intersection point $p$ of this line with the reflection line of $\rho(s_1)$ is fixed by $\rho(s_2s_3)$.
This element has a unique fixed point, which is fixed by $\rho(s_2)$ and $\rho(s_3)$ individually.
So all of $\rho(\Gamma)$ fixes $p$, hence $\rho$ is reducible.
By \autoref{lem:reducible_cartan_matrix} and \eqref{eq:determinant_trace} this implies $u_{\mathrm{crit}} = -1$, a contradiction.

If $t_\rho = \tcrit$ the eigenvalues of $\rho(s_1s_2s_3)$ are of the form $\lambda,\lambda,\lambda^{-2}$ with $\lambda < 0$ ($\lambda \geq 0$ would imply $u \geq 3$ but we showed $u_{\mathrm{crit}} < u_- <3$ above).
Then $(\lambda - \lambda^{-1})(2 - \lambda - \lambda^{-1}) = x - y = (-t_\rho + t_\rho^{-1})c_1c_2c_3 < 0$, hence $\lambda < -1$.
Analogously, we get $\lambda > -1$ if $t_\rho = \tcrit^{-1}$.
\end{proof}

\begin{Rem}\label{rem:critical_moment}
  As $t_\rho$ approaches $\tcrit$ from above, the attracting and neutral fixed points of $\rho(s_1s_2s_3)$ in $\RP^2$ merge, and so do its repelling and neutral fixed lines.
  If $t_\rho$ approaches $\tcrit^{-1}$ from below, it's instead the repelling and neutal fixed points, as well as the attracting and neutral fixed lines, which merge.
\end{Rem}

\begin{Rem}
  The proof of \autoref{lem:coxeter_eigenvalues} shows that if $q_k < \frac{p_k}{2}$ for all $k \in \{1,2,3\}$, then the space $\mathscr{C}_{q_1,q_2,q_3}$ of Coxeter characters of type $(q_1,q_2,q_3)$ may be identified with an algebraic curve defined by the polynomial equation $p(x,y) = c_1^2 c_2^2 c_3^2$ in the plane $\mathbb{R}^2$, where $x$ and $y$ are the traces of $\rho(s_1 s_2 s_3)$ and $\rho(s_3 s_2 s_1)$ respectively.
\end{Rem}

\subsection{From \texorpdfstring{$\PGL(2,\bR)$}{PGL(2,R)} to \texorpdfstring{$\SL(3,\bR)$}{SL(3,R)}}\label{sec:reducible}

The Hitchin and Barbot components in $\chi^{\mathrm{Cox}}(\Gamma,\SL(3,\bR))$ are distinguished by the fact that they contain certain representations factoring through $\PGL(2,\bR)$ or $\SL^\pm(2,\bR)$.
We will describe these now.
We start with a discrete and faithful representation $\rho_0\colon\Gamma\to\PGL(2,\bR)$.
It is unique up to conjugation since there is a unique hyperbolic triangle with fixed angles, up to isometry.

Let $\iota \colon \PGL(2,\bR) \to \PGL(3,\bR) \cong \SL(3,\bR)$ be the irreducible embedding, which is unique up to conjugation.
Concretely, it can be realized by the action of $\PGL(2,\bR)$ on the projectivization of the symmetric square $\Sym^2\bR^2 \cong \bR^3$.
The composition
\[\rho_{F} = \iota \circ \rho_0 \colon \Gamma \to \SL(3,\bR)\]
is called a Fuchsian representation.
The Hitchin component will be the component of $\chi^{\mathrm{Cox}}(\Gamma,\SL(3,\bR))$ containing $\rho_F$.
We identify it by the following lemma, a proof of which can be found e.g. in \cite[Proposition 24]{Vinberg}.

\begin{Lem}\label{lem:irreducible_identify}
  The representation $\rho_{F}$ is of type $(1,1,1)$ and has the parameter $t_{\rho_{F}} = 1$.  
\end{Lem}

A second way to create special $\SL(3,\bR)$ representations out of $\rho_0$ is by using the embedding
\[\jmath \colon \SL^\pm(2,\bR) \to \SL(3,\bR), \quad A \mapsto \begin{pmatrix}A & 0 \\ 0 & \det(A)\end{pmatrix}.\]
Here $\SL^\pm(2,\bR)$ is the group of $2 \times 2$ matrices with determinant $\pm 1$.
This requires lifting $\rho_0$ to $\SL^\pm(2,\bR)$, which is possible if and only if $p_1,p_2,p_3$ are all odd.

To see this, we first note that each $\rho_0(s_i)$ is a hyperbolic involution acting on $\RP^1$ with two distinct fixed points. In order to specify a lift of $\rho_0(s_i)$ in $\SL^\pm(2,\bR)$ we put an arbitrary order on these fixed points and regard them as representing an oriented geodesic $s_i^-s_i^+$ in $\bH^2$. The lift $\widetilde{\rho_0}(s_i)$ corresponding to $s_i^-s_i^+$ is defined as the reflection having $s_i^+$ as the $(+1)$-eigenspace and $s_i^-$ as the $(-1)$-eigenspace.

\begin{center}
  \begin{tikzpicture}[scale=1.6]
    \begin{scope}
      \begin{scope}[shift={(210:0.53)},rotate=0]
        \draw[fill=black!20] (226:0.15) arc (227:370:0.15) -- (0,0) -- cycle;
      \end{scope}
      \begin{scope}[shift={(330:0.53)},rotate=120]
        \draw[fill=black!20] (226:0.15) arc (227:370:0.15) -- (0,0) -- cycle;
      \end{scope}
      \begin{scope}[shift={(90:0.53)},rotate=240]
        \draw[fill=black!20] (226:0.15) arc (227:370:0.15) -- (0,0) -- cycle;
      \end{scope}

      \draw[very thick] (0,0) circle (1);
      \draw[->,thick,red] (84.45:1) arc (-5.55:-29.975:2.2) coordinate (M1);
      \draw[->,thick,red] (84.45+120:1) arc (-5.55+120:120-29.975:2.2) coordinate (M2);
      \draw[->,thick,red] (84.45+240:1) arc (-5.55+240:240-29.975:2.2) coordinate (M3);

      \draw[thick,red] (M1) arc (-29.975:-54.4:2.2);
      \draw[thick,red] (M2) arc (120-29.975:120-54.4:2.2);
      \draw[thick,red] (M3) arc (240-29.975:240-54.4:2.2);

      \node at (80:1.2) {$s_1^-$};
      \node at (100:1.2) {$s_3^+$};
      \node at (120+80:1.2) {$s_2^-$};
      \node at (120+100:1.2) {$s_1^+$};
      \node at (240+80:1.2) {$s_3^-$};
      \node at (240+100:1.2) {$s_2^+$};

      \node at (135:0.65) {$\scriptstyle\pi - \frac{\pi}{p_2}$};
      \node at (120+130:0.65) {$\scriptstyle\pi - \frac{\pi}{p_3}$};
      \node at (240+125:0.65) {$\scriptstyle\pi - \frac{\pi}{p_1}$};
    \end{scope}

    \begin{scope}[xshift=3cm]
      \begin{scope}[shift={(210:0.53)},rotate=180]
        \draw[fill=black!20] (229:0.15) arc (229:373:0.15) -- (0,0) -- cycle;
      \end{scope}
      \begin{scope}[shift={(330:0.53)},rotate=300]
        \draw[fill=black!20] (229:0.15) arc (229:373:0.15) -- (0,0) -- cycle;
      \end{scope}
      \begin{scope}[shift={(90:0.53)},rotate=60]
        \draw[fill=black!20] (229:0.15) arc (229:373:0.15) -- (0,0) -- cycle;
      \end{scope}

      \draw[very thick] (0,0) circle (1);
      \draw[->,thick,red] (120+95.55:1) arc (   -54.4:   -29.975:2.2) coordinate (M1);
      \draw[->,thick,red] (240+95.55:1) arc (120-54.4:120-29.975:2.2) coordinate (M2);
      \draw[->,thick,red] (    95.55:1) arc (240-54.4:240-29.975:2.2) coordinate (M3);

      \draw[thick,red] (M1) arc (   -29.975:   -5.55:2.2);
      \draw[thick,red] (M2) arc (120-29.975:120-5.55:2.2);
      \draw[thick,red] (M3) arc (240-29.975:240-5.55:2.2);

      \node at (80:1.2) {$s_1^+$};
      \node at (100:1.2) {$s_3^-$};
      \node at (120+80:1.2) {$s_2^+$};
      \node at (120+100:1.2) {$s_1^-$};
      \node at (240+80:1.2) {$s_3^+$};
      \node at (240+100:1.2) {$s_2^-$};

      \node at (50:0.65) {$\scriptstyle\pi - \frac{\pi}{p_2}$};
      \node at (120+50:0.65) {$\scriptstyle\pi - \frac{\pi}{p_3}$};
      \node at (240+50:0.65) {$\scriptstyle\pi - \frac{\pi}{p_1}$};
    \end{scope}
  \end{tikzpicture}
\end{center}

If $0<\theta<\pi$ is the angle between the two intersecting oriented geodesics $s_1^-s_1^+$ and $s_2^-s_2^+$ then $\theta=\frac{\pi}{p_3}$ or $\theta=\pi-\frac{\pi}{p_3}$ depending on the chosen orientations.
The product $\widetilde{\rho_0}(s_1s_2)$ is conjugate in $\SL^\pm(2,\bR)$ to the rotation matrix $R(\theta)=\begin{psmallmatrix}\cos\theta&-\sin\theta\\\sin\theta&\cos\theta\end{psmallmatrix}$.
In order for $\rho_0$ to lift it is necessary that $R(\theta)^{p_3} = \id$.
But $R(\pi/p_3)^{p_3} = -\id$ and $R(\pi - \pi/p_3)^{p_3} = (-1)^{p_3}(-\id)$, so $\rho_0$ can lift only if $p_3$ is odd.

If $p_1,p_2,p_3$ are odd, there are two possible lifts of $\rho_0$, corresponding to the choices of orientations with angles $\pi - \frac{\pi}{p_k}$ between $s_i^-s_i^+$ and $s_j^-s_j^+$, for all $\{i,j,k\} = \{1,2,3\}$; see the pictures above.
Let $\widetilde{\rho_0} \colon \Gamma \to \SL^\pm(2,\bR)$ be the unique lift with $\tr \widetilde\rho_0(s_1s_2s_3) < 0$.
Composing with $\jmath$, we obtain the two representations
\[\rho_{\mathrm{red}}, \rho_{\mathrm{red}}' \colon \Gamma \to \SL(3,\bR), \qquad \rho_{\mathrm{red}}(\gamma) = \jmath(\widetilde{\rho_0}(\gamma)), \quad \rho_{\mathrm{red}}'(\gamma) = \jmath((-1)^{\ell(\gamma)}\widetilde{\rho_0}(\gamma)) \quad \forall \gamma \in \Gamma,\]
where $\ell(\gamma)$ is the word length of $\gamma$.

\begin{Lem}\label{lem:reducible_identify}
  The representations $\rho_{\mathrm{red}}$ and $\rho_{\mathrm{red}}'$ are of type $(\tfrac{p_1-1}{2},\tfrac{p_2-1}{2},\tfrac{p_3-1}{2})$ and are the only reducible representations of this type.
  If we define $\tred \coloneqq t_{\rho_{\mathrm{red}}}$ then $\tred > \tcrit > 1$ and $t_{\rho_{\mathrm{red}}'} = \tred^{-1}$.

  Furthermore, $\rho_{\mathrm{red}}(s_1s_2s_3)$ has eigenvalues $-\lambda,-1,\lambda^{-1}$, for some $\lambda > 1$.
\end{Lem}

\begin{Prf}
  We have $\tr \jmath(A) = \tr(A) + \det(A)$, so
  \[\tr \rho_{\mathrm{red}}(s_is_j) = \tr \jmath(\widetilde{\rho_0}(s_is_j)) = 2\cos(\pi - \tfrac{\pi}{p_k}) + 1 = 4\cos^2(\tfrac{p_k-1}{2p_k}\pi) - 1.\]
  Hence $\rho_{\mathrm{red}}$ is of type $\big(\tfrac{p_1-1}{2},\tfrac{p_2-1}{2},\tfrac{p_3-1}{2}\big)$.
  To find $t_{\rho_{\mathrm{red}}}$ we use that $\rho_{\mathrm{red}}$ is reducible, so the determinant of its Cartan matrix vanishes by \autoref{lem:reducible_cartan_matrix}.
  Hence by \eqref{eq:determinant_trace}
  \[-c_1c_2c_3(t_{\rho_{\mathrm{red}}} + t_{\rho_{\mathrm{red}}}^{-1}) = \tr \rho_{\mathrm{red}}(s_1s_2s_3) + \tr \rho_{\mathrm{red}}(s_3s_2s_1) = -2.\]
  This equation has exactly two solutions, which are positive and inverses of each other.

  In $\SL^\pm(2,\bR)$ we have $\tr A^{-1} = \tr A / \det A$, so our convention $\tr\widetilde{\rho_0}(s_1s_2s_3) < 0$ implies that $\tr\widetilde{\rho_0}(s_3s_2s_1) > 0$.
  Hence $\tr\rho_{\mathrm{red}}(s_1s_2s_3) < \tr\rho_{\mathrm{red}}(s_3s_2s_1)$.
  Due to this and since $\rho_{\mathrm{red}}(s_1s_2s_3)$ has a $-1$ eigenvalue, its eigenvalues must be of the form $-\lambda, -1, \lambda^{-1}$ for $\lambda > 1$.
  By \eqref{eq:trcox} and \eqref{eq:trcoxinverse} we also have $t_{\rho_{\mathrm{red}}} > t_{\rho_{\mathrm{red}}}^{-1}$.
  The inequality $t_{\rho_{\mathrm{red}}} > \tcrit$ then follows from \autoref{lem:coxeter_eigenvalues} and the fact that $\rho_{\mathrm{red}}(s_1s_2s_3)$ has three distinct real eigenvalues.
\end{Prf}

\begin{Def}\label{def:hitchin_and_barbot}
  A Coxeter representation $\rho \colon \Gamma_{p_1,p_2,p_3} \to \SL(3,\bR)$ is in the
  \begin{itemize}
  \item \emph{Hitchin component} if $\rho$ has type $(1,1,1)$ and $t_\rho > 0$,
  \item \emph{Barbot component} if $\rho$ has type $(\frac{p_1-1}{2},\frac{p_2-1}{2},\frac{p_3-1}{2})$ and $t_\rho > 0$.
  \end{itemize}
\end{Def}

With these definitions, the Hitchin component contains the Fuchsian representation $\rho_F$ and the Barbot component contains $\rho_{\mathrm{red}}$ and $\rho_{\mathrm{red}}'$ by \autoref{lem:irreducible_identify} and \autoref{lem:reducible_identify}.
Now we can prove \autoref{cor:traces} from \autoref{thm:main}:

  \begin{Prf}[of \autoref{cor:traces}]
    Equations \eqref{eq:trrot}, \eqref{eq:trcox_general} and \eqref{eq:trcox_inverse_general} make it easy to identify the Hitchin and Barbot components using traces:
    by \eqref{eq:trho_definition} the sign of $t_\rho$ is opposite to that of
    \[\tau \coloneqq a_{12}a_{23}a_{31} = \tr\rho(s_1s_2s_3) + \tr\rho(s_1s_2) + \tr\rho(s_2s_3) + \tr\rho(s_3s_1)\]
    If $\rho$ is not a Coxeter representation then $\tr \rho(s_is_j) \in \{-1,3\}$ for at least one distinct pair $i,j \in \{1,2,3\}$.
    So $\rho$ is in the Hitchin component if and only if $\tr\rho(s_is_j) = 4\cos^2(\frac{\pi}{p_k}) - 1 $ for $\{i,j,k\} = \{1,2,3\}$ and $\tau < 0$, and in the Barbot component if and only if $\tr\rho(s_is_j) = 4\cos^2(\frac{p_k-1}{2p_k}\pi) - 1 = 1 - 2\cos(\frac{\pi}{p_k})$ and $\tau < 0$.
    Together with \eqref{eq:discriminant}, \autoref{thm:main} therefore implies \autoref{cor:traces}.
  \end{Prf}

\subsection{Hyperbolic geometry of Coxeter axes}\label{sec:hyperbolic}

In this section, we describe some aspects of the geometry and combinatorics of Coxeter axes, which will be used in \autoref{sec:nested_boxes} and \autoref{sec:transversality}.
Again, fix a discrete and faithful representation $\rho_0 \colon \Gamma \to \PGL(2,\bR)$.
It is unique up to conjugation and its generators $\rho_0(s_1),\rho_0(s_2),\rho_0(s_3)$ are the reflections on the sides of a hyperbolic triangle $T$ with angles $\tfrac{\pi}{p_1}, \tfrac{\pi}{p_2}, \tfrac{\pi}{p_3}$.
This triangle $T$ is a fundamental domain for $\Gamma$ and its translates tile the hyperbolic plane.

Adding the axes of all conjugates of the Coxeter element $s_1s_2s_3$ (shown in red in \autoref{fig:coxeterline}) gives a finer tessellation.
To understand its geometry, we consider the union $s_2T \cup T \cup s_3T \cup s_3s_1T$ as in the right part of \autoref{fig:coxeterline}.
Let $t_1,t_2,t_3$ be the vertices of $T$ and let $A$ be the altitude triangle of $T$, that is the vertices $a_1,a_2,a_3$ of $A$ are the base points of the three altitudes of $T$.
Note that every infinite geodesic in $\bH^2$ intersects a $\Gamma$--translate of $A$, since the complement of $\Gamma A$ is a disconnected union of bounded polygons.

It is an elementary fact, true in hyperbolic as in Euclidean geometry, that the orthocenter of $T$ is the incenter of $A$; see \cite[Section VI.7]{Fenchel}.
In particular, $\measuredangle t_2 a_2 a_1 = \measuredangle t_2a_2a_3$ and hence the points $s_2a_1, a_2, a_3$ lie on a common geodesic.
By the same argument, $s_3a_1$ and $s_3s_1a_2$ are also on that geodesic, which is therefore the axis of $s_3s_1s_2$.

\begin{center}
  \begin{tikzpicture}[scale=1.55]
    \begin{scope}[shift={(-4.4,0.2)}]
      \node at (0,0) {\includegraphics[width=6.78125cm]{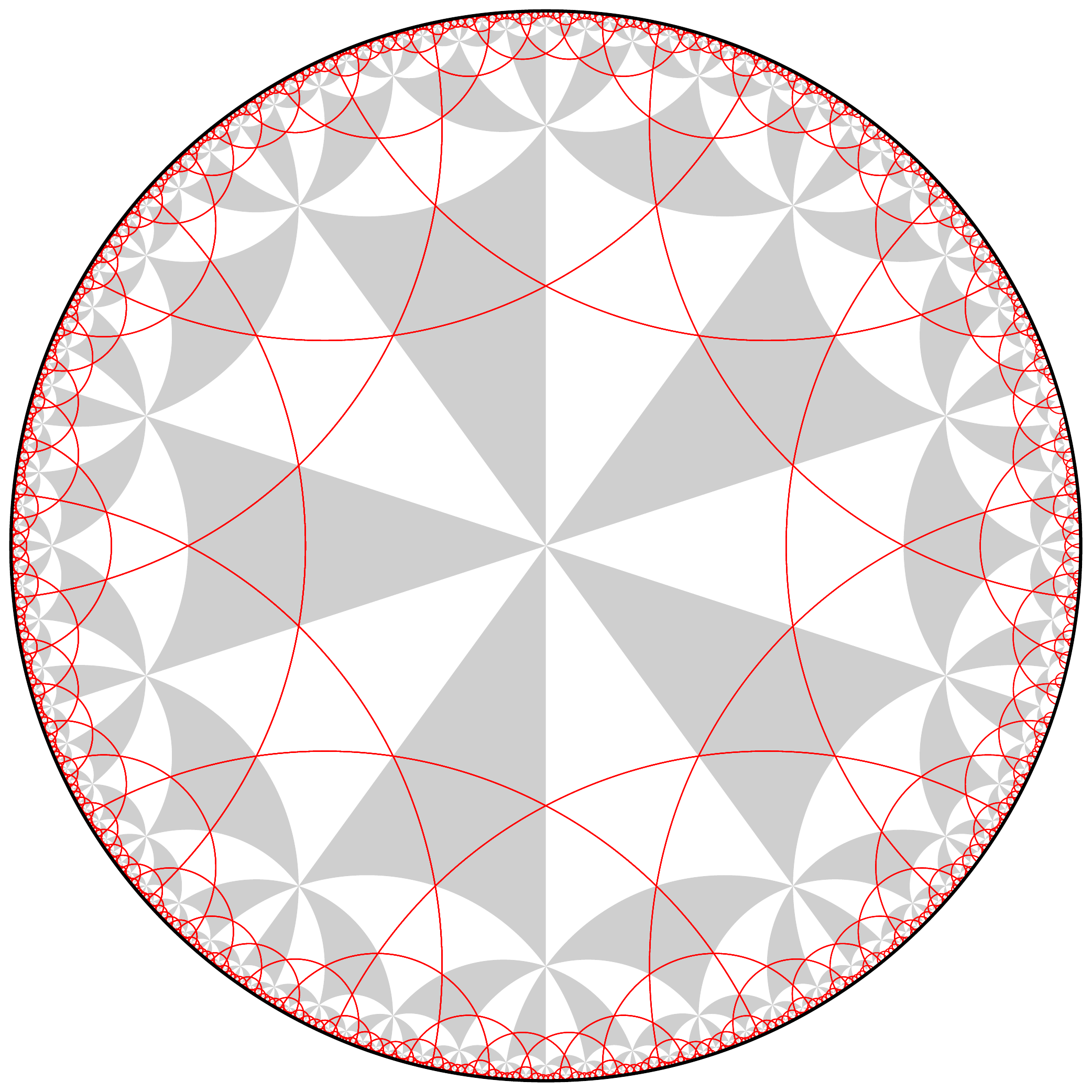}};
      \draw[thick] (0,0) circle (2.145);

      \draw[draw=ForestGreen,fill=ForestGreen!50,fill opacity=0.6] (0,0) -- (162:1.685) arc (35:-35:0.91) -- cycle;
      \node at (-0.5,0.3) {$s_1$};
      \node at (-0.5,-0.3) {$s_2$};
      \node at (-1.6,-0.2) {$s_3$};

      \coordinate (z2minus) at (257.6:2.145);
      \coordinate (z2plus) at (174.4:2.145);

      \draw[->,ultra thick] (z2minus) arc (-12.4:35.9:1.91) coordinate (midpoint);
      \draw[ultra thick] (midpoint) arc (35.9:84.2:1.91);
      \node[anchor=west,xshift=2mm,yshift=-1mm] at (midpoint) {$s_3s_1s_2$};

      \foreach \i in {0,2,4,6,8} \fill (174.4+\i*36:2.145) circle (0.05);
      \foreach \i in {1,3,5,7,9} \fill (257.6+\i*36:2.145) circle (0.05);

      \foreach \i in {0,2,4,6,8} \node at (174.4+\i*36-72:2.35) {$z_{\i}$};
      \foreach \i in {1,3,5,7,9} \node at (257.6+\i*36-144:2.35) {$z_{\i}$};
    \end{scope}

    \begin{scope}[scale=0.95]
    \coordinate (A1) at (30:0.64);
    \coordinate (A2) at (150:0.64);
    \coordinate (A3) at (270:0.64);
    \coordinate (S1) at (-10:1.05);
    \coordinate (S2) at (110:1.05);
    \coordinate (S3) at (230:1.05);

    \begin{scope}[shift={(A2)}]
      \draw[black,fill=black!20] (0,0) -- (-2:0.3) arc (-2:-30:0.3) -- cycle;
      \draw[black,fill=black!20] (0,0) -- (177:0.3) arc (177:150:0.3) -- cycle;
    \end{scope}

    \draw[gray] (90:1.545) -- (270:0.64);
    \draw[gray] (210:1.545) -- (30:0.64);
    \draw[gray] (330:1.545) -- (150:0.64);
    \draw[gray] (A2) -- (-1.98,1.147);

    \draw[thick,dotted,red] (2.15,0.47) -- (2.45,0.515);
    \draw[thick,dotted,blue] (-1.5,0.385) -- (-1.8,0.415);

    \begin{scope}
      \clip (-1.34,-0.77) arc (-60+17.5:-17.5:6.2) arc (180-235:180-281.8:2.55) arc (180-137.5:-5.5:2.48);
      \draw[thick,ForestGreen] (-2.15,0.47) arc (261.5:267.8:14.6) arc (267.8:278.5:14.6);
      \draw[thick,red] (0,-0.64) arc (27.8:35.9:14.6) coordinate (A);
      \draw[thick,blue] (A) arc (-20:-57.7:2.35);
    \end{scope}
    \begin{scope}
      \clip (-1.34,-0.77) arc (-60+17.5:-17.5:6.2) arc (180+17.5:240-17.5:6.2) arc (60+17.5:120+-17.5:6.2) -- cycle;
      \draw[thick,red] (-2.15,0.47) arc (261.5:267.8:14.6) arc (267.8:278.5:14.6);
      \draw[thick,ForestGreen] (0,-0.64) arc (27.8:35.9:14.6) arc (-20:-57.7:2.35) arc (-5:28:1);
      \draw[thick,blue] (0,-0.64) arc (180-27.8:180-35.9:14.6) arc (180+20:180+57.7:2.35) arc (180+5:180-28:1);
    \end{scope}
    \begin{scope}
      \clip (1.34,-0.77) arc (180+60-17.5:180+17.5:6.2) arc (235:281.5:2.55) arc (137.5:186:2.45);
      \draw[thick,blue] (-2.15,0.47) arc (261.5:267.8:14.6) arc (267.8:278.5:14.6);
      \draw[thick,red] (0,-0.64) arc (180-27.8:180-35.9:14.6) coordinate (A);
      \draw[thick,ForestGreen] (A) arc (180+20:180+57.7:2.35);
    \end{scope}
    \begin{scope}
      \clip (1.34,-0.77) arc (148:107:2.2) arc (250:172:0.9) arc (137.5:185:2.48) -- cycle;
      \draw[thick,ForestGreen] (-2.15,0.47) arc (261.5:267.8:14.6) arc (267.8:278.5:14.6);
      \draw[thick,blue] (0,-0.64) arc (180-27.8:180-35.9:14.6) arc (180+20:180+57.7:2.35) coordinate (A);
      \draw[thick,red] (A) arc (180+5:180-28:1);
    \end{scope}

    \draw[very thick,line join=bevel] (-1.34,-0.77) arc (-60+17.5:-17.5:6.2) coordinate arc (180-235:180-281.8:2.55) arc (180-137.5:-5.5:2.48);
    \draw[very thick,line join=bevel] (-1.34,-0.77) arc (-60+17.5:-17.5:6.2) coordinate (T3) arc (180+17.5:240-17.5:6.2) coordinate(T2) arc (60+17.5:120+-17.5:6.2) coordinate (T1) -- cycle;
    \draw[very thick,line join=bevel] (1.34,-0.77) arc (180+60-17.5:180+17.5:6.2) arc (235:281.8:2.55);
    \draw[very thick,line join=bevel] (1.34,-0.77) arc (148:107:2.2) arc (250:172:0.9) arc (137.5:185:2.48) -- cycle;

    \node[anchor=north] at (T1) {$t_3$};
    \node[anchor=north] at (T2) {$t_2$};
    \node[anchor=south] at (T3) {$t_1$};
    \node at (S1) {$s_3$};
    \node at (S2) {$s_2$};
    \node at (S3) {$s_1$};

    \node[anchor=west,yshift=-1.5mm] at (A1) {$a_3$};
    \node[anchor=east,yshift=-1.5mm] at (A2) {$a_2$};
    \node[anchor=north] at (A3) {$a_1$};
    \node at (-1.7,0.2) {$s_2a_1$};
    \node at (1.7,0.2) {$s_3a_1$};
    \node at (2.2,0.65) {$s_3s_1a_2$};

    \node at (0.15,0.6) {$T$};
    \node at (0.15,-0.05) {$A$};
  \end{scope}
  \end{tikzpicture}
  
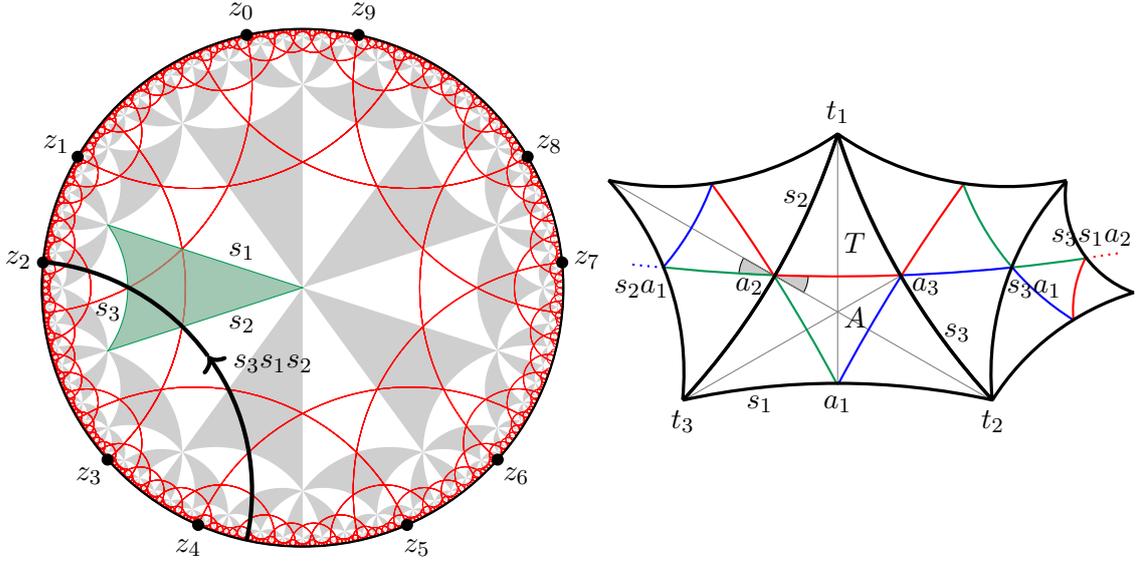
\captionof{figure}{(Left) The tiling of the hyperbolic plane in the case $p_1 = p_2 = p_3 = 5$, with the Coxeter axes in red. (Right) Four triangles along the axis of $s_3s_1s_2$.}
  \label{fig:coxeterline}
\end{center}

Let $F=s_2T \cup T \cup s_3T$ and consider the union $N$ of its orbit under the glide reflection $s_3s_1s_2$.
That is, $s_3s_1s_2$ acts on $N$ with fundamental domain $F$.
We can also define $N$ as the union of all triangles that intersect the axis of $s_3s_1s_2$ (in the $\Gamma$--tessellation of $\bH^2$ by $T$).
In any case, $N$ is a neighborhood of the axis of $s_3s_1s_2$ with two piecewise geodesic boundaries.
We claim that $N$ is convex if $p_1,p_2,p_3 \geq 3$.
To see this, take a look at the vertices on the boundary of $N$.
Every vertex belongs to three triangles in $N$ and the adjacent angles are all $\frac{\pi}{p_k}$ for some $k$.
So their sum is at most $\pi=3 \cdot \frac{\pi}{3}$.
This means $N$ is a convex neighborhood of the axis.
Hence this axis cannot intersect the reflection line of $s_1$.

Now consider the $\langle s_1,s_2 \rangle$--orbit of the point $(s_1s_2s_3)_+$, the black dots in \autoref{fig:coxeterline}. The sector bounded by the reflection lines of $s_1$ and $s_2$ contains exactly one orbit point.
Since the axis of $s_3s_1s_2$ does not intersect the reflection line of $s_1$, and intersects the reflection line of $s_2$ before the reflection line of $s_3$, $(s_3s_1s_2)_+$ is this point.
Then $(s_1s_2s_3)_+ = s_1s_2(s_3s_1s_2)_+$ is two sectors away.
So we can label the orbit points $z_0,\dots,z_{2p_3-1}$ in order along $S^1$, so that $z_0 = (s_1s_2s_3)_+$ and $z_2 = (s_3s_1s_2)_+$, and we have the following identities, which will be essential in \autoref{sec:intersecting_conics} (with indices $\mathrm{mod}\;2p_3$):
\begin{equation}\label{eq:pointlabels}
  s_1z_i = z_{3-i}, \qquad s_2z_i = z_{5-i} \qquad \forall i \in \{0,\dots,2p_3-1\}.
\end{equation}
If we repeat the same argument with $s_1$ and $s_2$ switched, we find that $(s_3s_2s_1)_+ = (s_1s_2s_3)_-$ is in the same sector of $S^1$ bounded by the reflection lines of $s_1$ and $s_2$, just like $z_2 = (s_3s_1s_2)_+$.
In particular, $(s_1s_2s_3)_-$ is in the component of $S^1 \setminus \{z_0,z_3\}$ also containing $z_1$ and $z_2$.
We will use this later in the proof of \autoref{lem:negative_coxeter_box}.

\subsection{Anosov representations}\label{sec:Anosov_representations}

  We will define Anosov representations and list their most important properties.
  Although they can be defined for any hyperbolic group $\Gamma$ and every semisimple Lie group $G$, we restrict to the case of triangle reflection groups $\Gamma$ and $G = \SL(3,\bR)$.
  As before, we fix a discrete and faithful representation $\rho_0 \colon \Gamma \to \PGL(2,\bR)$.
  The corresponding action of $\Gamma$ on $\bH^2$ extends to the visual boundary $S^1 = \partial\bH^2$, which can be identified with the Gromov boundary $\partial\Gamma$ of $\Gamma$ as a word hyperbolic group.

Let $\F$ be the the \emph{flag manifold} in $\bR^3$, that is the space of all pairs $F = (F^{(1)},F^{(2)})$ (called flags) where $F^{(i)}$ is an $i$--dimensional subspace of $\bR^3$ and $F^{(1)} \subset F^{(2)}$.
Alternatively $\F$ is the homogeneous space $\SL(3,\bR)/B$ where $B$ is the subgroup of upper triangular matrices with determinant $1$.
It carries a natural action of $\SL(3,\bR)$.
Two flags $F, F'$ are \emph{transverse} if $F^{(1)} \not\subset F'^{(2)}$ and $F'^{(1)} \not\subset F^{(2)}$.

\begin{Def}\label{def:Anosov}
  A representation $\rho \colon \Gamma \to \SL(3,\bR)$ is \emph{Anosov} if
  \begin{enumerate}
  \item there exists a map
    \[\xi \colon S^1 \to \F \]
    which is $\rho$--equivariant and continuous, maps the attracting fixed point $\gamma_+$ of every infinite order element $\gamma \in \Gamma$ to an attracting fixed point of $\rho(\gamma)$, and $\xi(x)$ and $\xi(y)$ are transverse whenever $x \neq y$, and
  \item for every sequence $\gamma_n \to \infty$ in $\Gamma$ we have
    \[\frac{\sigma_1(\rho(\gamma_n))}{\sigma_2(\rho(\gamma_n))} \to \infty, \quad \text{and} \quad \frac{\sigma_2(\rho(\gamma_n))}{\sigma_3(\rho(\gamma_n))} \to \infty\]
    where $\sigma_1(A) \geq \sigma_2(A) \geq \sigma_3(A) \geq 0$ are the singular values of a matrix $A$.
  \end{enumerate}
  Such a map $\xi$ is unique and is called the \emph{limit curve} or \emph{boundary map} of $\rho$.
  We sometimes use the same notation $\xi$ for the projection of the limit curve to $\RP^2$.
\end{Def}

\begin{Rem}
  The definition of Anosov representations given here is a characterization from \cite{GGKW}, specialized to the case of triangle group representations into $\SL(3,\bR)$.
  Definitions of Anosov representations into more general Lie groups usually have an additional qualifier, e.g. $P$--Anosov, $i$--Anosov, Borel Anosov, projective Anosov etc. In $\SL(3,\bR)$ these notions are equivalent, so we can just call them ``Anosov''.
\end{Rem}

\begin{Fact}\label{fact:anosov_properties}
Anosov representations have a number of desirable properties, including

\begin{enumerate}
\item The set of Anosov representations is open in $\Hom(\Gamma,\SL(3,\bR))$.
\item If $\rho$ and $\rho'$ define the same point in $\chi(\Gamma,\SL(3,\bR))$, then $\rho$ is Anosov if and only if $\rho'$ is Anosov.
\item The image $\rho(\Gamma)$ of an Anosov representation $\rho$ is discrete in $\SL(3,\bR)$.
\item If $\rho$ is Anosov and $\gamma \in \Gamma$ has infinite order then $\rho(\gamma)$ has distinct real eigenvalues.
\item The boundary map varies continuously with the representation. More precisely, the map $\Hom_{\mathrm{Anosov}}(\Gamma,\SL(3,\bR)) \to C^0(S^1,\F)$ mapping a representation $\rho$ to its boundary map $\xi$ is continuous.
\end{enumerate}
\end{Fact}

See \cite{GuichardWienhard,GGKW} for proofs of these facts and more information on Anosov representations.

The representations $\rho_F$, $\rho_{\mathrm{red}}$ and $\rho_{\mathrm{red}}'$ from \autoref{sec:reducible} are Anosov:
since $\iota$ and $\jmath$ map upper triangular matrices in $\PGL(2,\bR)$ and $\SL^\pm(2,\bR)$ into $B$, they induce maps $\RP^1 \to \F$, which are the boundary maps of $\rho_F$ respectively $\rho_{\mathrm{red}},\rho_{\mathrm{red}}'$.
Here $\RP^1$ is identified with $S^1 = \partial\bH^2$ as the boundary of the upper half--plane model.
It is easy to check that they satisfy all assumptions in \autoref{def:Anosov}.

It is well--known that all representations in the same component as $\rho_F$ (the Hitchin component) are Anosov \cites{ChoiGoldman,Labourie}.
More on Hitchin components of orbifold groups can be found in \cite{ALS}.

To prove that representations are Anosov, we will use another lemma from \cite{GuichardWienhard}, which says that \autoref{def:Anosov}(ii) is redundant for irreducible representations:

\begin{Fact}[{\cite[Proposition 4.10]{GuichardWienhard}}]\label{fact:irreducible_anosov}
  An irreducible representation $\rho \colon \Gamma \to \SL(3,\bR)$ is Anosov if and only if there exists a map $\xi \colon S^1 \to \F$ which is $\rho$--equivariant, continuous and transverse.
\end{Fact}

While Anosov representations of general hyperbolic groups can have a finite non--trivial kernel, they are always faithful for triangle groups.

\begin{Lem}
If a representation $\rho \colon \Gamma \to \SL(3,\bR)$ is Anosov, then it is faithful.
\end{Lem}
\begin{proof}
By \cite[Theorem 1.7]{GuichardWienhard}, the kernel of $\rho$ is finite. But, since $\Gamma$ is an irreducible infinite Coxeter group, any non--trivial normal subgroup of $\Gamma$ is infinite (see e.g. Assertion 2 in the proof of \cite[Proposition 4.3]{Paris}). Here by irreducible, we mean that $\Gamma$ cannot be written as the direct product of two non--trivial subgroups each of which is generated by a subset of the generating set $\{s_1,s_2,s_3\}$. So, the kernel of $\rho$ is trivial, i.e. $\rho$ is faithful.
\end{proof}

\begin{Rem}\label{rem:discrete_faithful_closed}
  While the set of Anosov representations is open, the set of discrete and faithful representations is a closed subset of $\Hom(\Gamma,\SL(3,\bR))$.
  A proof is given in \cite[Theorem 1.1]{GoldmanMillson} or \cite[Theorem 8.4]{Kapovich}.
\end{Rem}

\subsection{Conics in $\RP^2$}\label{sec:conics}

\autoref{sec:nested_boxes} will make extensive use of conics in $\RP^2$, so this section serves to gather some basic facts about them.
All of this material is standard and the proofs are elementary.
They usually proceed by using projective transformations to get to a standard configuration, and then doing a simple computation.

\begin{Def}
  A \emph{conic} in $\RP^2$ is the projectivization of the null cone of a quadratic form $Q$ of signature $(2,1)$ or $(1,2)$ on $\bR^3$.
\end{Def}

In an affine chart a conic appears as an ellipse, parabola, or hyperbola.
An alternative definition may allow ``degenerate conics'' from degenerate qua\-dra\-tic forms, but we require them to be non--degenerate.
Then $\SL(3,\bR)$ acts transitively on the set of conics, so every conic is projectively equivalent to the ``standard conic'' defined by $x^2 + y^2 = z^2$.

If $C$ is a conic, its complement $\RP^2 \setminus C$ has two connected components.
One of them, the ``inside'', is homeomorphic to a disk, the other is a Möbius strip.
$\SL(3,\bR)$ even acts transitively on pairs $(C,w)$ where $C$ is a conic and $w \in \RP^2$ is a point on the inside of $C$.
A standard such pair is given by the conic $x^2 + y^2 = z^2$ and the point $x=y=0$.
The subgroup of $\SL(3,\bR)$ preserving a pair $(C,w)$ is isomorphic to $\rO(2)$.
The same $\rO(2)$ also preserves a projective line, which is given by the ``orthogonal complement'' of $w$ with respect to the form defining $C$.
In the standard model, this corresponds to the line $z = 0$.

This gives an important second characterization of conics as the generic orbits of rotations in $\RP^2$.
By ``rotations'' we generally mean any 1--parameter subgroup of $\SL(3,\bR)$ whose image is isomorphic to $\SO(2)$.
All of these subgroups are conjugate, and they correspond to rotations around the $z$--axis in our standard configuration.
Hence the orbits of such a rotation group are a point, a line, and a 1--parameter family of conics, $x^2 + y^2 = cz^2$ for all $c > 0$ in the standard model.

If $g \in \SL(3,\bR)$ has finite order at least 3 then it is contained in a unique rotation subgroup.
If further $w \in \RP^2$ is not the fixed point of $g$, then its orbit $\{g^n w \mid n \in \bZ\}$ either lies on a line or on a unique $g$--invariant conic.
If $g, h \in \SL(3,\bR)$ are involutions generating a finite dihedral group where $gh$ has at least order 3, then the dihedral group is contained in a unique subgroup isomorphic to $\rO(2)$, and any $gh$--invariant conic $C$ is invariant by the entire $\rO(2)$.

We say that a set of points in $\RP^2$ is \emph{in general position} if no three of them are collinear.
$\SL(3,\bR)$ acts simply transitively on the set of quadruples of points in general position (also called \emph{projective frames)}.
Given such a quadruple, the conics passing through these form a 1--parameter family called a \emph{pencil of conics}.
A standard quadruple and the quadratic forms defining the corresponding pencil of conics are given by
\[\begin{bmatrix}1 \\ 1 \\ 1 \end{bmatrix},\begin{bmatrix}-1 \\ 1 \\ 1 \end{bmatrix},\begin{bmatrix}1 \\ -1 \\ 1 \end{bmatrix},\begin{bmatrix}1 \\ 1 \\ -1 \end{bmatrix} \quad \text{and} \quad \quad a(x^2 - z^2) + b(y^2 - z^2) = 0.\]
Here $a$ and $b$ are any real parameters, giving us a two--dimensional space of quadratic forms, but they only have signature $(2,1)$ or $(1,2)$ if $a \neq 0$, $b \neq 0$, and $a+b \neq 0$.
Projectivizing these forms gives a one--dimensional family of conics with 3 connected components (characterized by whether $e_1$, $e_2$, or $e_3$ lies inside the conic).

These conics sweep exactly once over all points in $\RP^2$ not collinear with any two of the quadruple points.
Hence, for every quintuple of points in general position, there is a unique conic passing through all of them.
A direct consequence of this is that two different conics can intersect in at most 4 points.
In fact, any number of intersections from 0 to 4 is possible.
If two different conics intersect in 4 points, the intersections are necessarily transverse.

\section{Non--Anosov components}\label{sec:non_anosov}

In this section we show that an Anosov representation $\rho \colon \Gamma_{p_1,p_2,p_3} \to \SL(3,\bR)$ of a compact hyperbolic triangle reflection group is either of type $(1,1,1)$ or $p_1,p_2,p_3$ are all odd and $\rho$ is of type $(\frac{p_1-1}{2}, \frac{p_2-1}{2}, \frac{p_3-1}{2})$ (\autoref{def:cartan_matrix}).
The basic topological reason is that a loop with winding number greater than $2$ cannot be embedded in $\RP^2$.
We can restrict our attention to Coxeter representations as all others have a finite image (trivial, $\bZ/2\bZ$ or a dihedral group), and thus cannot be Anosov.

\begin{center}
  \begin{tikzpicture}
    \begin{scope}[xshift=0cm]
      \node at (0,-2.2) {$k = 1$};
      \node[blue] at (1.2,-1) {$\gamma$};
      
      \foreach \i in {90,141,193,244,296,347,399,450}
      {
        \begin{scope}[rotate=\i]
          \draw[thick,blue] (0:1) arc (-60:111:0.43);
        \end{scope}
      };
      
      \draw[fill=black!20] (0,0) circle (1);
      \foreach \i in {90,141,193,244,296,347,399,450} \fill (\i:1) circle (0.05);
      \node at (90:0.75) {$0$};
      \node at (141:0.75) {$1$};
      \node at (193:0.75) {$2$};
      \node at (244:0.75) {$3$};
      \node at (296:0.75) {$4$};
      \node at (347:0.75) {$5$};
      \node at (399:0.75) {$6$};
      \node at (0,0) {$S^2\!\setminus\!D$};
    \end{scope}
    
    \begin{scope}[xshift=8cm]
      \node at (0,-2.2) {$k = 3$};
      \node[blue] at (0.8,-1.2) {$\gamma$};
      
      \foreach \i in {90,141,193,244,296,347,399,450}
      {
        \begin{scope}[rotate=\i]
          \draw[thick,blue] (0:1) -- (-5:1.8);
          \draw[thick,blue] (154:1) -- (159:1.8);
        \end{scope}
      };
      
      \draw[fill=black!20] (0,0) circle (1);
      \foreach \i in {90,141,193,244,296,347,399,450} \fill (\i:1) circle (0.05);
      \node at (90:0.75) {$0$};
      \node at (141:0.75) {$5$};
      \node at (193:0.75) {$3$};
      \node at (244:0.75) {$1$};
      \node at (296:0.75) {$6$};
      \node at (347:0.75) {$4$};
      \node at (399:0.75) {$2$};
      \node at (0,0) {$S^2\!\setminus\!D$};
    \end{scope}
    \begin{scope}[xshift=4cm]
      \node at (0,-2.2) {$k = 2$};
      \node[blue] at (1.5,-1.1) {$\gamma$};
      
      \foreach \i in {90,141,193,244,296,347,399,450}
      {
        \begin{scope}[rotate=\i]
          \draw[dashed,blue] (0:1) arc (-60:162:0.85);
        \end{scope}
      };
      
      \draw[fill=black!20] (0,0) circle (1);
      \foreach \i in {90,141,193,244,296,347,399,450} \fill (\i:1) circle (0.05);
      \node at (90:0.75) {$0$};
      \node at (141:0.75) {$4$};
      \node at (193:0.75) {$1$};
      \node at (244:0.75) {$5$};
      \node at (296:0.75) {$2$};
      \node at (347:0.75) {$6$};
      \node at (399:0.75) {$3$};
      \node at (0,0) {$S^2\!\setminus\!D$};      
    \end{scope}
  \end{tikzpicture}
  \captionof{figure}{The idea of \autoref{lem:nonintersecting_symmetric_curves} in the case $p = 7$. If $R$ is the order $7$ rotation around the center of the disk $D$ (or equivalently, rotation around the center of $S^2 \setminus D$), an $R$--invariant injective curve $\gamma$ can pass through the 7 orbit points in the order required for $k=1$ or $k=3$, but not $k=2$.}
  \label{fig:jordan_curve}
\end{center}

\begin{Prop}\label{lem:nonintersecting_symmetric_curves}
  For integers $p \geq 3$ and $1 \leq k \leq \frac{p}{2}$, let $R \in \SL(3,\bR)$ be a rotation by the angle $\frac{2\pi}{p}$ (that is, $R$ has eigenvalues $1$ and $e^{\pm 2\pi i /p}$) and $\gamma \colon S^1 \to \RP^2$ an injective continuous curve satisfying $\gamma(t+ \frac{1}{p}) = R^k\gamma(t)$ for all $t \in S^1 \cong \bR / \bZ$. Then
  \begin{enumerate}
  \item either $k=1$ and $\gamma$ is null--homotopic,
  \item or $p$ is odd, $k = \frac{p-1}{2}$ and $\gamma$ is not null--homotopic.
  \end{enumerate}
\end{Prop}

\begin{Prf}
  Observe that $\gamma(\frac{1}{\gcd(k,p)}) = \gamma(0)$, so $k$ and $p$ must be coprime.
  In this case $k$ is invertible modulo $p$, i.e. there is an integer $1 \leq l < p$ with $kl \equiv 1$ mod $p$.
  If $p = 3$ the lemma is trivially true, so we can assume $p \geq 4$.

  We will use this simple consequence of the Jordan curve theorem: If $x,y,z,w$ are distinct points on the boundary of a disk in this cyclic order, and $x$ and $z$ as well as $y$ and $w$ are connected by curves in the closed disk, then these curves intersect.

  We pass to the universal cover $S^2 \to \RP^2$ and write $\iota \colon S^2 \to S^2$ for its non--trivial deck transformation, the antipodal involution. Let $\widehat\gamma \colon [0,1] \to S^2$ be one lift of $\gamma$, the other one being $\iota \circ \widehat\gamma$. The matrix $R$ still acts as a rotation by $\frac{2\pi}{p}$ on $S^2$. Its two fixed points cannot be in the image of $\widehat\gamma$ or $\iota \circ \widehat\gamma$. Choose one of them and let $D$ be the smallest $R$--invariant elliptic (i.e. bounded by a conic) closed disk around it which contains the images of $\widehat\gamma$ and $\iota \circ \widehat\gamma$. We can assume that $\partial D$ intersects $\widehat\gamma$ in at least one point $w$ (otherwise replace $D$ by $\iota D$). We may also assume that $\widehat\gamma(0) = w$.

  The symmetry of $\gamma$ can lift in two ways: either $\widehat\gamma(t+\frac{1}{p}) = R^k\widehat\gamma(t)$ or $\widehat\gamma(t+\frac{1}{p}) = \iota(R^k\widehat\gamma(t))$. By continuity one of these relations holds for all $t \in S^1$. In the first case, consider the arcs $\widehat\gamma|_{[0,1/p]}$ and $\widehat\gamma|_{[l/p,(l+1)/p]}$. Their endpoints are
  \[\textstyle\widehat\gamma(\{0,\frac{1}{p}\}) = \{w, R^kw\}, \quad \widehat\gamma(\{\frac{l}{p},\frac{l+1}{p}\}) = \{Rw, R^{k+1}w\}.\]
  If $k \neq 1$ then these four points are distinct and their cyclic order along $\partial D$ is $w$,$Rw$,$R^kw$, $R^{k+1}w$. So the arcs have to intersect, which is a contradiction to the injectivity of $\gamma$. Furthermore, $\widehat\gamma(1) = R^{kp}\widehat\gamma(0) = \widehat\gamma(0)$, so $\gamma$ is null--homotopic.

  Now assume the second case, $\widehat\gamma(t+\frac{1}{p}) = \iota(R^k\widehat\gamma(t))$. Then we consider instead the arcs $\widehat\gamma|_{[0,2/p]}$ and $\iota^l \circ \widehat\gamma|_{[l/p,(l+2)/p]}$. Their endpoints are
  \[\textstyle\widehat\gamma(\{0,\frac{2}{p}\}) = \{w,R^{2k}w\}, \quad \iota^l\widehat\gamma(\{\frac{l}{p},\frac{l+2}{p}\}) = \{Rw, R^{2k+1}w\}.\]
  If these points are distinct their cyclic order is $w,Rw,R^{2k}w,R^{2k+1}w$, which again contradicts $\gamma$ being injective. So $2k$ must be congruent to $-1$, $0$, or $1$ modulo $p$. As $k$ and $p$ are coprime, this only happens if $p$ is odd and $k = \frac{p-1}{2}$. In this case $\widehat\gamma(1) = \iota^pR^{kp}\widehat\gamma(0) = \iota\widehat\gamma(0)$ since $p$ is odd, so $\gamma$ is not null--homotopic.
\end{Prf}

\begin{Lem}\label{lem:orthogonal_anosov}
  Assume that one of $p_1,p_2,p_3$ equals $2$ and $\rho \colon \Gamma \to \SL(3,\bR)$ is a Coxeter representation which is Anosov.
  Then $\rho$ is of type $(1,1,1)$.
\end{Lem}

\begin{Prf}
  We showed in the proof of \autoref{lem:homeoToR} that if one of $p_1,p_2,p_3$ equals $2$, then the Cartan matrix $A = ( \alpha_i(b_j) )_{1 \leq i, j \leq 3} = ( a_{ij} )_{1 \leq i, j \leq 3}$ is equivalent to a symmetric matrix.
  We may assume that $A$ is symmetric.
  Then there exists a scalar product $(\cdot,\cdot)$ in $\Span\{b_i\}_{1\leq i \leq 3}$ such that $(b_i,b_j) = a_{ij} = a_{ji}$, and this scalar product is $\rho(\Gamma)$-invariant.
  Since every principal $2 \times 2$ submatrix of $A$ is positive definite, the signature of $A$ is $(3,0,0)$, $(2,1,0)$ or $(2,0,1)$.
  Here, a symmetric matrix has \emph{signature $(p,q,r)$} if the triple $(p,q,r)$ is the number of positive, negative and zero eigenvalues (counted with multiplicity).

  In the case $(p,q,r) = (3,0,0)$, the image $\rho(\Gamma)$ lies in a compact subgroup, which is (a conjugate of) $\rO(3)$, hence $\rho$ cannot be faithful with discrete image.
  
  In the case $(p,q,r) = (2,1,0)$, the image $\rho(\Gamma)$ lies in $\SO(2,1)$ and acts convex cocompactly on the hyperbolic plane $\mathbb{H}^2$ (see e.g. \cite[Theorem 1.8]{GuichardWienhard}). Since $\partial\Gamma$ is homeomorphic to $S^1$, so is the limit set $\Lambda_\rho$ of $\rho(\Gamma)$, which lies in $\partial\mathbb{H}^2$. Consequently, $\Lambda_\rho = \partial\mathbb{H}^2$ and the convex hull of $\Lambda_\rho$ in $\mathbb{H}^2$ is the entire $\mathbb{H}^2$, i.e. the action of $\rho(\Gamma)$ on $\mathbb{H}^2$ is cocompact.
  After possibly negating the generators, we can apply \cite[Lemma 5.4]{LeeMarquis}, to obtain that $\rho(\Gamma)$ is a hyperbolic reflection group, hence $\rho$ is of Hitchin type.

  In the case of $(p,q,r) = (2,0,1)$, the image $\rho(\Gamma)$ lies in $\rO(2) \rtimes \bR^2$. This is a contradiction by Bieberbach's Theorem.
\end{Prf}

\begin{Prop}\label{prop:non_anosov}
  Let $\rho \colon \Gamma_{p_1,p_2,p_3} \to \SL(3,\bR)$ be a representation of type $(q_1,q_2,q_3)$ which is Anosov. Then either $q_1=q_2=q_3=1$ or $p_1,p_2,p_3$ are all odd and $q_i = \frac{p_i-1}{2}$ for all $i \in \{1,2,3\}$.
  Furthermore $t_\rho > 0$, so $\rho$ is in the Hitchin or Barbot component.
\end{Prop}

\begin{Prf}
  We can assume $p_1,p_2,p_3 \neq 2$, otherwise this follows from \autoref{lem:orthogonal_anosov}.
  As $\rho$ is Anosov, it comes with an continuous, injective, and equivariant boundary map $\xi \colon S^1 \to \RP^2$.
  Let $R \in \SL(3,\bR)$ be the rotation in the 1--parameter subgroup containing $\rho(s_1s_2)$, but only by the angle $\frac{2\pi}{p_3}$, so that $\rho(s_1s_2) = R^{q_3}$.
  If we parametrize $\partial\Gamma_{p_1,p_2,p_3} = S^1$ by the unit interval so that the rotation $s_1s_2$ is a shift by $\frac{1}{p_3}$, then the assumptions of \autoref{lem:nonintersecting_symmetric_curves} are satisfied.

  So if $\xi$ is null--homotopic then $q_3 = 1$, while if $\xi$ is not null--homotopic $p_3$ must be odd and $q_3 = \frac{p_3-1}{2}$.
  We can repeat the argument for the rotations $s_2s_3$ and $s_3s_1$ to obtain the analogous constraints for $q_1$ and $q_2$.

    We claim that $t_\rho > 0$. In the case of $q_i = \frac{p_i-1}{2}$ for all $i \in \{1,2,3\}$, it follows from \autoref{lem:coxeter_eigenvalues} since $\rho(s_1s_2s_3)$ has distinct real eigenvalues by \autoref{fact:anosov_properties}(iv).
    Now we assume that $q_i = 1$ for all $i \in \{1,2,3\}$.
    Consider two lines in the image of the dual boundary map, splitting $\RP^2$ into two bigons.
    As $\xi$ is transverse, it intersects each of them exactly once and consists of two arcs between them.
    Since it is null--homotopic, both arcs must lie in the same (closed) bigon. So $\xi(S^1)$ is contained in an affine chart.
    Then its convex hull $C$ is a properly convex set preserved by $\rho(\Gamma)$.
    By the same reason as in the proof of \autoref{lem:orthogonal_anosov}, $\rho(\Gamma)$ acts properly discontinuously and cocompactly on the interior $C^\circ$. Hence $\rho$ is of Hitchin type again by \cite[Lemma 5.4]{LeeMarquis}. That is, $t_\rho > 0$.
  \end{Prf}

\section{Constructing a boundary map}\label{sec:limit_curve}

The setup in this section is more general than in the rest of the paper. Let $\Gamma$ be a cocompact discrete subgroup of the isometries of the hyperbolic plane $\bH^2$, i.e. the fundamental group of a closed hyperbolic $2$--orbifold. Its action extends to $\overline{\bH^2} = \bH^2 \sqcup S^1$.

Let $I \subset S^1$ be a proper closed interval and $T \subset \Gamma$ a finite subset which satisfies
\begin{enumerate}
\item $\gamma I \subsetneq I$ for all $\gamma \in T$,
\item if $\gamma \in T$ fixes an endpoint $x \in \partial I$, then $\gamma$ is hyperbolic and $x$ is its attracting fixed point,
\item $\bigcup_{\gamma \in T} \gamma I = I$.
\end{enumerate}

Let further $\rho \colon \Gamma \to \SL(3,\bR)$ be a representation, and let $\bx \subset \RP^2$ be a closed set with non--empty interior which is contained in an affine chart.
Assume there exists $N \in \bN$ and a special element $\overline t \in T$ such that
\begin{enumerate}
  \setcounter{enumi}{3}
\item $\rho(t) \bx \subset \bx$ for all $t \in T$,
\item $\rho(t_1) \cdots \rho(t_N) \bx \subset \bx^\circ$ if $t_1, \dots, t_N \in T$ and $t_N \neq \overline t$,
\item the intersection of the sets $\rho(\overline t^i)\bx$ for all $i \in \bN$ is a point.
\end{enumerate}
The goal of this section is to show under these assumptions

\begin{Prop}\label{prop:limit_map_exists}
  There exists a $\rho$--equivariant continuous map $\xi \colon S^1 \to \RP^2$ satisfying $\xi(I) \subset \bx$.
\end{Prop}

In \autoref{sec:definition_of_I_T_box} we will define a concrete interval $I$, a set $T$ of group elements and a closed set $\bx$ in $\RP^2$ which satisfy the assumptions above, in the case of Barbot representations of triangle groups.

The proof of \autoref{prop:limit_map_exists} needs some more setup. Fix a basepoint $o \in \bH^2$ and a finite generating set $S \subset \Gamma$, and denote by $\ell \colon \Gamma \to \bN_0$ the word length in $S$. We call a sequence $(g_n)_n \in \Gamma^\bN$ a quasigeodesic ray to $z \in S^1$ if $g_n o \to z$ and $(g_no)_n$ is a quasigeodesic ray in $\bH^2$, i.e. $\dd(o,g_no)$ is bounded by increasing affine linear functions of $n$ from below and above.

We call a sequence $(g_n)_n \in \Gamma^{\bN}$ an \emph{$I$--code} for $z \in S^1$ if $g_n^{-1}g_{n+1} \in T$ and $g_n^{-1}z \in I$ for all $n \geq 1$.
Note that $g_1$ can be any element of $\Gamma$ satisfying the second condition.
By assumption (iii) and the minimality of the $\Gamma$--action on $S^1$, there is an $I$--code for every $z \in S^1$, although it is usually not unique.

To construct the boundary map $\xi$ for a representation $\rho$, we will take an $I$--code $(g_n)_n$ for $z \in S^1$, apply $\rho$ to it, and then take $\xi(z) \in \RP^2$ to be the limit of $(\rho(g_n))_n$ in a suitable sense.
To make this work, we need to show that the limit exists and that it does not depend on the chosen $I$--code for $z$.
The latter part is general and the content of \autoref{lem:quasigeodesic} and \autoref{lem:quasigeodesics_have_same_limits}.
To show the existence of the limit in \autoref{lem:convergent_boxes}, we will use the set $\bx$ which gets mapped into itself by the elements of $\rho(T)$.
The objective of \autoref{sec:nested_boxes} will then be to find such a set for the representations we are interested in.

\begin{Rem}
  The uniqueness part is inspired by \cite[Section 9]{Sullivan} and the existence part follows the strategy of \cite{Schwartz}.
  A similar criterion for general Anosov representations is shown in \cite[Section 5]{BPS}.
\end{Rem}

What does it mean for a sequence $(g_n)_n \in \SL(3,\bR)^\bN$ to converge to $x \in \RP^2$? Let $\mu$ be the unique $\SO(3)$--invariant Borel probability measure on $\RP^2$. Then we say $g_n \to x$ if $(g_n)_*\mu \to \delta_x$ in the weak topology of measures, where $\delta_x$ is the Dirac measure at $x$. Explicitly, this means that
\[\int_{\RP^2} f \circ g_n\,\dd\mu \to f(x)\]
for every continuous function $f$ on $\RP^2$. This mode of convergence is equivalent to the ``flag convergence'' defined in \cite{KLP1}. If $g_n \to x$ and $g \in \SL(3,\bR)$, it is clear from the definition that $gg_n \to gx$, while on the other hand $g_n g \to x$.

\begin{Lem}\label{lem:quasigeodesic}
  Let $(g_n)_n \in \Gamma^{\bN}$ be an $I$--code for $z \in S^1$. Then $(g_no)_n$ is a quasigeodesic ray and $g_no \to z$ in $\overline{\bH^2}$.
\end{Lem}

\begin{Prf}
  Let $I'$ be a slightly enlarged version of $I$ so that $\gamma I' \subset I'^\circ$ for all $\gamma \in T$. The existence of $I'$ is guaranteed by properties (i) and (ii), but $I'$ will not satisfy property (iii). Let $\alpha$ be the hyperbolic geodesic connecting the two points of $\partial I'$. Then $\dist(\gamma \alpha, \alpha) > 0$ for all $\gamma \in T$. Let $C$ be the minimum of these distances. Since the quasigeodesic property and the limit of $g_no$ do not depend on the basepoint $o$, we can assume that $o \in \alpha$. So
  \[d(g_no, g_1o) \geq \dist(g_n\alpha, g_1\alpha) \geq \sum_{i=1}^{n-1}\dist(g_{i+1} \alpha, g_i\alpha) = \sum_{i=1}^{n-1}\dist(g_i^{-1}g_{i+1} \alpha, \alpha) \geq C(n-1).\]
  Here the second step is due to the fact that the $g_i\alpha$ don't intersect each other, so the geodesic realizing the distance between $g_1\alpha$ and $g_n\alpha$ is split into segments by the other $g_i\alpha$.
  The inequality then shows that $(g_no)_n$ is a quasigeodesic ray (the upper bound is clear). Its limit in $S^1$ is in $g_n I$ for all $n$, so it must be $z$.
\end{Prf}

\begin{Lem}\label{lem:code_difference}
  Let $(g_n)_n, (g_n')_n \in \Gamma^\bN$ be quasigeodesic rays to $z \in S^1$. There exists $N \in \bN$ such that for every $n \in \bN$ there is an $m(n) \in \bN$ with
  \[\ell(g_{m(n)}^{-1} g_n') \leq N.\]
\end{Lem}

\begin{Prf}
  Say $(z_n)_n = (g_no)_n$ and $(z_n')_n = (g_n'o)_n$ are both $(K,C)$--quasigeodesic rays from $o$ to $z$, for some $K$ and $C$. The Morse lemma tells us that both are contained in the $R$--neighborhood of the geodesic $oz$, for some $R$. Denote by $\pi$ the closest point projection of $\bH^2$ onto this geodesic. Since $d(\pi(z_n), \pi(z_{n+1})) \leq d(z_n,z_{n+1}) \leq K + C$, every point on the ray $oz$ is at most distance $R' = \max\{(K+C)/2,d(o,\pi(z_1))\}$ from some $\pi(z_m)$.

  Now for every $n \in \bN$, choose $m(n)$ such that $d(\pi(z_{m(n)}), \pi(z_n')) \leq R'$, and therefore
  \[d(z_{m(n)}, z_n') \leq d(\pi(z_{m(n)}), \pi(z_n')) + 2R \leq R' + 2R.\]
  The statement of the lemma follows since the orbit map of $\Gamma$ is a quasiisometry from the word metric on $\Gamma$ given by $d_\Gamma(g,h) = \ell(g^{-1}h)$.
\end{Prf}

\begin{Lem}\label{lem:quasigeodesics_have_same_limits}
  Let $(g_n)_n, (g_n')_n \in \Gamma^\bN$ be quasigeodesic rays to $z \in S^1$. If $\rho(g_n) \to x$ for some $x \in \RP^2$, then also $\rho(g_n') \to x$.
\end{Lem}

\begin{Prf}
  By \autoref{lem:code_difference} there is a sequence $m(n)$ such that $g_n' = g_{m(n)} h_n$ with the $h_n$ coming from a finite set. Clearly $m(n) \to \infty$ as $n \to \infty$, so $\rho(g_{m(n)}) \to x$. For every subsequence along which $h_n$ is constant we have $\rho(g_n') \to x$, so the same is true for the entire sequence.
\end{Prf}

\begin{Lem}\label{lem:vanishing_diam_flag_convergence}
  Let $A \subset \RP^2$ be closed with non--empty interior and $(g_n)_n \in \SL(3,\bR)^\bN$ a sequence satisfying $g_{n+1}A \subset g_n A$ as well as $\operatorname{diam}(g_nA) \to 0$ (in any Riemannian metric on $\RP^2$). Then $g_n \to x$ where $x$ is the unique element in the intersection $\bigcap_{n \in \bN} g_n A$.
\end{Lem}

\begin{Prf}
  By compactness of $A$ the choice of Riemannian metric doesn't matter.
  So we work with the spherical metric on $\RP^2$.
  Let $g_n = k_n a_n l_n$ be a singular value decomposition for $g_n$, that is $k_n, l_n \in \SO(3)$ and $a_n$ is a diagonal matrix with entries $\lambda_{1,n},\lambda_{2,n},\lambda_{3,n}$ sorted by absolute values, so that $|\lambda_{1,n}| \geq |\lambda_{2,n}| \geq |\lambda_{3,n}|$.
  Passing to a subsequence, we can assume $k_n \to k$ and $l_n \to l$.

  Then $lA$ contains an open rectangle in homogeneous coordinates, which by an elementary computation is compressed to a point only if $\lambda_{2,n}/\lambda_{1,n} \to 0$.
  This implies that $(a_n)_*\mu$ converges to the Dirac measure at $[e_1] \in \RP^2$, hence $(g_n)_*\mu \to \delta_x$ with $x = k[e_1]$.

  Now whenever $n \geq m$ then $g_nA \subset g_mA$ by assumption, so (using \cite[Theorem 13.16]{KlenkeProbabilityTheory})
  \[\delta_x(g_mA) \geq \limsup_{n\to\infty} \,(g_n)_* \mu(g_mA) = \limsup_{n\to\infty} \mu(g_n^{-1}g_mA)\geq \mu(A) > 0,\]
  hence $x \in g_mA$ for all $m$. This $x$ is unique since $\operatorname{diam}(g_nA) \to 0$, so the whole sequence converges.
\end{Prf}

\begin{center}
  \begin{tikzpicture}[scale=0.8]
    \draw[fill=black!20,rounded corners=0.3cm] (-4,-2) -- (-4.75,-1) -- (-4,2) -- (2,2) -- (2.75,1) -- (2,-2) -- cycle;

    \begin{scope}[scale=0.5,rotate=15,xshift=-1cm]
      \draw[fill=black!40,rounded corners=0.15cm] (-4,-2) -- (-4.75,-1) -- (-4,2) -- (2,2) -- (2.75,1) -- (2,-2) -- cycle;
    \end{scope}

    \draw[thick] (-2.58,-1.27) -- (0.68,0.7);
    \draw[thick] (-2.08,-2) -- (-2.08,-1.58);

    \draw[thick,blue] (-0.5,-2.5) -- (-1.5,2.5);

    \fill[blue] (-1.4,2) circle (0.08);
    \fill[blue] (-1.145,0.725) circle (0.08);
    \fill[blue] (-0.75,-1.25) circle (0.08);
    \fill[blue] (-0.6,-2) circle (0.08);

    \node[blue,anchor=south west] at (-1.4,2) {$a$};
    \node[blue,anchor=south east] at (-1.145,0.725) {$x$};
    \node[blue,anchor=north west] at (-0.75,-1.25) {$y$};
    \node[blue,anchor=north east] at (-0.6,-2) {$b$};
    
    \node at (0,-0.5) {$\rho(h)\bx$};
    \node at (1.5,-1) {$\bx$};
    \node at (-1.8,-0.55) {$D$};
    \node at (-2.3,-1.75) {$d$};
  \end{tikzpicture}
  \captionof{figure}{The proof of \autoref{lem:convergent_boxes}. The cross--ratio $[a:x:y:b]$ is bounded by a function of the distances $d$ and $D$. As the cross--ratio is a projective invariant, the same upper bound holds for $\rho(g_{i_k-N})\bx$ and $\rho(g_{i_k})\bx$ in place of $\bx$ and $\rho(h)\bx$, if $g_{i_k-N}^{-1}g_{i_k} = h$. The set of such words is finite, so we get a uniform bound.}
  \label{fig:cross_ratios}
\end{center}

Now let $\bx$ be the set whose existence we assumed in the beginning of the section.

\begin{Lem}\label{lem:convergent_boxes}
  Let $(g_i)_i \in \Gamma^{\bN}$ be an $I$--code for $z \in S^1$.
  Then $\diam(\rho(g_i)\bx) \to 0$ and $\rho(g_i)$ converges to the unique point $y \in \bigcap_{i\in\bN} \rho(g_i)\bx$.
\end{Lem}

\begin{Prf}
  We need to distinguish two cases: either $g_{i-1}^{-1}g_i = \overline t$ for all but finitely many $i$, or not.
  In the first case, the assumptions in the beginning of the section tell us that the intersection of the sets $\rho(g_i)\bx$ is a single point, hence $\diam(\rho(g_i)\bx) \to 0$.
  With \autoref{lem:vanishing_diam_flag_convergence} this proves the lemma.
  So we now assume that $g_{i-1}^{-1}g_i \neq \overline t$ for infinitely many $i$.

  Fix an affine chart containing $\bx$ and work with the Euclidean metric in this chart.
  Let $h = t_1 \cdots t_N$ be any product of $N$ elements of $T$ with $t_N \neq \overline t$, and let $x, y \in \rho(h)\bx$ as well as $a,b \in \RP^2 \setminus \bx^\circ$ be such that $a,x,y,b$ lie on a projective line in that order.
  Let $D = \diam(\rho(h)\bx)$ and let $d$ be the minimal distance between $\rho(h)\bx$ and $\RP^2 \setminus \bx^\circ$, which is positive since $\rho(h)\bx \subset \bx^\circ$.
  Then the cross ratio satisfies
  \[[a:x:y:b] = \frac{|y-a||b-x|}{|x-a||b-y|} = \left(1 + \frac{|y-x|}{|x-a|}\right)\left(1 + \frac{|y-x|}{|b-y|}\right) \leq (1 + D/d)^2.\]
  Doing this for any $h$ of this form gives a uniform upper bound on these cross ratios.

  We want to show that $\operatorname{diam}(\rho(g_i) \bx) \to 0$ and then employ \autoref{lem:vanishing_diam_flag_convergence}.
  It is clear that this sequence is non--increasing.
  Assume it converges to $c > 0$.
  Then choose, for every $i$, points $x_i, y_i \in \rho(g_i) \bx$ with $|y_i - x_i| = \operatorname{diam}(\rho(g_i)\bx)$.
  Let $a_i,b_i$ be the closest points of the boundary of $\rho(g_{i-N})\bx$ on the projective line through $x_i$ and $y_i$ in either direction, so that the points are ordered $a_i,x_i,y_i,b_i$.

  Then $|y_i - x_i| \to c$ and $|y_i - x_i| < |b_i - a_i| \leq \operatorname{diam}(\rho(g_{i-N})\bx)$, so also $|b_i - a_i| \to c$. By the way the points are ordered, $|y_i - a_i|$ and $|b_i - x_i|$ must also converge to $c$, while $|x_i - a_i|$ and $|b_i - y_i|$ go to $0$.
  Hence the cross ratio $[a_i:x_i:y_i:b_i]$ goes to $\infty$.

  Now choose a subsequence $(g_{i_k})$ for which $g_{i_k-1}^{-1}g_{i_k} \neq \overline t$, and also $i_k \geq i_{k-1} + N$.
  Then since the cross ratio is a projective invariant and $g_{i_k-N}^{-1}g_{i_k}$ is a product of $N$ elements of $T$, the last one different from $\overline t$, the cross ratio $[a_{i_k}:x_{i_k}:y_{i_k}:b_{i_k}]$ equals one of the cross ratios we bounded above, for every $k$.
  This is a contradiction, so $\diam(\rho(g_i)\bx) \to 0$ and $\rho(g_i)$ converges by \autoref{lem:vanishing_diam_flag_convergence}.
\end{Prf}

\begin{Def}
  We define the map
  \[\xi \colon S^1 \to \RP^2\]
  by requiring that $\rho(g_n) \to \xi(z)$ for every $I$--code $(g_n)_n \in \Gamma^{\bN}$ for $z \in S^1$.
\end{Def}

We know that such an $I$--code exists for every $z \in S^1$ and that $\rho(g_n)$ converges by \autoref{lem:convergent_boxes}. The limit is independent of the choice of $I$--code by \autoref{lem:quasigeodesic} and \autoref{lem:quasigeodesics_have_same_limits}.
More generally, $\rho(g_n) \to \xi(z)$ for any quasigeodesic ray $(g_n)_n \in \Gamma^{\bN}$ such that $g_no \to z$. So $\xi$ is well--defined.
Also note that every $z \in I$ has an $I$--code $(g_n)_n$ with $g_1 = 1$, so $\xi(I) \subset \bx$.

\begin{Lem}
  $\xi$ is $\rho$--equivariant.
\end{Lem}

\begin{Prf}
  Let $z \in S^1$ and $g \in \Gamma$. If $(g_n)_n \in \Gamma^{\bN}$ is a quasigeodesic ray going to $z$, then $(gg_n)_n$ is a quasigeodesic ray going to $gz$. So
  \[\xi(gz) \leftarrow \rho(g g_n) = \rho(g)\rho(g_n) \rightarrow \rho(g)\xi(z). \qedhere\]
  This shows equivariance.
\end{Prf}

\begin{Lem}
  $\xi$ is continuous.
\end{Lem}

\begin{Prf}
  Let $z \in S^1$. We inductively construct an $I$--code $(g_n)_n \in \Gamma^\bN$ for $z$: first we choose $g_1$ so that $g_1^{-1}z \in I^\circ$.
  It exists since $I^\circ \neq \varnothing$ and $\Gamma$ acts minimally on $S^1$.
  Then for every $n \geq 1$, since $g_n^{-1}z \in I$, property (iii) ensures that $g_n^{-1}z \in tI$ for some $t \in T$.
  We set $g_{n+1} = g_n t$.
  If possible, we choose $t$ so that $g_n^{-1}z \in tI^\circ$.
  Otherwise, if $g_n^{-1}z$ is in $I^\circ$ but not in $tI^\circ$ for any $t \in T$, then there are at least two different choices for $t$, of which we choose one which results in $z$ being the clockwise boundary point of $g_{n+1}I$ (for an arbitrary choice of orientation on $S^1$).
  Note that $z$ is then also the clockwise boundary point of all $g_m I$ for $m > n$.
  The sequence constructed this way is a quasigeodesic ray going to $z$ by \autoref{lem:quasigeodesic}.

  Let $\varepsilon > 0$. Since $\xi(z) \in \rho(g_n)\bx$ for all $n$ and $\operatorname{diam}(\rho(g_n)\bx) \to 0$ by \autoref{lem:convergent_boxes} there is some $n$ with
  \[\rho(g_n) \xi(I) \subset \rho(g_n)\bx \subset B_\varepsilon(\xi(z)).\]
  If $g_n^{-1}z \in I^\circ$, then $g_n I$ is a neighborhood of $z$, so this shows continuity at $z$. On the other hand, if $g_n^{-1} z \in \partial I$ then $z$ is the clockwise boundary point of $g_n I$. So in this case we only get semicontinuity of $\xi$ at $z$ in the clockwise direction. But we can repeat the argument replacing ``clockwise'' by ``counter--clockwise'' to get full continuity.
\end{Prf}

This finishes the proof of \autoref{prop:limit_map_exists}.

\section{Nested boxes}\label{sec:nested_boxes}

The goal of this section is to find an interval $I$ in $S^1$, a finite subset $T$ of $\Gamma$, and a closed set $\bx$ of $\RP^2$ which satisfy the assumptions of \autoref{sec:limit_curve}.
In \autoref{sec:intersecting_conics} we use the reflection structure of $\Gamma$ to find that certain orbit points in $\RP^2$ lie on a common conic.
We use this to define suitable choices for $I$, $T$ and $\bx$ in \autoref{sec:definition_of_I_T_box}.
The set $\bx$ will be defined as the convex hull of certain intersection points of conics.
This allows deriving its properties from the order of points along conics.
The remainder of \autoref{sec:nested_boxes} shows that $I$, $T$, and $\bx$ satisfy all the assumptions of \autoref{sec:limit_curve}, culminating in \autoref{lem:box_inclusion_final} as the main result of this section.

Assume $p_1,p_2,p_3 \geq 3$ are odd and not all equal to $3$.
Let $\Gamma = \Gamma_{p_1,p_2,p_3}$ and $\rho \colon \Gamma \to \SL(3,\bR)$ be a representation of type $(\frac{p_1-1}{2},\frac{p_2-1}{2},\frac{p_3-1}{2})$ with parameter $t_\rho \geq \tcrit > 1$ (see \autoref{sec:prelims}).
We assume that $\rho$ is line--irreducible (see \autoref{def:irreducible}).
As noted in \autoref{sec:Coxeter_representations} every Coxeter character has such a representative.

This section will use some long words in $\rho(\Gamma)$. To simplify the notation, we will write $a \coloneqq \rho(s_1)$, $b \coloneqq \rho(s_2)$ and $c \coloneqq \rho(s_3)$ for the remainder of \autoref{sec:nested_boxes}.

\subsection{Intersecting conics}\label{sec:intersecting_conics}

Recall from \eqref{eq:pointlabels} in \autoref{sec:hyperbolic} that there are points $z_0, \dots z_{2p_3-1} \in S^1$, in order along $S^1$, so that $z_0$ is the attracting fixed point of $s_1s_2s_3$ and $s_1z_i = z_{3-i}$ as well as $s_2 z_i = z_{5-i}$ for every $i$.
We treat these indices as elements of $\bZ/2p_3\bZ$ and will sometimes write e.g. $z_{-1}$ instead of $z_{2p_3-1}$.

To each of the $z_i \in S^1$ we define a corresponding point $w_i$ and a line $\ell_i$ in $\RP^2$.
Here $w_0$ shall be the attracting fixed point of $abc = \rho(s_1s_2s_3)$, and $\ell_0$ the attracting fixed line of $abc$.
Then the points in the orbit of $w_0$ and $\ell_0$ by the dihedral group $\langle a,b \rangle$ will be labeled so that, analogously to the $z_i$,
\[aw_i = w_{3-i}, \quad bw_i = w_{5-i}, \quad a\ell_i = \ell_{3-i}, \quad b\ell_i = \ell_{5-i}, \quad \forall i \in \{0,\dots,2p_3-1\}.\]
Let $C \subset \RP^2$ be the unique conic which passes through all the points $w_i$ (see \autoref{sec:conics}).
It is clearly invariant by $a$ and $b$.
The complement of $C$ has two connected components, a disk and a Möbius strip, which we call $M$.
Each of the lines $\ell_i$ intersects $C$ in two points, one of which is $w_i$.
The other point will be called $u_i$ (if $\ell_i$ and $C$ intersected only in $w_i$ we would set $u_i = w_i$, but \autoref{lem:order} below shows this doesn't happen).

\begin{center}
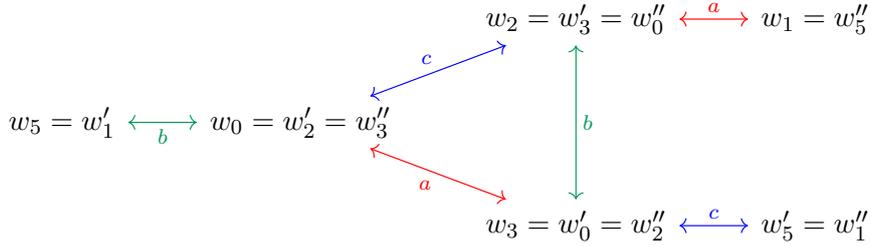

  \begin{tikzcd}
    &&&&& \\
    &&& w_2 = w_3' = w_0'' \arrow[r, leftrightarrow, red, "a"] \arrow[dd, leftrightarrow, ForestGreen, "b"] & w_1 = w_5'' & \\
    & w_5 = w_1' & w_{0} = w_2' = w_3'' \arrow[l, leftrightarrow, ForestGreen, "b"] \arrow[ru, leftrightarrow, blue, "c"] &&& \\
    &&& w_3 = w_0' = w_2'' \arrow[lu, leftrightarrow, red, "a"] \arrow[r, leftrightarrow, blue, "c"] & w_5' = w_1'' &  \\
    &&&&&
  \end{tikzcd}
  \captionof{figure}{Coincidences of the points $w_i,w_i',w_i''$ and how the generators map them to each other. The same relations hold for $\ell_i,\ell_i',\ell_i''$ and $z_i,z_i',z_i''$ (with $s_1,s_2,s_3$ in place of $a,b,c$). It follows from the discussion below that this list of coincidences is complete.}
  \label{fig:conic_mappings}
\end{center}

We adopt the following convention:
for any object $O$ defined using the generators $s_1,s_2,s_3$, its ``primed'' version $O'$ shall have the same definition, except that the generators are cyclically permuted, i.e. $s_2,s_3,s_1$ are used in place of $s_1,s_2,s_3$.
For example, $w_0$ was defined as the attracting fixed point of $abc$, so $w_0'$ is the attracting fixed point of $bca$, and therefore equal to $w_3 = a w_0$.
Cyclically permuting another time, $w_0''$ is the attracting fixed point of $cab$, hence equal to $w_2 = baw_0$.

It is easy to find more coincidences like this between the $w_i$, $w_i'$, and $w_i''$; \autoref{fig:conic_mappings} shows the complete list, which we will use throughout \autoref{sec:nested_boxes}.
Analogous coincidences also hold for the $z_i$ and $\ell_i$, e.g. $z_0' = (s_2s_3s_1)_+ = s_1(s_1s_2s_3)_+ = z_3$.
But this principle does not extend to the points $u_i$:
for example, $u_0'$ is on the intersection of the line $\ell_0' = \ell_3$ with the conic $C'$, while $u_3$ is on the intersection of the same line with $C$.

Every lemma we prove in this section also has a primed and a double--primed version:
their statements are the same, just with every object $O$ replaced by $O'$ or $O''$ and $s_1,s_2,s_3,p_1,p_2,p_3$ replaced by $s_2,s_3,s_1,p_2,p_3,p_1$ or $s_3,s_1,s_2,p_3,p_1,p_2$, respectively.
We only state and prove one version, but will make use of the others as needed.

Now we study how the conics $C$, $C'$ and $C''$ intersect.
First note that they are distinct:
if $C = C'$ or $C = C''$ then $C$ would be preserved by all of $\rho(\Gamma)$.
So $\rho(\Gamma)$ would preserve a symmetric bilinear form of signature $(2,1)$, hence be contained in a conjugate of $\SO(2,1)$. This would imply $\tr \rho(\gamma) = \tr \rho(\gamma^{-1})$ for all $\gamma \in \Gamma$ and therefore $t_\rho = 1$.

Distinct conics can intersect in at most four points (see \autoref{sec:conics}).
So provided that the points $w_i$ are distinct (which is clear if $t_\rho = \tred$ and proved by \autoref{lem:wi_distinct} in general), the identities in \autoref{fig:conic_mappings} show that
\[C \cap C' = \{w_0,w_2,w_3,w_5\}, \qquad C \cap C'' = \{w_0,w_1,w_2,w_3\}.\]
  In fact, the configuration looks like \autoref{fig:conic_intersections} (Left). We first show this in the case $t_\rho = \tred$ (i.e. if $\rho$ is reducible), and then deform to the general case $t_\rho \geq \tcrit$.
  
  \begin{Lem}\label{lem:reducible_order}
    Assume that $t_\rho = \tred > 1$, so $\rho$ is reducible, but line--irreducible.
    Then the $4p_3$ points $w_i$ and $u_i$ are in the cyclic order
    \[\dots, w_0,u_1,u_2,w_3,w_4,u_5,u_6,w_7,\dots, w_{-2}, u_{-1}, u_0,w_1,w_2,u_3,u_4, \dots\]
    along $C$. Also, we have $w_1 \in M'$.
  \end{Lem}

  \begin{Prf}
    The representation $\rho$ fixes a point $x \in \RP^2$, and there is a continuous map from $S^1$ to the space of lines through $x$, which maps the attracting fixed point of any infinite order $\gamma \in \Gamma$ to the attracting line of $\rho(\gamma)$, so in particular $z_i$ to $\ell_i$.
    Hence the lines $\ell_i$ all pass through the point $x$ and are ordered $\ell_0,\ell_1,\ell_2,\dots$.
    Along the conic $C$, therefore, $w_0$ is followed by either $w_1$ or $u_1$.
    The next point along $C$ after that is either $w_2$ or $u_2$.
    It must actually be $u_2$, as $w_2 = (ba)w_0$ is at an ``angle'' of $\pi - \pi / p_3$ from $w_0$ (more precisely, at parameter $\pi - \pi/p_3$ of the 1--parameter subgroup of $\SL(3,\bR)$ containing $ba$, parametrized with period $2\pi$).
    Applying the same argument to the odd indices, we find that the first four points along $C$ are either
    \[(a) \ w_0,u_1,u_2,w_3 \quad \text{or} \quad (b) \ w_0,w_1,u_2,u_3\]
    After that the same pattern repeats with indices increased by $4$, since $w_{i+4} = (ba)^2w_i$ and $(ba)^2$ rotates $C$ by an angle of $2\pi/p_3$.

    In case $(a)$, the points $w_0,u_2,w_3,u_0,w_2,w_5$ are in this order along $C$, while in case $(b)$ their order is $w_0,u_2,w_5,u_0,w_2,w_3$.
    Here the points $w_0,w_2,w_3,w_5$ are exactly the intersection points of $C$ and $C'$.
    Since the conics intersect transversely, this means that in either case exactly one of the points $u_0$ and $u_2$ is in $M'$.

    Now consider the line $\ell_0 = \ell_2'$, which is fixed by the Coxeter element $abc$.
    It contains the points $w_0 = w_2'$ and $x$, both of which are fixed by $abc$, as well as $u_0$, $cu_2 = cbau_0$ and $u_2'$.
    By \autoref{lem:reducible_identify} the eigenvalues of $abc$ are $-\lambda, -1, \lambda^{-1}$, with $\lambda > 1$, hence the action of $abc$ on $\ell_0$ is orientation--preserving.
    So $u_0$ and $cu_2$ lie in the same component of $\ell_0 \setminus \{w_0,x\}$.
    As $w_0$ is the attracting fixed point of $abc$, the points are in the order $w_0,u_0,cu_2,x$ on $\ell_0$.
    As $x$ is the unique fixed point of the rotations $ab$ and $bc$, it is not contained in $M$ or $M'$. Also $x \not\in cM$ since $cx = x \not\in M$.
    
  \begin{center}
    \begin{tikzpicture}[black]
      \begin{scope}
        \draw[thick] (0,0) -- (6,0);
        \fill (1,0) circle (0.07);
        \fill (2,0) circle (0.07);
        \fill (3,0) circle (0.07);
        \fill (4,0) circle (0.07);
        \fill (5,0) circle (0.07);
        \node at (1,-0.4) {$w_0$};
        \node at (2,-0.4) {$u_0$};
        \node at (3,-0.4) {$u_2'$};
        \node at (4,-0.4) {$cu_2$};
        \node at (5,-0.4) {$x$};

        \draw[|-|,thick,red] (1,0.2) -- (2,0.2);
        \draw[|-|,thick,ForestGreen] (1,0.4) -- (3,0.4);
        \draw[|-|,thick,orange] (1,0.6) -- (4,0.6);

        \node[red] at (1.5,0.9) {$M$};
        \node[ForestGreen] at (2.5,0.9) {$M'$};
        \node[orange] at (3.5,0.9) {$cM$};
        \node at (-0.5,0) {$\ell_0$};
      \end{scope}

      \begin{scope}[xshift=-4cm]
        \node at (0,-1.2) {$w_0$};
        \node at (0.1,-0.1) {$w_3$};
        \node at (0.6,1.2) {$w_4$};
        \node at (-0.2,3) {$w_7$};
        \node at (-1.3,3) {$w_8$};
        \node at (-1.7,0.8) {$w_1$};
        \node at (-2.1,0.5) {$w_6''$};
        \node at (-1.95,-0.4) {$w_9''$};
        \node at (-2.4,-1) {$w_4'$};
        \node at (-3.5,-2.6) {$w_7'$};
        \node at (-3,-3.3) {$w_8'$};
        \node at (-0.7,-2.5) {$w_5$};
        \node at (-0.2,-2.65) {$w_6$};
        \node at (0.4,-2.1) {$w_9$};
        \node at (1.1,-2.25) {$w_4''$};
        \node at (3.1,-2.4) {$w_7''$};
        \node at (3.4,-1.5) {$w_8''$};
        \node at (1.8,-0.2) {$w_5'$};
        \node at (1.7,0.3) {$w_6'$};
        \node at (0.8,0.65) {$w_9'$};
        \node at (-0.9,-0.6) {$w_2$};

        \node[red] at (-1,2.3) {$C$};
        \node[ForestGreen] at (-2.5,-1.9) {$C'$};
        \node[blue] at (2.5,-1.3) {$C''$};

        \begin{scope}[shift={(-0.24,-0.62)}]

          \begin{scope}[rotate=4.5]
            \draw[red,thick] (-0.22,0.93) ellipse (0.86cm and 2.76cm);

            \fill (0.64,0.83) circle (0.07);
            \fill (0.6,1.78) circle (0.07);
            \fill (-0.03,3.63) circle (0.07);
            \fill (-0.45,3.6) circle (0.07);
            \fill (-1.07,1.29) circle (0.07);
            \fill (-0.58,-1.57) circle (0.07);
            \fill (-0.335,-1.8) circle (0.07);
            \fill (0.23,-1.43) circle (0.07);
          \end{scope}
          \begin{scope}[rotate=124.5]
            \draw[ForestGreen,thick] (-0.22,0.93) ellipse (0.86cm and 2.76cm);

            \fill (0.64,0.83) circle (0.07);
            \fill (0.6,1.78) circle (0.07);
            \fill (-0.03,3.63) circle (0.07);
            \fill (-0.45,3.6) circle (0.07);
            \fill (-1.07,1.29) circle (0.07);
            \fill (-0.58,-1.57) circle (0.07);
            \fill (-0.335,-1.8) circle (0.07);
            \fill (0.23,-1.43) circle (0.07);

          \end{scope}
          \begin{scope}[rotate=244.5]
            \draw[blue,thick] (-0.22,0.93) ellipse (0.86cm and 2.76cm);

            \fill (0.64,0.83) circle (0.07);
            \fill (0.6,1.78) circle (0.07);
            \fill (-0.03,3.63) circle (0.07);
            \fill (-0.45,3.6) circle (0.07);
            \fill (-1.07,1.29) circle (0.07);
            \fill (-0.58,-1.57) circle (0.07);
            \fill (-0.335,-1.8) circle (0.07);
            \fill (0.23,-1.43) circle (0.07);
          \end{scope}
        \end{scope}
      \end{scope}
    \end{tikzpicture}
    \reallynopagebreak
    
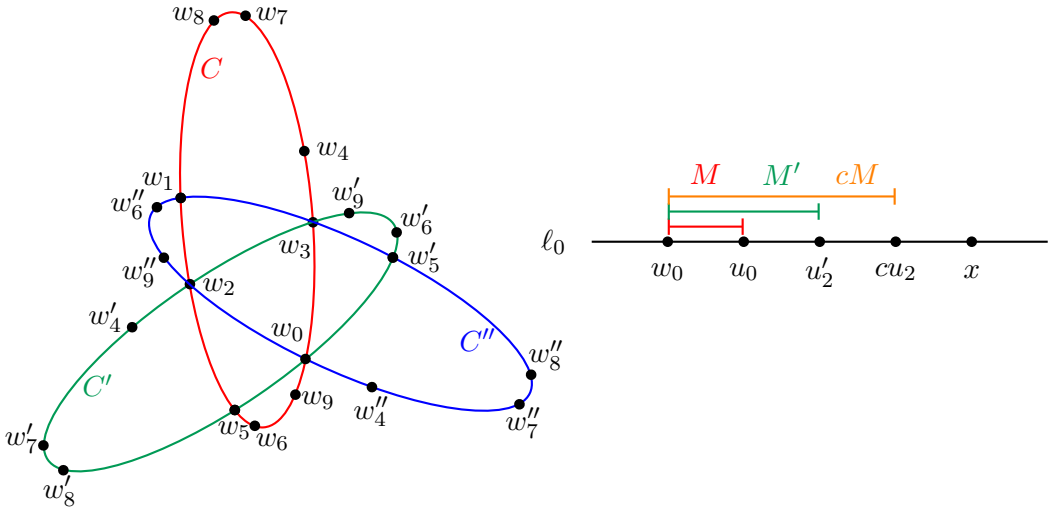
\captionof{figure}{(Left) The configuration of conics $C,C',C''$ and the order of the points $w_i,w_i',w_i''$ on them according to \autoref{lem:reducible_order}, for the case $(p_1,p_2,p_3) = (5,5,5)$. (Right) The order of the points $w_0,u_0,u_2',cu_2,x$ on $\ell_0$ as in the proof of \autoref{lem:reducible_order}.}
    \label{fig:conic_intersections}
  \end{center}

    So the intersection $\ell_0 \cap M$ is the interval between $w_0$ and $u_0$ not containing $x$, $\ell_0 \cap cM$ is the interval from $w_0 = cw_2$ to $cu_2$, and $\ell_0 \cap M'$ is the interval between $w_0 = w_2'$ and $u_2'$.
    As $M'$ is invariant by $c$, we have that exactly one of the points $u_0$ and $cu_2$ is in $M'$.
    So $u_2'$ must lie between $u_0$ and $cu_2$.
    Therefore, we have $u_0 \in M'$, $cu_2 \not\in M'$, $u_2' \not\in M$, and $u_2' \in cM$ (see \autoref{fig:conic_intersections} right).

    The same reasoning applies to the ``double--primed'' situation, giving $u_2 \not\in M''$ and $u_2 \in bM''$, which is equivalent to $u_3 = bu_2 \in M''$.
    But this is a contradiction to order $(b)$:
    if the points $u_2$ and $u_3$ have none of the points $w_i$ between them, they must both be either in $M''$ or not.
    So we have order $(a)$. In this order, the points $u_0$ and $w_1$ are neighbors, so $u_0 \in M'$ implies $w_1 \in M'$.
  \end{Prf}

  Now we consider a general representation $\rho$ with $t_\rho \geq \tcrit$.

  \begin{Lem}\label{lem:wi_distinct}
    The points $w_0,\dots,w_{2p_3-1}$ are pairwise distinct.
  \end{Lem}

  \begin{Prf}
    We have seen that in the reducible case, the $w_i$ are distinct and in the order $w_0,w_3,w_4,w_7,w_8,\dots$ along $C$.
    The same holds for the $w_i'$ and $w_i''$.
    For contradiction, we take a path of representations starting from a reducible one and follow it until any pair of the $w_i$ or $w_i'$ or $w_i''$ coincides for the first time.
    Without loss of generality, we can assume this collision happens among the $w_i$.
    As $\langle a,b \rangle$ acts transitively on the $w_i$, $w_3$ has to equal another $w_i$. 
    By continuity, this has to be $w_0 = w_3$ or $w_3 = w_4$.

  Note that $cb$ rotates $C'$ by the fixed angle $\frac{p_1-1}{p_1}\pi$ and maps $w_i'$ to $w_{i+2}'$.
  So $w_i' \neq w_j'$ whenever $i-j$ is even and not a multiple of $2p_1$.
  In particular $w_0 = w_2' \neq w_0' = w_3$.
  So we have $w_3 = w_4$, and therefore $w_0'' = w_2 = bw_3 = bw_4 = w_1 = w_5''$.
  The four points $w_0'', w_4'', w_5'', w_6''$ have been in this cyclic order for every representation on the path, so they degenerate to either $w_0'' = w_4'' = w_5''$ or $w_5'' = w_6'' = w_0''$.
  If $p_2 \neq 3$ this is a contradictoin by the same argument as above, only using the rotation $ac$ instead of $cb$.

  On the other hand, if $p_2=3$, then $w_0' = w_2'' = ac w_0'' = ac w_5'' = w_1'' = w_5'$, and we can repeat the same argument with $w_0'$ and $w_5'$ instead of $w_0''$ and $w_5''$.
  Again, this leads to a contradiction unless $p_1 = 3$ and $w_0 = w_2' = cbw_0' = cb w_5' = w_1' = w_5$.
  But $p_1,p_2,p_3$ cannot all be $3$, so repeating the argument another time gives the contradiction we want.
\end{Prf}

  \begin{center}
    \begin{tikzpicture}
      \begin{scope}
        \fill[black!20] (-3.3,0) -- (6.3,0) -- (6.3,1) -- (-3.3,1) -- cycle;
        \draw[very thick,black] (-3.3,0) -- (6.3,0);
        \draw[very thick,black] (-3.3,1) -- (6.3,1);
        \foreach \i in {-1,0,3,4} \fill (\i,0) circle (0.07);
        \foreach \i in {-1,0,3,4} \node[anchor=north] at (\i,0) {$w_{\i}$};
        \foreach \i in {-1,0,3,4} \node[anchor=south,gray] at (\i,1) {$u_{\i}$};
        \foreach \i in {-3,-2,1,2,5,6} \fill (\i,1) circle (0.07);
        \foreach \i in {-3,-2,1,2,5,6} \node[anchor=south] at (\i,1) {$w_{\i}$};
        \foreach \i in {-3,-2,1,2,5,6} \node[anchor=north,gray] at (\i,0) {$u_{\i}$};
        \foreach \i in {-3,-2,-1,0,1,2,3,4,5,6} \draw (\i,0) -- (\i,1);
      \end{scope}

      \begin{scope}[yshift=-2cm]
        \fill[black!20] (-3.3,0) -- (6.3,0) -- (6.3,1) -- (-3.3,1) -- cycle;
        \draw[very thick,black] (-3.3,0) -- (6.3,0);
        \draw[very thick,black] (-3.3,1) -- (6.3,1);
        \foreach \i in {-1,0,3,4} \fill (\i,0) circle (0.07);
        \foreach \i in {-1,0,3,4} \node[anchor=north] at (\i,0) {$w_{\i}$};
        \foreach \i in {-1,3} \node[anchor=south,gray] at (\i+0.8,1) {$u_{\i}$};
        \foreach \i in {0,4} \node[anchor=south,gray] at (\i-0.8,1) {$u_{\i}$};
        \foreach \i in {-3,-2,1,2,5,6} \fill (\i,1) circle (0.07);
        \foreach \i in {-3,-2,1,2,5,6} \node[anchor=south] at (\i,1) {$w_{\i}$};
        \foreach \i in {-3,1,5} \node[anchor=north,gray] at (\i+0.8,0) {$u_{\i}$};
        \foreach \i in {-2,2,6} \node[anchor=north,gray] at (\i-0.8,0) {$u_{\i}$};       
        \foreach \i in {-1,3} \draw (\i,0) -- (\i+0.75,1);
        \foreach \i in {0,4} \draw (\i,0) -- (\i-0.75,1);
        \foreach \i in {-3,1,5} \draw (\i,1) -- (\i+0.75,0);
        \foreach \i in {-2,2,6} \draw (\i,1) -- (\i-0.75,0);
      \end{scope}
    \end{tikzpicture}
    
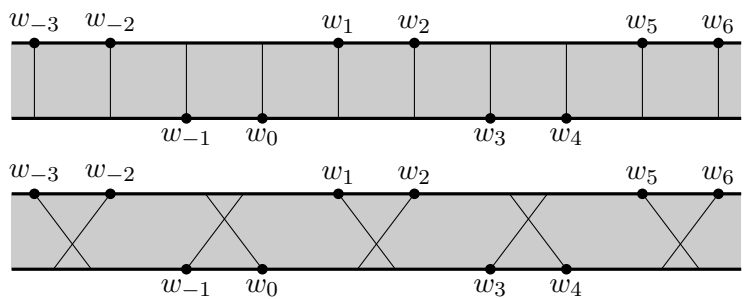
\captionof{figure}{The two possible orders of points along $C$. The shaded area represents the Möbius strip $M$, bounded by $C$, and the thin vertical lines are the $\ell_i$. The lines $\ell_i$ and $\ell_j$ can only intersect in $M$ if $\{i,j\} = \{2k-1,2k\}$ for some $k$.}
    \label{fig:conic_orders}
  \end{center}

\begin{Lem}\label{lem:order}
  The cyclic order of the points $u_i$ and $w_i$ along $C$ is either
  \[\dots, w_0,u_1,u_2,w_3,w_4,u_5,u_6,w_7,\dots, w_{-2}, u_{-1}, u_0,w_1,w_2,u_3,u_4, \dots\]
  as in the reducible case, or differs from it only by switching the order of $u_{2k-1}$ and $u_{2k}$ for all $k$ (or $u_{2k-1} = u_{2k}$ for all $k$), see \autoref{fig:conic_orders}.
  Furthermore, we have $w_1 \in M'$.
\end{Lem}

\begin{Prf}
  By \autoref{lem:wi_distinct} the order of the $w_i$ must be the same as in the reducible case.
  Following a path of representations starting from the reducible one as in the proof of \autoref{lem:wi_distinct} we see that it is enough to show $u_2 \neq w_0$ and $u_2 \neq w_3$.

  If $w_3 = u_2$, then the line $\ell_2$ contains both $w_2$ and $w_3 = bw_2$, so it is preserved by $b$.
  Since $cab$ also fixes $\ell_2$, so does $ca$. As $ca$ has finite order greater than 2, it has a unique fixed line, which is fixed by $a$ and $c$ individually. So we showed that $\ell_2$ is fixed by $a$, $b$ and $c$, contradicting our assumption that $\rho$ has no fixed line. The case $u_2 = w_0$ is similar.

  We have $w_1 \in M'$ when $t_\rho = \tred$, and this cannot change under continuous deformation, since $w_1$ is never in $C \cap C'$.
\end{Prf}

  Assuming $i \neq j$, we can read off from the order of $w_i,u_i,w_j,u_j$ on $C$ whether the intersection point of $\ell_i$ and $\ell_j$ is in the Möbius strip $M$ or in the disk bounded by $C$.
  If it is in $M$, we say \emph{$\ell_i$ and $\ell_j$ cross in $M$.}
  \autoref{lem:order} shows that $\ell_i$ and $\ell_j$ can only cross in $M$ if $\{i,j\} = \{2k-1,2k\}$ for some $k$; see \autoref{fig:conic_orders}.

  \begin{Lem}\label{lem:u0inM}
    $u_0'' \in M$ and $cu_0 \in M$.
  \end{Lem}

  \begin{Prf}
    The points $w_0,u_2,w_3,u_0,w_1,w_2,w_5$ lie in this cyclic order on $C$.
    Since $w_1 \in M'$ by \autoref{lem:order} and $C \cap C' = \{w_0,w_3,w_2,w_5\}$, we have that $u_0 \in M'$ and $u_2 \not\in M'$.
    Similarly, considering $w_0',w_3',w_1',w_2',u_3',w_5'$ along $C'$ and using that $w_5' = w_1'' \in M$ by \autoref{lem:order}, we find that $u_3' \in M$.

    Now consider the line $\ell_2 = \ell_3'$ which contains the points $w_2 = w_3'$, $u_3'$, and $u_2$.
    Since $u_3' \in M$ and $u_2 \not\in M'$ we see that $\ell_2 \cap M' \subset \ell_2 \cap M$.
    In particular $cu_0 \in \ell_2 \cap M' \subset M$.
    The other statement $u_0'' \in M$ is just the ``double primed'' version of $u_0 \in M'$.
  \end{Prf}

\subsection{Definition of \texorpdfstring{$I$, $T$}{I, T}, and \texorpdfstring{$\bx$}{the box}}\label{sec:definition_of_I_T_box}

In order to apply \autoref{prop:limit_map_exists} and get a boundary map, we first choose an interval $I \subset S^1$ and a finite set $T \subset \Gamma$ satisfying the axioms in the beginning of \autoref{sec:limit_curve}.
Then we construct a closed set $\bx \subset \RP^2$ satisfying the assumptions needed for \autoref{prop:limit_map_exists}, in particular that $\rho(\gamma)\bx \subset \bx$ for all $\gamma \in T$.

Let $I = [z_3,z_0]$, that is the component of $S^1 \setminus \{z_0,z_3\}$ which does not contain the points $z_1,z_2$ (see \autoref{fig:hyperbolic_intervals}).
Next, we define $T$ by
\[Q = \{s_1^\delta(s_1s_2)^j \mid \delta \in \{0,1\}, 1 \leq j \leq \textstyle\frac{p_3-1}{2}\}, \quad T = QQ''Q'.\]
Since $(s_1s_2)^jI'' = (s_1s_2)^j[z_3'',z_0''] = (s_1s_2)^j[z_0,z_2] = [z_{-2j},z_{2-2j}]$ (see \autoref{fig:conic_mappings}), we have
\begin{equation*}
  \label{eq:Iunion}
  \bigcup_{\gamma \in Q}\gamma I'' = \bigcup_{j=1}^{\frac{p_3-1}{2}}[z_{-2j},z_{2-2j}] \cup \bigcup_{j=1}^{\frac{p_3-1}{2}}[z_{1+2j},z_{3+2j}] = [z_{p+1},z_0] \cup [z_3,z_{p+2}] = I.
\end{equation*}
Together with the analogous versions $\bigcup_{\gamma' \in Q'}\gamma'I = I'$ and $\bigcup_{\gamma'' \in Q''}\gamma''I' = I''$, this shows that $\bigcup_{\gamma \in T}\gamma I = I$, so $I$ satisfies assumptions (i) and (iii) of \autoref{sec:limit_curve}.
The only $\gamma \in T$ which fixes an endpoint of $I$ is $\gamma = (s_1s_2)(s_3s_1)(s_2s_3) = (s_1s_2s_3)^2$, the attracting fixed point of which is the endpoint $z_0$.
So assumption (ii) also holds.

It remains to find a definition of $\bx \subset \RP^2$ (and analogously $\bxp$ and $\bxpp$) which satisfies
\begin{equation}\label{eq:box_inclusion}
  \rho(\gamma)\bxpp \subset \bx \quad \forall \gamma \in Q,
\end{equation}
and therefore $\rho(\gamma\gamma''\gamma')\bx \subset \rho(\gamma\gamma'')\bxp \subset \rho(\gamma)\bxpp \subset \bx$ for all $\gamma\gamma''\gamma' \in T$.

\begin{center}
  \begin{tikzpicture}
    \fill[ForestGreen!30] (-108:3.42) arc (-108:12:3.42) -- (12:3.9) arc (12:-108:3.9) -- cycle;
    \fill[blue!30] (12:3.42) arc (12:132:3.42) -- (132:3.9) arc (132:12:3.9) -- cycle;
    \fill[purple!30] (132:3.42) arc (132:252:3.42) -- (252:3.9) arc (252:132:3.9) -- cycle;
    \draw[thick] (12:3.42) -- (12:3.9);
    \draw[thick] (132:3.42) -- (132:3.9);
    \draw[thick] (252:3.42) -- (252:3.9);

    \begin{scope}
      \fill[blue!60] (252:3.91) arc (252:307:3.91) -- (307:3.99) arc (307:252:3.99) -- cycle;
      \fill[blue!60] (307:4.01) arc (307:331.7:4.01) -- (331.7:4.09) arc (331.7:307:4.09) -- cycle;
      \fill[blue!60] (316.7:4.11) arc (316.7:339.5:4.11) -- (339.5:4.19) arc (339.5:316.7:4.19) -- cycle;
      \fill[blue!60] (339.5:4.21) arc (339.5:372:4.21) -- (372:4.29) arc (372:339.5:4.29) -- cycle;
    \end{scope}
    \begin{scope}[rotate=120]
      \fill[red!60] (252:3.91) arc (252:307:3.91) -- (307:3.99) arc (307:252:3.99) -- cycle;
      \fill[red!60] (307:4.01) arc (307:331.7:4.01) -- (331.7:4.09) arc (331.7:307:4.09) -- cycle;
      \fill[red!60] (316.7:4.11) arc (316.7:339.5:4.11) -- (339.5:4.19) arc (339.5:316.7:4.19) -- cycle;
      \fill[red!60] (339.5:4.21) arc (339.5:372:4.21) -- (372:4.29) arc (372:339.5:4.29) -- cycle;
    \end{scope}
    \begin{scope}[rotate=240]
      \fill[ForestGreen!60] (252:3.91) arc (252:307:3.91) -- (307:3.99) arc (307:252:3.99) -- cycle;
      \fill[ForestGreen!60] (307:4.01) arc (307:331.7:4.01) -- (331.7:4.09) arc (331.7:307:4.09) -- cycle;
      \fill[ForestGreen!60] (316.7:4.11) arc (316.7:339.5:4.11) -- (339.5:4.19) arc (339.5:316.7:4.19) -- cycle;
      \fill[ForestGreen!60] (339.5:4.21) arc (339.5:372:4.21) -- (372:4.29) arc (372:339.5:4.29) -- cycle;
    \end{scope}

    \node[blue] at (290:4.2) {$s_2I''$};
    \node[blue] at (318:4.6) {$s_2s_1s_2I''$};
    \node[blue] at (335:5.0) {$s_1s_2s_1s_2I''$};
    \node[blue] at (355:4.85) {$s_1s_2I''$};
    \node[red] at (30:4.4) {$s_1I'$};
    \node[red] at (65:4.55) {$s_1s_3s_1I'$};
    \node[red] at (85:4.5) {$s_3s_1s_3s_1I'$};
    \node[red] at (115:4.65) {$s_3s_1I'$};
    \node[ForestGreen] at (160:4.4) {$s_3I$};
    \node[ForestGreen] at (190:4.9) {$s_3s_2s_3I$};
    \node[ForestGreen] at (210:5.05) {$s_2s_3s_2s_3I$};
    \node[ForestGreen] at (230:4.8) {$s_2s_3I$};

    \node at (0,0) {\includegraphics[width=7cm]{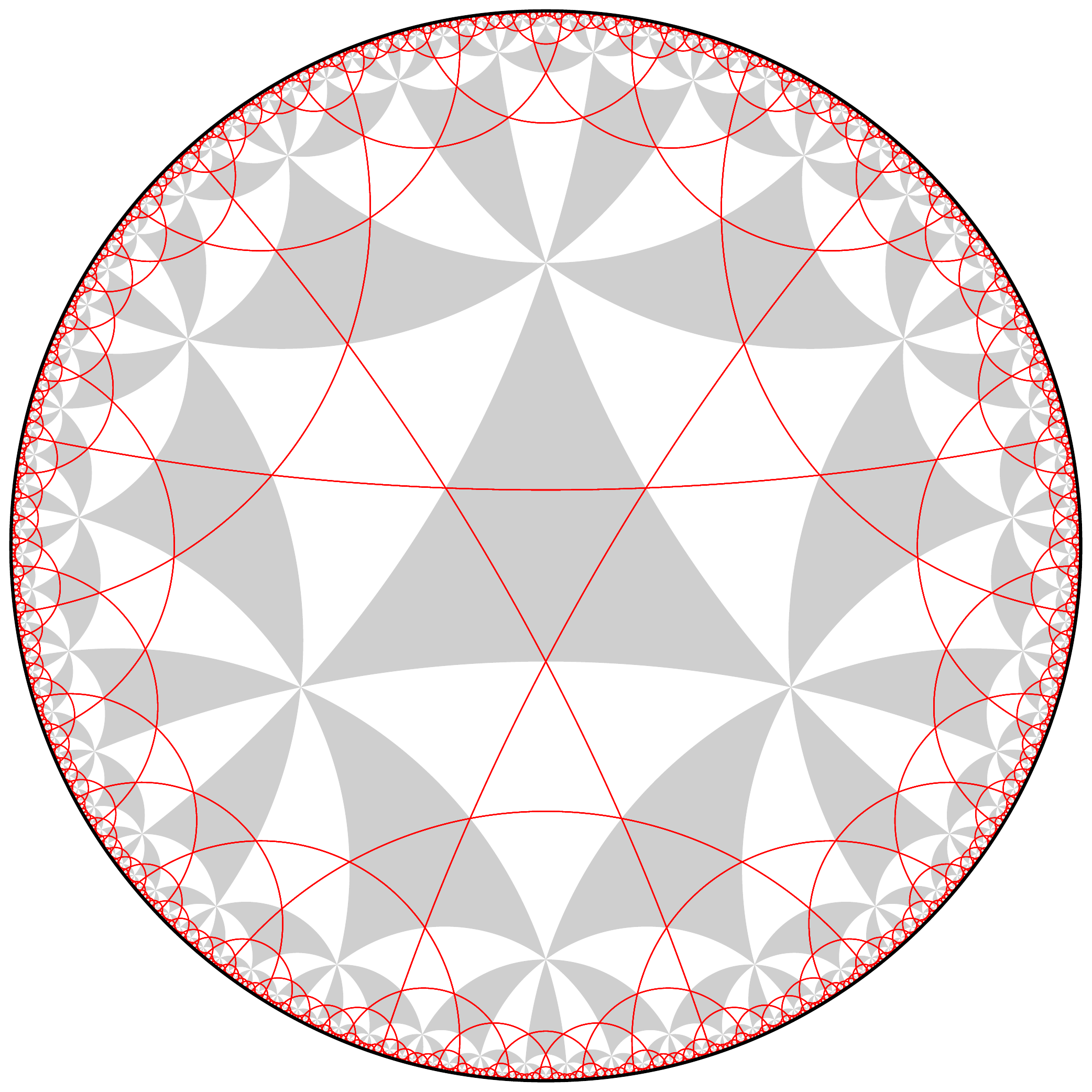}};
    \node at (0:1.1) {$s_1$};
    \node at (120:1.1) {$s_3$};
    \node at (240:1.1) {$s_2$};
    \node at (15:3.7) {$z_0$};
    \node at (67.5:3.7) {$z_1$};
    \node at (135:3.7) {$z_2$};
    \node at (255:3.7) {$z_3$};
    \node at (279.5:3.7) {$z_4$};
    \node at (307.3:3.7) {$z_5$};
    \node at (316.7:3.7) {$z_6$};
    \node at (331.7:3.7) {$z_7$};
    \node at (339.5:3.7) {$z_8$};
    \node at (358:3.7) {$z_9$};
    \node at (42:4.2) {$(s_2s_3s_1)_-$};
    \node at (168.1:4.3) {$(s_1s_2s_3)_-$};

    \foreach \i in {12,67.5,132,252,279.5,307.3,316.7,331.7,339.5,358} \fill (\i:3.43) circle (0.08);

    \fill[black] (48.1:3.43) circle (0.08);
    \fill[black] (168.1:3.43) circle (0.08);

    \node[blue] at (60:3.7) {$I''$};
    \node[purple] at (180:3.7) {$I'$};
    \node[ForestGreen] at (300:3.7) {$I$};
  \end{tikzpicture}
  \captionof{figure}{The intervals $I,I',I''$ and the first level of subintervals. Here $p_1 = p_2 = p_3 = 5$, hence $Q = \{s_2,s_1s_2,s_2s_1s_2,s_1s_2s_1s_2\}$, and $Q'$ and $Q''$ are analogous. The points $(s_1s_2s_3)_-$ and $(s_2s_3s_1)_-$, used in \autoref{lem:transversality_Gamma_orbit}, are also shown.}
  \label{fig:hyperbolic_intervals}
\end{center}

The set $\bx$ is supposed to behave like the interval $I$, so it should be bounded by $\ell_0$ and $\ell_3$ in the direction ``along'' the limit curve.
In the transverse direction, there is no such obvious choice.
A simple idea would be to ``roughly'' bound it by $C$, e.g. to define $\bx$ as convex hull of the points $w_0,w_3,u_0,u_3$.
Unfortunately, this box does not quite satisfy \eqref{eq:box_inclusion}.
We can fix this by cutting off two of its corners.
This yields the convex hexagon we consider below.

  Defining $\bx$ as a convex hull requires some care, as the convex hull is only well--defined in a fixed affine chart.
  For a projective line $\ell$ and points $x_1, \dots, x_n \in \RP^2$ not on $\ell$, we write
  \[\CH{\ell}{x_1,\dots,x_n}\]
  for the convex hull of $x_1,\dots,x_n$ in the affine chart $\RP^2 \setminus \ell$.
  Sometimes we have to switch from one affine chart to another. The following lemma will be useful for this.

  \begin{Lem}\label{lem:ch_chart_change}
    Consider a collection of points $x_1,\dots,x_n \in \overline{M}$ so that $x_i \in \ell_{j_i}$ for some $j_i$.
    Let $\ell_k$ and $\ell_m$ be two other lines which do not cross $\ell_{j_i}$ in $\overline M$ for any $i$.
    Assume there are two different connected components $I_1$ and $I_2$ of $C \setminus \{w_{j_1},u_{j_1},\dots,w_{j_n},u_{j_n}\}$ such that either $\{w_k,w_m\} \subset I_1$ and $\{u_k,u_m\} \subset I_2$, or $\{w_k,u_m\} \subset I_1$ and $\{u_k,w_m\} \subset I_2$.
    Then
    \[\CH{\ell_k}{x_1,\dots,x_n} = \CH{\ell_m}{x_1,\dots,x_n}.\]
  \end{Lem}

  \begin{Prf}
    Assume for simplicity that $\{w_k, w_m\} \subset I_1$ and $\{u_k,u_m\} \subset I_2$, otherwise exchange the roles of $w_m$ and $u_m$.
    We want to find a continuous path of lines $\ell(t)$ from $\ell_k$ to $\ell_m$ which avoids $x_1,\dots,x_n$.
    Let $w(t)$ and $u(t)$ be continuous paths connecting $w_k$ with $w_m$ within $I_1$ and $u_k$ with $u_m$ within $I_2$.
    The cyclic order of $w_{j_i},w(t),u_{j_i},u(t)$ is constant for every $i \in \{1,\dots,n\}$.
    For all $t$, let $\ell(t)$ be the line through $w(t)$ and $u(t)$.
    Then the intersection point $\ell_{j_i} \cap \ell(t)$ is not in $\overline{M}$, hence $\ell(t)$ cannot contain the point $x_i$.
  \end{Prf}

  Now we can define the box as (see \autoref{fig:box_inclusions})
  \[\bx = \CH{\ell_2}{w_0,w_3,w_5,w_{-2},bcu_0,abcu_0}.\]
  \autoref{lem:ch_chart_change} shows that $\ell_1$ can be used instead of $\ell_2$ in this definition:
  the 6 vertices defining $\bx$ lie on the lines $\ell_0,\ell_3,\ell_5,\ell_{-2}$ and in the order of the $w_i$ and $u_i$ on $C$ both $\{w_1,w_2\}$ and $\{u_1,u_2\}$ are neighboring pairs.
  In particular, this shows $a \bx = \bx$, as $a\ell_2 = \ell_1$ and the set of vertices is invariant.
  As always, $\bxp$ and $\bxpp$ are analogously defined by
  \[\bxp = \CH{\ell_2'}{w_0',w_3',w_5',w_{-2}',cau_0',bcau_0'}, \quad \bxpp = \CH{\ell_2''}{w_0'',w_3'',w_5'',w_{-2}'',abu_0'',cabu_0''}.\]

  \subsection{The box inclusions}

  Our goal is now to show $g\bxpp \subset \bx$ for all $g \in \rho(Q)$, essentially by proving that all vertices of $g\bxpp$ are in $\bx$.
  For most $g$ this is achieved by \autoref{lem:box_in_ch} together with \autoref{lem:wi_in_box}, with a few special cases handled separately afterwards.
  \autoref{fig:box_inclusions} shows the configuration of boxes in an example.

  Many arguments in this section rely on the following simple fact: if $x,y,z$ are three distinct points on a conic $C$, splitting $C$ into three arcs, and $\ell$ is a line intersecting the conic in two of these arcs, then the third arc is completely contained in $\CH{\ell}{x,y,z}$.

  \begin{Lem}\label{lem:wi_in_box}
    $w_i \in \bx$ for all $i \not\in \{1,2\}$, $w_i \in \bx^\circ$ for all $i \not\in \{-2,0,1,2,3,5\}$, and $u_i \in \bx^\circ$ for all $i \not\in \{-1,0,1,2,3,4\}$.
  \end{Lem}

  \begin{Prf}
    Splitting $C$ into three arcs along $w_0$, $w_3$, and $w_5$, the arc from $w_0$ to $w_3$ contains $u_2$, and the arc from $w_3$ to $w_5$ contains $w_2$.
    Hence the arc from $w_5$ to $w_0$ is contained in $\CH{\ell_2}{w_0,w_3,w_5} \subset \bx$.
    Since $\bx$ is invariant by $a$, it also contains the arc from $w_3$ to $w_{-2}$ avoiding $w_0$.
    Together, these contain all points $w_i$ with $i \not\in \{1,2\}$.
    The interiors of these arcs are even contained in $\bx^\circ$.
  \end{Prf}

  \begin{Lem}\label{lem:u0inCH}
    $u_0'', cu_0 \in \CH{\ell_3}{w_2,u_2}^\circ$.
  \end{Lem}

  \begin{Prf}
    Both $u_0''$ and $cu_0$ lie on the line $\ell_2$ (since $\ell_2 = \ell_0'' = c\ell_0$), and in $M$ by \autoref{lem:u0inM}.
    Since $\ell_3$ and $\ell_2$ do no intersect within $M$, both points are in $\CH{\ell_3}{w_2,u_2}^\circ$.
  \end{Prf}

  \begin{Lem}\label{lem:w2inCH}
    $w_{-2}'' \in \CH{\ell_3}{w_2,w_0,cu_0}^\circ$.
  \end{Lem}

  \begin{Prf}
    Since the points $w_0,w_3,u_0,w_2,u_3$ lie in this order on the conic $C$, the arc from $u_0$ to $w_2$ avoiding $w_0$, and in particular the point $w_1$, is contained in $\CH{\ell_3}{w_0,w_2,u_0}^\circ$.
    So $w_{-2}'' = cw_1$ is contained in $c\CH{\ell_3}{w_0,w_2,u_0}^\circ = \CH{\ell_5'}{w_3',w_2',cu_0}^\circ$.

    Further $\CH{\ell_5'}{w_3',w_2',cu_0} = \CH{\ell_0'}{w_3',w_2',cu_0} = \CH{\ell_3}{w_2,w_0,cu_0}$, since $cu_0 \in \ell_3'$ and $cu_0 \in M'$ (by \autoref{lem:u0inM}) and the pairs $w_0',w_5'$ as well as $u_0',u_5'$ each lie on common arcs of $C' \setminus \{w_2',w_3',u_2',u_3'\}$.
  \end{Prf}

  \begin{Lem}\label{lem:cabu0inCH}
    $cabu_0'' \in \CH{\ell_3}{w_2,cu_0}^\circ$.
  \end{Lem}

  \begin{Prf}
    Since $u_0'' \in \CH{\ell_3}{w_2,u_2}^\circ$ by \autoref{lem:u0inCH}, we have (using \autoref{lem:ch_chart_change} where necessary)
    \[cabu_0'' \in c\CH{\ell_1}{w_0,u_0}^\circ = c\CH{\ell_2}{w_0,u_0}^\circ = \CH{\ell_0}{w_2,c u_0}^\circ = \CH{\ell_3}{w_2,cu_0}^\circ.\qedhere\]
  \end{Prf}

  \begin{Lem}\label{lem:box_in_ch}
    $\bxpp \subset \CH{\ell_i}{w_0,w_2,u_0,u_2}$ for any $i \not\in \{-1,0,1,2\}$. In particular these $\ell_i$ do not intersect $\bxpp$, and $\ell_0$ and $\ell_2$ only intersect it in its boundary $\partial\bxpp$.
  \end{Lem}

  \begin{Prf}
    By definition,
    \[\bxpp = \CH{\ell_2''}{w_0'',w_3'',w_5'',w_{-2}'',abu_0'',cabu_0''} = \CH{\ell_3}{w_2,w_0,w_1,w_{-2}'',abu_0'',cabu_0''}.\]
    We need to show that the last four points are in the convex hull.
    The points $w_0, w_2, u_0, u_2$ split $C$ into 4 connected components, which are alternatingly contained or not contained in $\CH{\ell_i}{w_0,w_2,u_0,u_2}$.
    Since  $w_i,u_i \not\in \CH{\ell_i}{w_0,w_2,u_0,u_2}$, $w_1$ must be contained in it, see \autoref{fig:conic_orders}.
    For $w_{-2}''$ and $cabu_0''$ it follows from \autoref{lem:w2inCH}, \autoref{lem:cabu0inCH} and \autoref{lem:u0inCH}.
    Finally, $abu_0'' \in \CH{\ell_1}{w_0,u_0} = \CH{\ell_3}{w_0,u_0}$ by \autoref{lem:u0inCH} and \autoref{lem:ch_chart_change}.
    This proves the lemma for $i = 3$.
    We can then use \autoref{lem:ch_chart_change} to change $i$ to what we want.
  \end{Prf}

  Below we will use \autoref{lem:box_in_ch} to show that $g \bxpp \subset \bx$ for all $g \in \rho(Q) \setminus \{b, ab\}$.
  The case $g \in \{b,ab\}$, and particularly the vertex $ababu_0''$, needs extra attention.
  \autoref{lem:ababu0_in_box} will show that it is in $\bx$ if $p_2 > 3$ or $p_3 > 3$, but this is false if $p_2 = p_3 = 3$.
  We will work around this issue in \autoref{sec:p2_p3_3}.

  \begin{Lem}\label{lem:ababu0_in_box}
    If $p_2 > 3$ or $p_3 > 3$ then $ababu_0'' \in \bx^\circ$.
  \end{Lem}

  \begin{Prf}
    By \autoref{lem:u0inCH} and \autoref{lem:ch_chart_change}
    \[ababu_0'' \in \CH{\ell_{-1}}{w_{-2},u_{-2}} = \CH{\ell_2}{w_{-2},u_{-2}}.\]
    If $p_3 > 3$ then $w_{-2} \in \bx$ and $u_{-2} \in \bx^\circ$ by \autoref{lem:wi_in_box}, proving the lemma.

    If however $p_3 = 3$, then $u_{-2} = u_4$, so \autoref{lem:wi_in_box} does not apply.
    In this case, consider the points $w_5'',u_2'',w_3'',w_{-2}'',w_2''$.
    They lie in this order on the conic $C''$.
    Therefore, the arc from $w_3''$ to $w_{-2}''$, which contains $u_5''$ if $p_2 > 3$, is contained in $\CH{\ell_2''}{w_3'',w_5'',w_{-2}''}^\circ$.
    Hence
    \[ababu_0'' = bu_5'' \in \CH{b\ell_2''}{bw_3'',bw_5'',bw_{-2}''}^\circ = \CH{\ell_2}{w_5,w_4,bw_{-2}''}^\circ.\]
    The points $w_4$ and $w_5$ are in $\bx$ by \autoref{lem:wi_in_box} and $bw_{-2}'' \in \CH{\ell_2}{w_3,w_5,bcu_0} \subset \bx$ by \autoref{lem:w2inCH}.
    So $ababu_0'' \in \bx^\circ$.
  \end{Prf}

  \begin{center} 
    \begin{tikzpicture}[scale=1.25] 
      \begin{scope}
        \coordinate (u8) at (0.43,-1.615);
        \coordinate (u6) at (-0.14,1.59);

        \coordinate (w8) at (1.36,2.04);
        \coordinate (w6) at (-0.4,-1.71);
        \coordinate (w3) at (-2.865,2.97);
        \coordinate (bcu0) at (-2.865,-0.34);
        \coordinate (w5) at (-1.135,-2.065);
        \coordinate (ababw-2pp) at (0.71,0.04);
        \coordinate (w7) at (0.63,1.69);
        \coordinate (ababu0pp) at (0.825,-0.055);
        \coordinate (ababcabu0pp) at (1.05,0.8);
        \coordinate (abw-2pp) at (2.67,0.65);
        \coordinate (w9) at (1.42,-1.925);
        \coordinate (abcu0) at (3.1,0.31);
        \coordinate (w0) at (3.1,-3);
        \coordinate (abcabu0pp) at (3.1,-0.53);
        \coordinate (abababu0pp) at (-0.255,0.1);
        \coordinate (u0) at (3.1,3.37);

        \draw[thick,blue,fill=blue!20] (w0) -- (w5) -- (bcu0) -- (w3) -- (w8) -- (abcu0) -- cycle;
        \draw[thick] (w0) -- (u0);
        \draw[thick] (w8) -- (u8);
        \draw[thick] (w6) -- (u6);
        \node[opacity=0.4] at (0,0.15) {\includegraphics[width=10.47619cm,bb=18.38095cm 14.5cm 45.61905cm 36.5cm,clip]{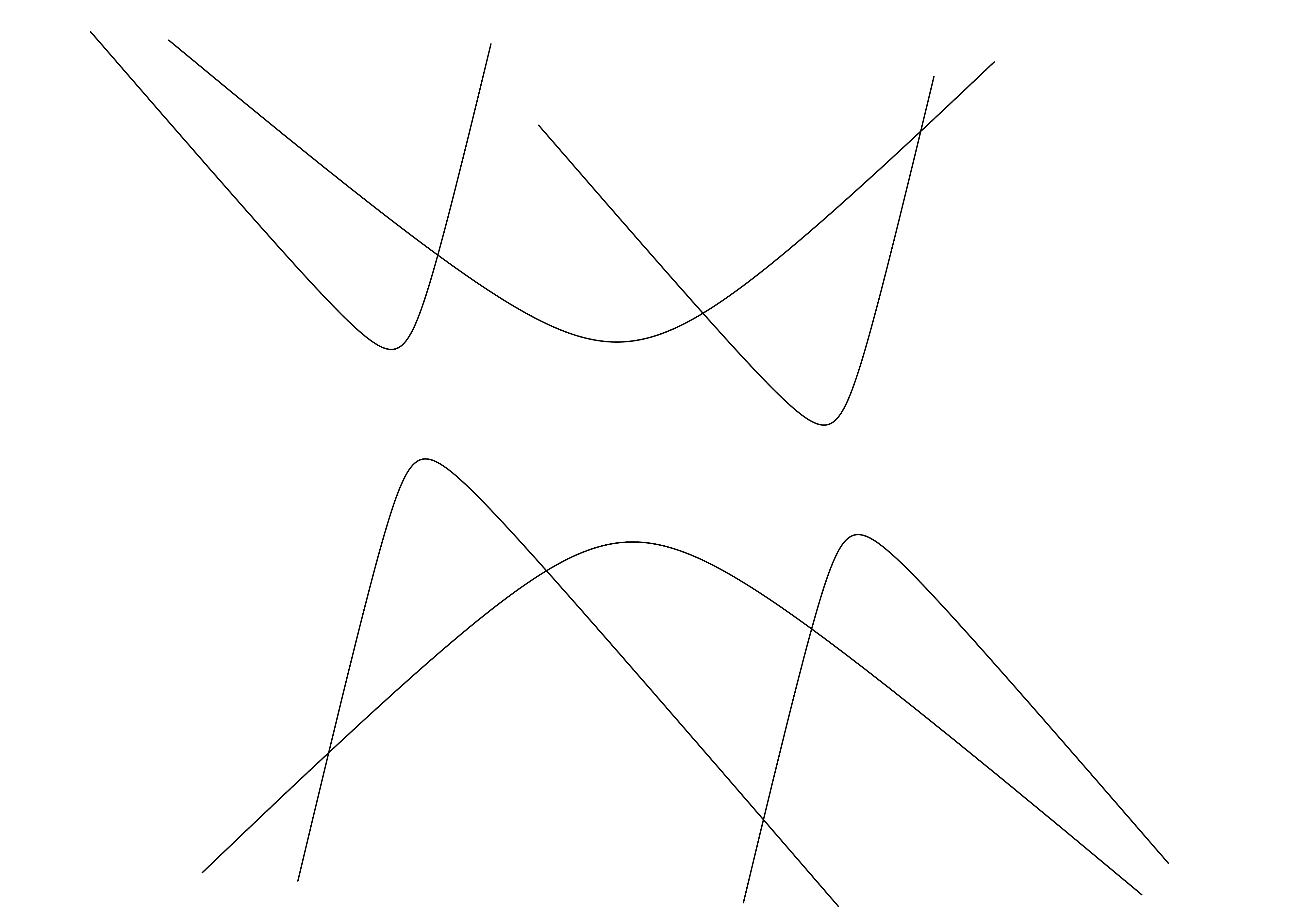}};
        \draw[thick,red,fill=red!20] (w0) -- (w9) -- (ababu0pp) -- (w8) -- (abw-2pp) -- (abcabu0pp) -- cycle;
        \draw[thick,red,fill=red!20] (w8) -- (w7) -- (abababu0pp) -- (w6) -- (ababw-2pp) -- (ababcabu0pp) -- cycle;

        \node[anchor=east] at (w3) {$w_3$};
        \node[anchor=east] at (w5) {$w_5$};
        \node[anchor=north] at (w6) {$w_6$};
        \node[anchor=south] at (u6) {$u_6$};
        \node[anchor=south] at (w7) {$w_7$};
        \node[anchor=south] at (w8) {$w_8$};
        \node[anchor=north] at (u8) {$u_8$};
        \node[anchor=north] at (w9) {$w_9$};
        \node[anchor=west] at (w0) {$w_0$};
        \node[anchor=east] at (u0) {$u_0$};
        \node[anchor=north west] at (abcu0) {$abcu_0$};
        \node[anchor=west] at (abcabu0pp) {$abcabu_0''$};
        \node[anchor=west] at (ababu0pp) {$ababu_0''$};
        \node[anchor=east] at (abw-2pp) {$abw_{-2}''$};
        \node[anchor=south east] at (bcu0) {$bcu_0$};

        \node[blue] at (-1.5,0.5) {$\bx$};
        \node[red] at (0.5,0.5) {\rotatebox{55}{$(ab)^2\bxpp$}};
        \node[red] at (2.3,-1) {$ab\bxpp$};

        \node[gray] at (-3.4,-1.7) {$bcC$};
        \node[gray] at (3.85,-1.9) {$abcC$};
        \node[gray] at (-1.8,-2.8) {$C$};
      \end{scope}
    \end{tikzpicture}
    \captionof{figure}{The relevant points and conics to prove the inclusions $(ab)^j \bxpp \subset \bx$, in the case $p_1 = p_2 = p_3 = 5$. The inclusions $a(ab)^j\bxpp \subset \bx$ follow by $a$--invariance of $\bx$.}
    \label{fig:box_inclusions}
  \end{center}

    \begin{Lem}\label{lem:box_inclusion_pr_not_3}
    Assume that $p_2 > 3$ or $p_3 > 3$ and let $g \in \rho(Q)$.
    Then $g\bxpp \subset \bx$.

    We even have $g\{w_{-2}'',abu_0''\} \subset \bx^\circ$, also $gw_5'' \in \bx^\circ$ if $p_3 > 3$, and also $gw_3'' \in \bx^\circ$ if in addition $g \not\in \{b,ab\}$.
  \end{Lem}

  \begin{Prf}
    We can assume that $g = (ab)^j$ with $1 \leq j \leq \frac{p_3-1}{2}$.
    If $j > 1$ then $\bxpp \subset \CH{\ell_{2+2j}}{w_0,w_2,u_0,u_2}$, so $g\bxpp \subset \CH{\ell_2}{w_{-2j},w_{2-2j},u_{-2j},u_{2-2j}} \subset \bx$ by \autoref{lem:box_in_ch} and \autoref{lem:wi_in_box}.
    Of the points $w_{-2j},w_{2-2j},u_{-2j},u_{2-2j}$ the only one which can hit $\partial\bx$ is $w_{2-2j} = gw_0''$ if $j = 2$, so all other points in the convex hull are in $\bx^\circ$.

    Now suppose $j = 1$.
    Then $gw_0'' = w_0$, $gw_3'' = w_{-2}$, and $gw_5'' = w_{-1}$.
    These points are in $\bx$ by \autoref{lem:wi_in_box}, and even $gw_5'' \in \bx^\circ$ if $p_3 > 3$.
    Further $gw_{-2}'' = \CH{\ell_1}{w_0,w_{-2},abcu_0}^\circ \subset \bx^\circ$ and $gcabu_0'' \in \CH{\ell_1}{w_0,abcu_0} \subset \bx$ by \autoref{lem:w2inCH} and \autoref{lem:cabu0inCH}.
    And finally $gabu_0'' \in \bx^\circ$ by \autoref{lem:ababu0_in_box}.

    This shows all six vertices of $g\bxpp$ are in $\bx$.
    The only thing we still need is that $g\bxpp \cap \ell_2 = \varnothing$, so that we can take the convex hull in the affine chart $\RP^2 \setminus \ell_2$.
    But by \autoref{lem:box_in_ch} $g\bxpp \cap \ell_2 = \varnothing$ if $g\ell_i = \ell_2$ for some $i \not\in \{-1,0,1,2\}$.
    It is easy to see that this is true for all $g \in \rho(Q)$.
  \end{Prf}

  \subsection{The case \texorpdfstring{$p_2=p_3=3$}{p2=p3=3}}\label{sec:p2_p3_3}
  
  If $p_2 = p_3 = 3$ then $g\bxpp$ is generally not contained in $\bx$ for $g \in \rho(Q)$.
  But we can skip one step and prove that $gg''\bxp \subset \bx$ for all $g \in \rho(Q)$ and $g'' \in \rho(Q'')$.
  Again, we do this by showing that the vertices of these boxes are in $\bx$ (see \autoref{fig:box_inclusions_p3}).

  \begin{Lem}\label{lem:pr_3_points}
    Assume $p_2 = p_3 = 3$. The points $bcaw_{-2}'$, $bcacau_0'$, $bcabcau_0'$ and $bacau_0'$ are all contained in $\bx^\circ$.
  \end{Lem}

  \begin{Prf}
    Applying $b$ to \autoref{lem:cabu0inCH} gives $bcabu_0'' \in \CH{\ell_2}{w_3,bcu_0}^\circ$.
    The ``double primed'' version of this is $abcau_0' \in \CH{\ell_2''}{w_3'',abu_0''}^\circ$, but by \autoref{lem:ch_chart_change} we could also take $\ell_1''$ instead of $\ell_2''$.
    If we apply $bc$ to that we get
    \[bcabcau_0' \in \CH{bc\ell_1''}{bcw_3'',bcabu_0''}^\circ = \CH{\ell_2}{w_3,bcabu_0''}^\circ \subset \CH{\ell_2}{w_3,bcu_0}^\circ \subset \bx^\circ.\]

    Applying $a$ to the double primed version of \autoref{lem:w2inCH} gives $aw_{-2}' \in \CH{\ell_2''}{w_3'',w_5'',abu_0''}^\circ$.
    Again, \autoref{lem:ch_chart_change} allows to replace $\ell_2''$ by $\ell_1''$, and therefore
    \[bcaw_{-2}' \in \CH{bc\ell_1''}{bcw_3'',bcw_5'',bcabu_0''}^\circ = \CH{\ell_2}{w_3,bw_{-2}'',bcabu_0''}^\circ \subset \bx^\circ,\]
    where $bw_{-2}'' \in \bx$ follows from \autoref{lem:w2inCH}.

    Next, in the proof of \autoref{lem:ababu0_in_box} we showed that $ababu_0'' \in \CH{\ell_2}{w_5,w_4,bw_{-2}''}^\circ$ if $p_3=3$ and $p_2 > 3$.
    In the present case we have $p_2 = 3$ and $p_1 > 3$, so $cacau_0' \in \CH{\ell_2''}{w_5'',w_4'',aw_{-2}'}^\circ$.
    Applying $b$ to this shows
    \[bcacau_0' \in \CH{b\ell_2''}{bw_5'',bw_4'',baw_{-2}'}^\circ = \CH{\ell_2}{w_{-2},bw_{-2}'',abw_{-2}''}^\circ \subset \bx^\circ.\]
    Finally, consider the points $w_4 = bacaw_0'$, $bcaw_{-2}' = bacaw_7'$ and $bw_{-2}'' = bacaw_2'$ and the line $\ell_{-2}' = baca\ell_3'$.
    We showed above that these points are in $\bx$.
    The points $w_0',w_3',w_7',u_0',w_2',u_3'$ are in this order along $C'$, so $u_0' \in \CH{\ell_3'}{w_0',w_7',w_2'}^\circ$ and therefore
    \[bacau_0' \in \CH{baca\ell_3'}{bacaw_0',bacaw_7',bacaw_2'}^\circ = \CH{\ell_{-2}'}{w_4,bcaw_{-2}',bw_{-2}''}^\circ \subset \bx^\circ.\]
    For the last inclusion we used that $\ell_{-2}' \cap \bx = \varnothing$ by \autoref{lem:box_in_ch}.
  \end{Prf}

  \begin{center}
    \begin{tikzpicture}[scale=1]
      \coordinate (p1) at (-3.95,3.78);
      \coordinate (p2) at (-3.95,0.55);
      \coordinate (p3) at (-3.95,-1.34);
      \coordinate (p4) at (-3.32,-0.365);
      \coordinate (p5) at (-2.69,-0.49);
      \coordinate (p6) at (-1.56,3.32);
      \coordinate (p7) at (-1.47,-1.92);
      \coordinate (p8) at (0.1,1.75);
      \coordinate (p9) at (3.95-0.152,-3.78-0.185);
      \coordinate (p10) at (3.95-0.152,-0.55-0.185);
      \coordinate (p11) at (3.95-0.152,1.34-0.185);
      \coordinate (p12) at (3.32-0.152,0.365-0.185);
      \coordinate (p13) at (2.69-0.152,0.49-0.185);
      \coordinate (p14) at (1.56-0.152,-3.32-0.185);
      \coordinate (p15) at (1.47-0.152,1.92-0.185);
      \coordinate (p16) at (-0.1-0.152,-1.75-0.185);

      \draw[fill=blue!20] (p1) -- (p3) -- (p14) -- (p9) -- (p11) -- (p6) -- cycle;
      \draw[fill=red!20] (p6) -- (p5) -- (p7) -- (p14) -- (p13) -- (p15) -- cycle;
      \draw[fill=red!20] (p2) -- (p4) -- (p7) -- (p8) -- (p6) -- (p1) -- cycle;
      \node[opacity=0.4] at (0,0) {\includegraphics[width=10cm,bb=5cm 0.5cm 30cm 22.5cm,clip]{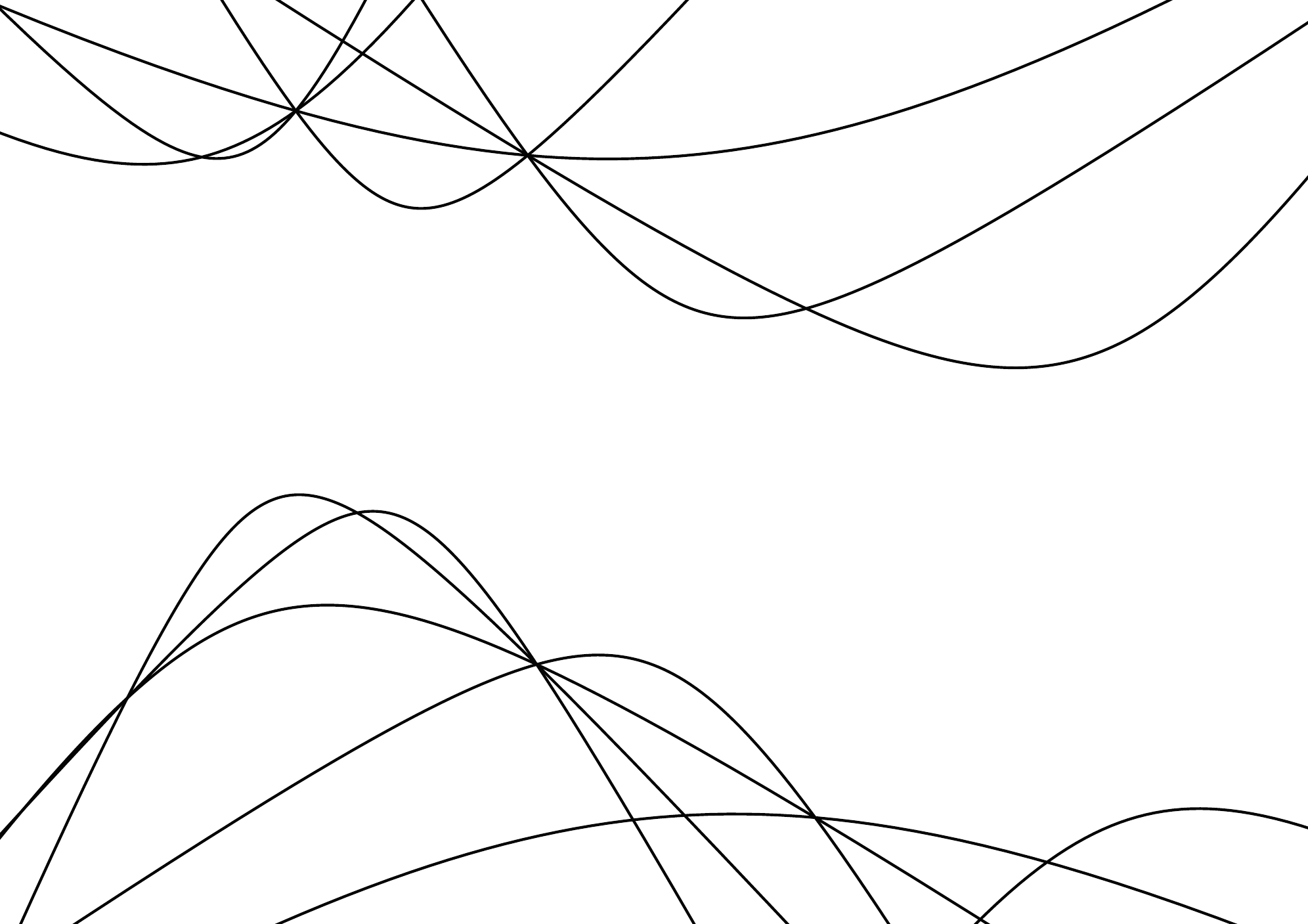}};
      \draw[thick,blue] (p1) -- (p3) -- (p14) -- (p9) -- (p11) -- (p6) -- cycle;
      \draw[thick,red] (p6) -- (p5) -- (p7) -- (p14) -- (p13) -- (p15) -- cycle;
      \draw[thick,red] (p2) -- (p4) -- (p7) -- (p8) -- (p6) -- (p1) -- cycle;

      \node[gray] at (2,4) {$C$};
      \node[gray] at (-0.8,-4.2) {$bacaC'$};
      \node[gray] at (-5.5,-0.8) {$bcabC''$};
      \node[gray] at (-4.8,-1.9) {$bcC$};
      \node[gray] at (4.8,1.9) {$abcC$};
      \node[gray] at (2.8,2.9) {$bcacaC'$};
      \node[gray] at (-2.6,-3.8) {$C$};

      \node[red] at (0.7,-0.7) {$ba\bxp$};
      \node[red] at (-3,2) {$bca\bxp$};
      \node[blue] at (2.7,-2.3) {$\bx$};


      \node[anchor=south] at (p1) {$w_3$};
      \node[anchor=east] at (p2) {$bcabcau_0'$};
      \node[anchor=north] at (p3) {$bcu_0$};
      \node[anchor=south] at (p6) {$w_4$};
      \node[anchor=north] at (p14) {$w_5$};
      \node[anchor=north] at (p9) {$w_0$};
      \node[anchor=west] at (p7) {$bw_{-2}''$};
      \node[anchor=east] at (-4,-0.2) {$bcabu_0''$};
      \node[shift={(0.35cm,0.65cm)}] at (p4) {\rotatebox{65}{$bcaw_{-2}'$}};
      \node[shift={(0.5cm,0.45cm)}] at (p5) {\rotatebox{65}{$bacau_0'$}};
      \node[shift={(-0.5cm,-0.55cm)}] at (p8) {\rotatebox{65}{$bcacau_0'$}};
      \node[shift={(-0.25cm,-0.6cm)}] at (p15) {\rotatebox{65}{$abw_{-2}''$}};
      \node[anchor=south] at (p11) {$abcu_0$};
      \node[anchor=north] at (p13) {$abacau_0'$};

    \end{tikzpicture}
    
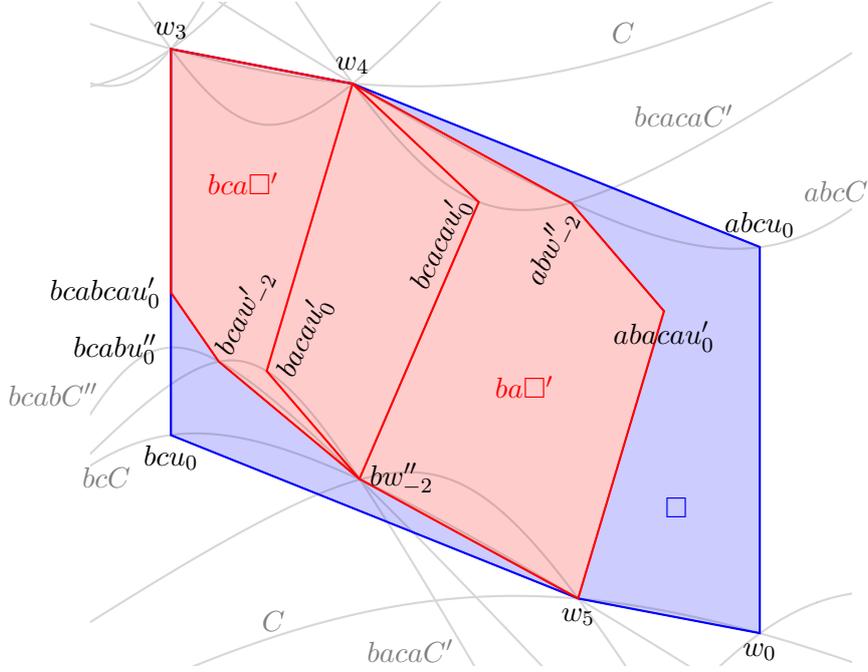
\captionof{figure}{The relevant points for the proof of \autoref{lem:pr_3_points} and the conics defining them, in the case $(p_1,p_2,p_3) = (5,3,3)$. Note that $bca\bxp$ and $ba\bxp$ overlap.}
    \label{fig:box_inclusions_p3}
  \end{center}

   \begin{Lem}\label{lem:box_inclusion_pr_3}
    Assume $p_2 = p_3 = 3$ and let $g \in \rho(Q)$ and $g'' \in \rho(Q'')$. Then $gg''\bxp \subset \bx$. Furthermore, $gg''\{w_{-2}',w_5'\} \subset \bx^\circ \cup \{w_{-2},w_5\}$.
  \end{Lem}

  \begin{Prf}
    Recall that $\rho(Q)\rho(Q'') = \{ba,aba,bca,abca\}$.
    Since $a\bx = \bx$ we can assume $gg'' \in \{ba,bca\}$.
    First, we show that the vertices of $ba\bxp$ are in $\bx$:
    $baw_0' = w_5 \in \bx$ by \autoref{lem:wi_in_box}, $bacau_0' \in \bx^\circ$ by \autoref{lem:pr_3_points}, and $baw_5' = bw_{-2}'' \in \CH{\ell_2}{w_3,w_5,bcu_0}^\circ \subset \bx^\circ$ by \autoref{lem:w2inCH}.
    The remaining three vertices follow by symmetry:
    for every vertex $x$ of $\bxp$ such that $bax \in \bx$, the point $bx$ is another vertex of $\bxp$ and $ba(bx) = abax \in a\bx = \bx$.

    The vertices of $bca\bxp$ are $bcaw_{-2}'$, $bcacau_0'$, $bcabcau_0'$, which are in $\bx^\circ$ by \autoref{lem:pr_3_points}, $bcaw_3' = bw_{-2}''$ which was shown to be in $\bx^\circ$ in the previous paragraph, and $bcaw_0' = w_3$ and $bcaw_5' = w_{-2}$, also in $\bx$.
    By \autoref{lem:box_in_ch} $\bxp \cap \ell_4'' = \varnothing$, and $bca \ell_4'' = \ell_2$, so $bca\bxp$ also avoids $\ell_2$, hence $bca\bxp \subset \bx$.
  \end{Prf}
  
  \subsection{Iteration}\label{sec:iteration}

  If $p_1 > 3$ and $h = gg''g' \in \rho(QQ''Q') = \rho(T)$ then either $h\bx \subset gg''\bxp \subset g\bxpp \subset \bx$ by \autoref{lem:box_inclusion_pr_not_3}, or, if $p_2 = p_3 = 3$, then $h\bx \subset gg''\bxp \subset \bx$ by \autoref{lem:box_inclusion_pr_3}.
  To apply \autoref{prop:limit_map_exists} we need a bit more, that $h\bx \subset \bx^\circ$ for most $h \in \rho(T)$.
  We get this by carefully examining which vertices of $h\bx$ can end up on the boundary of $\bx$.

  \begin{Lem}\label{lem:strict_inclusion_iteration}
    Assume $p_1 > 3$ and $h \in \rho(T)$. Then $h\{w_{-2},w_5,bcu_0\} \subset \bx^\circ$ and $hw_3 \in \bx^\circ \cup \{w_{-2},w_5\}$. Moreover, $hw_3 \in \bx^\circ$ unless $p_2 = p_3 = 3$.
  \end{Lem}

  \begin{Prf}
    We write $h = gg''g' \in \rho(Q)\rho(Q'')\rho(Q')$.
    If $x \in \{w_{-2},w_5,bcu_0\}$, then $g'x \in \bx'^\circ$ by \autoref{lem:box_inclusion_pr_not_3}, so $hx \in \bx^\circ$.

    \autoref{lem:box_inclusion_pr_not_3} also tells us that $g'w_3 \in \bxp^\circ$ unless $g' \in \{c,bc\}$, in which case $g'w_3 \in \{w_{-2}',w_5'\}$.
    If $p_2 > 3$ then $g''\{w_{-2}',w_5'\} \subset \bxpp^\circ$, so $hw_3 \in \bx^\circ$.
    If $p_2 = 3$ and $p_3 > 3$, then still $g''w_{-2}' \in \bxpp^\circ$, but $g''w_5' \in \{w_{-2}'',w_5''\}$, and $g\{w_{-2}'',w_5''\} \subset \bx^\circ$, so $hw_3 \in \bx^\circ$.
    Finally, in the case $p_2 = p_3 = 3$ we have $hw_3 \in gg'' \{w_{-2}',w_5'\} \subset \bx^\circ \cup \{w_{-2},w_5\}$ by \autoref{lem:box_inclusion_pr_3}.
  \end{Prf}

  \begin{Lem}\label{lem:box_nonintersects_l0}
    Assume $p_2 > 3$ or $p_3 > 3$ and let $g \in \rho(Q)$. Then $g\bxpp \cap \ell_0 = \varnothing$ unless $g = ab$. Similarly, if $p_2 = p_3 = 3$ and $g \in \rho(Q)$, $g'' \in \rho(Q'')$, then $gg''\bxp \cap \ell_0 = \varnothing$ unless $g = ab$ and $g'' = ca$.
  \end{Lem}

  \begin{Prf}
    By \autoref{lem:box_in_ch} $g\bxpp \cap \ell_0 = \varnothing$ if $g\ell_i = \ell_0$ for some $i \not\in \{-1,0,1,2\}$.
    This is true for all $g \in \rho(Q)$ except if $g = ab$ or $g = b$ and $p_3 = 3$.
    So assume $p_3 = 3$ and $b\bxpp \cap \ell_0 \neq \varnothing$.
    The ``primed'' version of \autoref{lem:box_in_ch} shows that $\ell_0 = \ell_2'$ only intersects $\bx$ in its boundary.
    As $b\bxpp \subset \bx$ it follows that $b\bxpp \cap \ell_0$ must contain a vertex of $b\bxpp$.
    These vertices are $bw_0'' = w_3$, $bw_3'' = w_5$, and $bw_5'' = w_4$, which are not on $\ell_0$, $babu_0''$ and $bw_{-2}''$, which are in $\bx^\circ$ by \autoref{lem:box_inclusion_pr_not_3}, and $bcabu_0''$.
    But $bcabu_0'' \in \CH{\ell_2}{w_3,bcu_0} \subset \ell_3 \cap M$ by \autoref{lem:cabu0inCH}, in particular it is not in $\ell_0$.

    Using the same argument in the case $p_2 = p_3 = 3$, if $gg''\bxp \cap \ell_0$ is non--empty it contains a vertex of $gg''\bxp$.
    Now $gg'' \in \{ba,aba,bca\}$, but we can ignore the case $aba$ since $aba\bxp = ba\bxp$.
    If we just list the twelve vertices of $ba\bxp$ and $bca\bxp$ we see they are only ten distinct points.
    We showed in the proof of \autoref{lem:box_inclusion_pr_3} already that seven of them are in $\bx^\circ$.
    The remaining points are $bcaw_0' = w_3$, $baw_0' = w_5$, and $bcaw_5' = baw_3' = w_{-2}$; see \autoref{fig:box_inclusions_p3}.
    None of them are on the line $\ell_0$.
  \end{Prf}

  \begin{Lem}\label{lem:box_nonintersects_l0_iteration}
    Assume $p_1 > 3$ and let $h \in \rho(T)$. Then $h\bx \cap \ell_0 = \varnothing$ unless $h = abcabc$.
  \end{Lem}

  \begin{Prf}
    We write $h = gg''g' \in \rho(QQ''Q')$.
    If $h\bx \cap \ell_0 \neq \varnothing$ and $p_2 > 3$ or $p_3 > 3$, then $g\bxpp \cap \ell_0 \neq \varnothing$, so $g = ab$ by \autoref{lem:box_nonintersects_l0}.
    Hence $g(g''\bxp \cap \ell_0'') = gg''\bxp \cap \ell_0 \neq \varnothing$, so $g'' = ca$ by the same lemma.
    And finally $gg''(g'\bx \cap \ell_0') = h\bx \cap \ell_0 \neq \varnothing$, so $g' = bc$.
    In summary we get $h = gg''g' = abcabc$.
    The case $p_2 = p_3 = 3$ is similar, we just use the second part of \autoref{lem:box_nonintersects_l0} in place of the first two steps.
  \end{Prf}

  \begin{Prop}\label{lem:box_inclusion_final}
    Assume $p_1 > 3$ and let $h_1,h_2,h_3 \in \rho(T)$. Then either $h_3 = abcabc$ or
    \[h_1 h_2 h_3 \bx \subset \bx^\circ.\]
  \end{Prop}

  \begin{Prf}
    We already know that $h_1h_2h_3\bx \subset \bx$.
    Assume $h_1h_2h_3 \bx \cap \partial\bx$ is non--empty.
    Then $h_1h_2\bx \cap \partial \bx$ is also non--empty (since $h_3\bx \subset \bx$), and by convexity it must be a union of closed edges and vertices of $h_1h_2\bx$.
    But \autoref{lem:strict_inclusion_iteration} shows that none of $h_1h_2\{w_{-2},w_5,bcu_0,w_3\}$ can be in $\partial\bx$, so
    \[h_1h_2\bx \cap \partial\bx \subset h_1h_2\CH{\ell_2}{w_0,abcu_0} \subset h_1h_2 \ell_0.\]
    So $h_1h_2h_3\bx \cap \partial\bx \subset h_1h_2(h_3\bx \cap \ell_0)$, which is empty by \autoref{lem:box_nonintersects_l0_iteration} unless $h_3 = abcabc$.
  \end{Prf}

  As a convention, we write $g_\pm \in \RP^2$ for the attracting and repelling fixed point and $g^\pm \subset \RP^2$ for the attracting and repelling fixed line of $g \in \SL(3,\bR)$, if they exist.
  This is still defined if $g$ is nondiagonalizable with an eigenvalue of algebraic multiplicity 2.
  The following lemmas will be needed later.
  
  \begin{Lem}\label{lem:coxeter_boundary_box}
    Let $p_1 > 3$ and $h = abcabc$. Then $h\bx \subset \bx^\circ \cup (\ell_0 \cap \bx)$.
  \end{Lem}

  \begin{Prf}
    Recall that $\bx$ is the convex hull of $w_0, w_3, w_5, w_{-2}, bcu_0$, and $abcu_0$.
    Since $\bx^\circ \cup (\ell_0 \cap \bx)$ is convex it suffices that these six points get mapped into this set by $h$.
    Now $w_0, abcu_0 \in \ell_0$ and $\ell_0$ is preserved by $h$, so it only remains to show this for the other four points.
    
    Assume the opposite.
    By \autoref{lem:strict_inclusion_iteration} this is only possible if $p_2 = p_3 = 3$ and $hw_3 \in \{w_{-2}, w_5\}$, or equivalently $bcw_3 \in \{acbaw_{-2}, acbaw_5\}$.
    But $bcw_3 = bcw_0' = w_{-2}'$ and
    \[acbaw_{-2} = acw_0 = acw_3'' = w_5'', \qquad acbaw_5 = acw_7 = acw_1 = acw_5'' = w_7'' = w_1'',\]
    so this would imply $w_{-2}' \in \{w_1'',w_5''\}$.
    But according to the discussion before \autoref{lem:reducible_order} there can be no fifth point in $C' \cap C''$ except $w_0',w_2',w_3',w_5'$.
    So $h\bx \subset \bx^\circ \cup \ell_0$.
  \end{Prf}
  
  \begin{Lem}\label{lem:negative_coxeter_box}
    Let $p_1 > 3$ and $h = abcabc$.
    If $t_\rho > \tcrit$ then the repelling line $h^-$ does not intersect $\bx$. If $t_\rho = \tcrit$ then $h^- \cap \bx = \{h_+\}$.
  \end{Lem}

  \begin{Prf}
    First consider the case $t = \tred$, where $\rho$ has a global fixed point $x \in \RP^2$.
    The order of the lines $\ell_i$ and $h^- = \rho(s_1s_2s_3)^-$ in the pencil of lines through $x$ is the same as the order of the points $z_i$ and $(s_1s_2s_3)_-$ along $S^1$.
    We found in the end of \autoref{sec:hyperbolic} that $z_2$ and $(s_1s_2s_3)_-$ lie in the arc of $S^1 \setminus \{z_0, z_3\}$ not containing $z_5$ and $z_{-2}$.
    Hence $\ell_2$ and $h^-$ also lie in the component of $\RP^2 \setminus (\ell_0 \cup \ell_3)$ opposite to that containing $\ell_5$ and $\ell_{-2}$ (ignoring the point $x$ itself).
    Since $\bx$ is defined as the convex hull of points contained in $\ell_0, \ell_3, \ell_5, \ell_{-2}$, and avoiding $\ell_2$, this implies that $h^-$ does not intersect $\bx$.

    For the general case, continuously deform $\rho$ starting from $t_\rho = \tred$ until $h^-$ and $\bx$ intersect for the first time.
    Let $A = h^- \cap \bx$.
    As the definitions of $\bx$ and $h^-$ change continuously with $t_\rho$, we have $A \subset \partial\bx$.
    Then $hA = h^- \cap h\bx \subset A$ (since $h\bx \subset \bx$, see beginning of this section), so $hA \subset \partial\bx \cap h\bx$.
    By \autoref{lem:coxeter_boundary_box} this implies $hA \subset \ell_0 = h^+$, and in fact $A \subset h^+$.
    But the intersection $h^- \cap h^+$ contains only a single point, the ``neutral'' fixed point $h_0$.
    To be more precise, if $t_\rho > \tcrit$ then $h$ has distinct real eigenvalues and $h_0$ is the middle eigenspace, and if $t_\rho = \tcrit$ then $h$ is nondiagonalizable with an eigenvalue of multiplicty 2 and $h_0 = h_+$ (see \autoref{rem:critical_moment}).
    So $A = \{h_0\}$.

    Note that $h_0$ must be on the boundary of $\bx \cap \ell_0$ within $\ell_0$.
    This means either $h_0 = abcu_0$ or $h_0 = w_0$.
    $h_0 = abcu_0$ would imply that $u_0$ is fixed by $abc$.
    So $cu_0 = bau_0 = u_2 \in C$, contradicting \autoref{lem:u0inM}.
    Therefore, we have $h_0 = w_0 = h_+$.
    By \autoref{lem:coxeter_eigenvalues} this means that $t_\rho = \tcrit$ and $h^- \cap \bx = \{h_+\}$.
  \end{Prf}
  
  \section{Duality}\label{sec:duality}

  The results from the previous two sections allow us to construct boundary maps into $\RP^2$ for representations $\rho$ of type $(\frac{p_1-1}{2},\frac{p_2-1}{2},\frac{p_3-1}{2})$ with parameter $t_\rho \in [\tcrit, \infty)$.
  Now we can leverage two forms of duality: first to extend this to the case $t_\rho \in (0,\tcrit^{-1}]$ and then to also construct a boundary map into the dual projective plane $(\RP^2)^*$.
  Note that by the proof of \autoref{lem:coxeter_eigenvalues} reordering $(p_1,p_2,p_3)$ does not change the value of $\tcrit$.
  Also note that in this section we write $\partial\Gamma$ instead of $S^1$ for the group boundary, to be more precise when two different groups are involved.

  \begin{Lem}\label{lem:parameter_less_than_1}
    Let $\rho \colon \partial\Gamma_{p_1,p_2,p_3} \to \SL(3,\bR)$ have type $(\frac{p_1-1}{2},\frac{p_2-1}{2},\frac{p_3-1}{2})$ and parameter $t_\rho \in (0,\tcrit^{-1}] \cup [\tcrit,\infty)$. Then there exists a continuous $\rho$--equivariant map
    \[\xi^{(1)} \colon \partial\Gamma_{p_1,p_2,p_3} \to \RP^2.\]
  \end{Lem}

  \begin{Prf}
    As discussed in \autoref{lem:reducible_identify} and \autoref{sec:Anosov_representations}, the only reducible representations in this component are $\rho_{\mathrm{red}}$ and $\rho_{\mathrm{red}}'$, and both are Anosov.
    So we can assume $\rho$ is irreducible.
    If $t_\rho \geq \tcrit$ then the set $\bx$ from \autoref{lem:box_inclusion_final} satisfies the assumptions of \autoref{prop:limit_map_exists}, where assumption (vi) follows from \autoref{lem:negative_coxeter_box}.
    So there exists a continuous $\rho$--equivariant map $\xi^{(1)} \colon \partial\Gamma_{p_1,p_2,p_3} \to \RP^2$.

    Now assume $t_\rho \in (0,\tcrit^{-1}]$ and let $\psi \colon \Gamma_{p_1,p_3,p_2} \to \Gamma_{p_1,p_2,p_3}$ be the group isomorphism which fixes $s_1$ and interchanges $s_2$ with $s_3$.
    It is an isometry of Cayley graphs and hence induces a homeomorphism $\partial \psi \colon \partial\Gamma_{p_1,p_3,p_2} \to \partial\Gamma_{p_1,p_2,p_3}$ of the group boundaries, which is $\psi$--equivariant.

    Using the notation $c_i = 2\cos(\frac{p_i-1}{2p_i}\pi)$ as in the proof of \autoref{lem:homeoToR}, the Cartan matrix of $\rho \circ \psi$ is (equivalent to)
    \[\begin{pmatrix}1 & 0 & 0 \\ 0 & 0 & 1 \\ 0 & 1 & 0\end{pmatrix}\begin{pmatrix}2 & -c_3 & -c_2 \\ -c_3 & 2 & -t_{\rho}c_1 \\ -c_2 & -t_{\rho}^{-1}c_1 & 2\end{pmatrix}\begin{pmatrix}1 & 0 & 0 \\ 0 & 0 & 1 \\ 0 & 1 & 0\end{pmatrix} = \begin{pmatrix}2 & -c_2 & -c_3 \\ -c_2 & 2 & -t_\rho^{-1} c_1 \\ -c_3 & -t_\rho c_1 & 2\end{pmatrix}.\]
    Hence $\rho \circ \psi$ is a representation of $\Gamma_{p_1,p_3,p_2}$ of type $(\frac{p_1-1}{2},\frac{p_3-1}{2},\frac{p_2-1}{2})$ with parameter $t_{\rho \circ \psi} = t_\rho^{-1} \in [\tcrit,\infty)$.
    So there exists a continuous $(\rho \circ \psi)$--equivariant boundary map $\xi'^{(1)} \colon \partial\Gamma_{p_1,p_3,p_2} \to \RP^2$ by the above.
    But then
    \[\xi^{(1)} = \xi'^{(1)} \circ \partial\psi^{-1} \colon \partial\Gamma_{p_1,p_2,p_3} \to \RP^2\]
    is $\rho$--equivariant.
  \end{Prf}

  \begin{Prop}\label{prop:limit_curve_in_flags}
    Let $\rho \colon \partial\Gamma_{p_1,p_2,p_3} \to \SL(3,\bR)$ be as in \autoref{lem:parameter_less_than_1}. Then there exists a continuous $\rho$--equivariant map
    \[\xi \colon \partial\Gamma_{p_1,p_2,p_3} \to \F.\]
    into the flag manifold $\F$ (defined in \autoref{sec:Anosov_representations}).
    It maps the attracting (resp. repelling) fixed point $\gamma_\pm$ of the Coxeter element $\gamma = s_1s_2s_3$ to the attracting (repelling) flag of $\rho(\gamma)$.
  \end{Prop}

  \begin{Prf}
    As in the proof of \autoref{lem:parameter_less_than_1} we can assume that $\rho$ is irreducible.
    By \autoref{lem:parameter_less_than_1} there is $\rho$--equivariant boundary map $\xi^{(1)} \colon \partial \Gamma \to \RP^2$ .

    To obtain a dual boundary map, we consider the inverse transposed representation, which we call $\rho^{-T}$. Since
    \[\rho(s_i)^{-T} = (b_i \otimes \alpha_i - 1)^T = (\alpha_i^T \otimes b_i^T - 1)\]
    its Cartan matrix is just the transpose of that of $\rho$.
    Hence the type of $\rho^{-T}$ is also $(\frac{p_1-1}{2},\frac{p_2-1}{2},\frac{p_3-1}{2})$, but $t_{\rho^{-T}} = t_\rho^{-1}$.
    By \autoref{lem:parameter_less_than_1} there is a $\rho^{-T}$--equivariant continuous boundary map $\xi' \colon \partial\Gamma \to \RP^2$.
    If we write $D \colon \RP^2 \to (\RP^2)^*$ for the duality induced by the standard scalar product on $\bR^3$, then
    \[\xi^{(2)} = D \circ \xi' \colon \partial\Gamma \to (\RP^2)^*\]
    is $\rho$--equivariant.

    For the remaining part, note that the Coxeter element $\rho(s_1s_2s_3)$ does not have three distinct eigenvalues if $t_\rho \in \{\tcrit^{-1},\tcrit\}$ (\autoref{lem:coxeter_eigenvalues}).
    But we can choose $\gamma \in \{s_1s_2s_3,s_3s_2s_1\}$ so that $\rho(\gamma)$ is proximal, i.e. has a unique eigenvalue of maximal modulus.
    We write $\rho(\gamma)_+$ for its attracting point and $\rho(\gamma)^-$ for its repelling line in the projective plane.

    By irreducibility $\xi^{(1)}(\partial\Gamma)$ cannot be contained in a line, so there is some $z \in \partial \Gamma$ with $\xi^{(1)}(z) \not\in \rho(\gamma)^-$.
    Hence $\xi^{(1)}(\gamma^nz) = \rho(\gamma)^n\xi^{(1)}(z) \to \rho(\gamma)_+$.
    By continuity of $\xi^{(1)}$ there are infinitely many such $z$, so we can assume $z \neq \gamma_-$.
    But then $\gamma^n z \to \gamma_+$ in $\partial\Gamma$, so continuity implies $\xi^{(1)}(\gamma_+) = \rho(\gamma)_+$.
    An analogous argument shows that $\xi^{(2)}(\gamma_-) = \rho(\gamma)^-$.

    If $t_\rho \not\in \{\tcrit^{-1},\tcrit\}$ then $\rho(\gamma^{-1})$ is also proximal, so we can repeat the argument with $\gamma^{-1}$ in place of $\gamma$ and get $\xi^{(1)}(\gamma_-) = \rho(\gamma)_-$ and $\xi^{(2)}(\gamma_+) = \rho(\gamma)^+$.
    If $t_\rho \in \{\tcrit^{-1},\tcrit\}$, $\rho(\gamma)$ is not diagonalizable by \autoref{lem:coxeter_eigenvalues}.
    Hence it has exactly two fixed points $\rho(\gamma)_+$ and $\rho(\gamma)_-$ and two fixed lines $\rho(\gamma)^+$ and $\rho(\gamma)^-$, with $\rho(\gamma)_\pm \in \rho(\gamma)^\pm$.
    Since $\rho(\gamma)_+ \not\in \rho(\gamma)^-$, we necessarily have $\xi^{(1)}(\gamma_-) = \rho(\gamma)_-$ and $\xi^{(2)}(\gamma_+) = \rho(\gamma)^+$.

    In any case $\xi^{(1)}(\gamma_+) \in \xi^{(2)}(\gamma_+)$, and since the orbit of $\gamma_+$ is dense in $\partial\Gamma$, we have $\xi^{(1)}(x) \in \xi^{(2)}(x)$ for all $x \in \partial\Gamma$, so $\xi^{(1)}$ and $\xi^{(2)}$ combine to a map $\xi$ into $\F$.
\end{Prf}

\section{Transversality}\label{sec:transversality}

In this section let $\rho \colon \Gamma_{p_1,p_2,p_3} \to \SL(3,\bR)$ be a representation of type $(\frac{p_1-1}{2}, \frac{p_2-1}{2}, \frac{p_3-1}{2})$ with $t_\rho \in (0, \tcrit^{-1}] \cup [\tcrit, \infty)$, and let $\xi \colon S^1 \to \F$ be its boundary map, which exists by \autoref{prop:limit_curve_in_flags}.
The main result is that $\xi$ is transverse if $t_\rho < \tcrit^{-1}$ or $t_\rho > \tcrit$, that is $\xi^{(1)}(x) \not\in \xi^{(2)}(y)$ whenever $x \neq y$.
This is false if $t_\rho = \tcrit^{-1}$ or $t_\rho = \tcrit$, as the Coxeter element is not diagonalizable by \autoref{lem:coxeter_eigenvalues}, so its attracting and repelling flags are not transverse.
We are first going to prove transversality for pairs in $S^1 = \partial \Gamma$ of which one element is a fixed point of the Coxeter element, then extend this to a certain open subset of pairs, and finally show that the $\Gamma$--orbit of this subset comprises all distinct pairs.

\begin{Lem}\label{lem:coxeter_transverse}
  Assume $t_\rho \geq \tcrit$, let $\gamma = s_1s_2s_3$ and $z \in S^1 \setminus \{\gamma_+,\gamma_-\}$. Then $\xi(z)$ is transverse to $\xi(\gamma_+)$ and $\xi(\gamma_-)$.
\end{Lem}

\begin{Prf}
  If $\xi^{(1)}(z) \in \xi^{(2)}(\gamma_+)$, then we would have, by \autoref{prop:limit_curve_in_flags},
  \[\rho(\gamma)_- = \xi^{(1)}(\gamma_-) = \lim_{n\to\infty} \xi^{(1)}(\gamma^{-n}z) = \lim_{n\to\infty} \rho(\gamma)^{-n}\xi^{(1)}(z) \in \xi^{(2)}(\gamma_+) = \rho(\gamma)^+.\]
  This is clearly false if $t_\rho > \tcrit$ and $\rho(\gamma)$ has distinct real eigenvalues.
  In the case $t_\rho = \tcrit$ $\rho(\gamma)$ is non--diagonalizable with eigenvalues of the form $\lambda,\lambda,\lambda^{-2}$ for $\lambda < -1$, according to \autoref{lem:coxeter_eigenvalues}.
  This implies that $\rho(\gamma)_+ \in \rho(\gamma)^-$, but nevertheless $\rho(\gamma)_- \not\in \rho(\gamma)^+$, so we still get $\xi^{(1)}(z) \not\in \xi^{(2)}(\gamma_+)$.

  An analogous proof for $\xi^{(1)}(z) \not\in \xi^{(2)}(\gamma_-)$ would require that $\rho(\gamma)_+ \not\in \rho(\gamma)^-$, so we use a different argument:
  assume that $\xi^{(1)}(z) \in \xi^{(2)}(\gamma_-)$.
  Recall that $\gamma$ acts on the hyperbolic plane as a glide reflection with attracting fixed point $\gamma_+$ (see \autoref{sec:hyperbolic}), which is one of the end points of the interval $I$ (defined in \autoref{sec:definition_of_I_T_box}).
  So by applying $\gamma$ often enough, we can assume that $z \in I$.
  Then $\xi^{(1)}(z) \in \xi^{(1)}(I) \subset \bx$ by \autoref{prop:limit_map_exists} and the construction of $\xi$, but we know from \autoref{lem:negative_coxeter_box} that $\bx$ and $\xi^{(2)}(\gamma_-) = \rho(\gamma)^-$ can only intersect in $\rho(\gamma)_+$.
  So $\xi^{(1)}(z) = \rho(\gamma)_+ = \xi^{(1)}(\gamma_+)$. But this contradicts $\xi^{(1)}(z) \not\in \xi^{(2)}(\gamma_+)$ from the first part of the lemma.
  The remaining statements $\xi^{(1)}(\gamma_\pm) \not\in \xi^{(2)}(z)$ follow by applying both parts to the dual limit curve, as constructed in \autoref{sec:duality}.
\end{Prf}

Recall from \autoref{sec:nested_boxes} the definition of the points $z_0,\dots,z_{2p_1-1} \in S^1$ such that
\[z_0 = (s_1s_2s_3)_+, \quad z_3 = (s_2s_3s_1)_+, \quad s_2s_1 z_i = z_{i+2}\]
and the corresponding points $w_i,u_i \in \RP^2$, and lines $\ell_i \subset \RP^2$.
It follows from \autoref{prop:limit_curve_in_flags} that $\xi(z_i) = (w_i,\ell_i)$ for all $i$.
Also, recall that there is a unique conic $C$ which passes through all the points $w_i$ and $u_i$.
We write $[z_i,z_j]$ for the closed interval in $S^1$ containing the points $z_i,z_{i+1},\dots,z_j$.
As before, there are also ``primed'' and ``double--primed'' versions of these points, obtained by cyclically permuting $s_1,s_2,s_3$ in the definition.

\begin{Lem}\label{lem:intersections_on_intervals}
  Assume $t_\rho \geq \tcrit$.
  Let $k \in \{1,\dots,2p_3\}$ and $z \in [z_k,z_{k+1}]$.
  If $k$ is odd, the intersection points $\xi^{(2)}(z) \cap \xi^{(2)}(z_{k-1})$ and $\xi^{(2)}(z) \cap \xi^{(2)}(z_{k+2})$ are contained in the closed disk bounded by the conic $C$ (i.e. not in the Möbius strip $M$).
  If $k$ is even, the same is true for the intersection points $\xi^{(2)}(z) \cap \xi^{(2)}(z_{k-2})$ and $\xi^{(2)}(z) \cap \xi^{(2)}(z_{k+3})$.
\end{Lem}

\begin{Prf}
  Which side of $C$ the intersection point of two lines $\ell, \ell'$ lies on can be read off the cyclic order along $C$ of their intersection points with $C$.
  Let $P(C)$ be the space of unordered pairs in $C$.
  Let $[z_k,z_{\mathrm{right}}] \subset [z_k,z_{k+1}]$ be the largest interval so that $\xi^{(2)}(z)$ intersects $C$ for all $z \in [z_k,z_{\mathrm{right}}]$ (it will turn out that $\xi^{(2)}(z)$ intersects $C$ for all $z$, i.e. $z_{\mathrm{right}} = z_{k+1}$).  
  Define a map $f \colon [z_k,z_{\mathrm{right}}] \to P(C)$ by mapping $z \in S^1$ to the two intersection points of the line $\xi^{(2)}(z)$ with $C$.
  Then $f(z_k) = \ell_k \cap C = \{w_k,u_k\}$.
  Let $\widetilde f \colon [z_k,z_{\mathrm{right}}] \to C^2$ be the lift of $f$ to ordered pairs such that $\widetilde f(z_k) = (w_k,u_k)$.

  By continuity either $z_{\mathrm{right}} = z_{k+1}$ or $\widetilde f(z_{\mathrm{right}}) = (w,w)$ for some $w \in C$.
  In the latter case, the first component $\widetilde f_1$ would be a continuous arc connecting $w_k$ and $w$ through $C$, and $\widetilde f_2$ would be an arc connecting $u_k$ and $w$.
  By \autoref{lem:coxeter_transverse} we have $\widetilde f_1(z) \neq w_i$ and $\widetilde f_2(z) \neq w_i$ for any $i$ and any $z$ in the open interval $(z_k, z_{k+1})$.
  So together, $\widetilde f_1$ and $\widetilde f_2$ give a continuous arc from $w_k$ to $u_k$ avoiding all other $w_i$.
  This is impossible by \autoref{lem:order}.
  So $z_{\mathrm{right}} = z_{k+1}$.

  Assume $k$ is odd.
  \Autoref{lem:order} also shows that the eight points
  \[u_{k-1},w_k,w_{k+1},u_{k+2},w_{k-1},u_k,u_{k+1},w_{k+2}\]
  are in this cyclic order along $C$ (with $u_k$ and $u_{k+1}$ possibly switched).
  So $\widetilde f_1$ is a path from $w_k$ to $w_{k+1}$ avoiding all other $w_i$, and $\widetilde f_2$ is a path from $u_k$ to $u_{k+1}$ avoiding all $w_i$.
  Hence we get the cyclic order
  \[u_{k-1},w_k,\widetilde f_1(z),w_{k+1},u_{k+2},w_{k-1},\widetilde f_2(z),w_{k+2}\]
  for all $z$ in the open interval $(z_k,z_{k+1})$.
  Looking at the quadruples $u_{k-1},\widetilde f_1(z),w_{k-1},\widetilde f_2(z)$ and $\widetilde f_1(z),u_{k+2},\widetilde f_2(z),w_{k+2}$, we see that $\xi^{(2)}(z)\cap \ell_{k-1}$ and $\xi^{(2)}(z) \cap \ell_{k+2}$ lie on the inside of $C$.

  Now assume $k$ is even and consider the eight points
  \[w_{k-2},u_k,w_{k+1},u_{k+3},u_{k-2},w_k,u_{k+1},w_{k+3}\]
  which are in this cyclic order (with $u_{k-2}$ and $u_{k+3}$ possibly switched if $p_3 = 3$).
  This time $\widetilde f_1$ is a path from $w_k$ to $u_{k+1}$ and $\widetilde f_2$ goes from $u_k$ to $w_{k+1}$, so we get the cyclic order
  \[w_{k-2},\widetilde f_2(z),w_{k+1},u_{k+3},u_{k-2},w_k,\widetilde f_1(z),w_{k+3}.\]
  As above, this shows $\xi^{(2)}(z) \cap \ell_{k-2}$ and $\xi^{(2)}(z) \cap \ell_{k+3}$ lie on the inside of $C$.
\end{Prf}

  \begin{center}
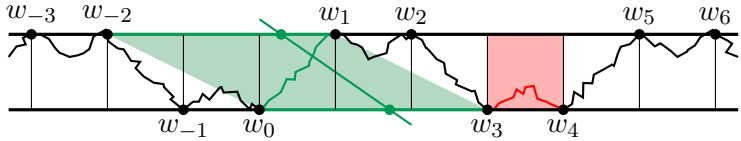

    \begin{tikzpicture}
      \begin{scope}
        \fill[red!30] (3,0) -- (3,1) -- (4,1) -- (4,0) -- cycle;
        \fill[ForestGreen!30] (0,0) -- (-2,1) -- (1,1) -- (3,0) -- cycle;

        \draw[thick,black,fuzzy] (-3.3,0.7) -- (-3,1) -- (-2.5,0.7) -- (-2,1) -- (-1,0) -- (-0.5,0.3) -- (0,0);
        \draw[thick,ForestGreen,fuzzy] (0,0) -- (1,1);
        \draw[thick,black,fuzzy] (1,1) -- (1.5,0.7) -- (2,1) -- (3,0);
        \draw[thick,red,fuzzy] (3,0) -- (3.5,0.3) -- (4,0);
        \draw[thick,black,fuzzy] (4,0) -- (5,1) -- (5.5,0.7) -- (6,1) -- (6.3, 0.7);
        \draw[thick,ForestGreen] (2,-0.2) -- (0,1.2);

        \draw[very thick,black] (-3.3,0) -- (0,0);
        \draw[very thick,black] (3,0) -- (6.3,0);
        \draw[very thick,black] (-3.3,1) -- (-2,1);
        \draw[very thick,black] (1,1) -- (6.3,1);
        \draw[ForestGreen,very thick] (0,0) -- (3,0);
        \draw[ForestGreen,very thick] (-2,1) -- (1,1);
        \foreach \i in {-1,0,3,4} \fill (\i,0) circle (0.07);
        \foreach \i in {-1,0,3,4} \node[anchor=north] at (\i,0) {$w_{\i}$};
        \foreach \i in {-3,-2,1,2,5,6} \fill (\i,1) circle (0.07);
        \foreach \i in {-3,-2,1,2,5,6} \node[anchor=south] at (\i,1) {$w_{\i}$};
        \foreach \i in {-3,-2,-1,0,1,2,3,4,5,6} \draw (\i,0) -- (\i,1);

        \fill[ForestGreen] (1.714,0) circle (0.07);     
        \fill[ForestGreen] (0.285,1) circle (0.07);     
      \end{scope}
    \end{tikzpicture}
    \captionof{figure}{The proof of \autoref{lem:intersections_on_intervals} and \autoref{lem:transversality_on_intervals} in the case $k = 0$ and $j = 3$. The line $\xi^{(2)}(z')$ (in green) can only intersect points in $M$ between $\ell_{-2}$ and $\ell_3$, but $\xi^{(1)}(z)$ is in the red region between $\ell_3$ and $\ell_4$.}
    \label{fig:transversality_proof}
  \end{center}

\begin{Lem}\label{lem:transversality_on_intervals}
  Assume $t_\rho \geq \tcrit$.
  If $z \in [z_j,z_{j+1}]$ and $z' \in [z_k,z_{k+1}]$ such that $\xi^{(1)}(z) \in \xi^{(2)}(z')$, then either $|j-k| \leq 1$ or $j$ and $k$ are even and $|j-k| = 2$.
\end{Lem}

\begin{Prf}
  We prove this by contradiction.
  Assuming the conclusion is false, we have one of three cases: either $|j-k| > 1$ and $k$ is odd, or $|j-k| > 1$ and $j$ is odd, or $|j-k| > 2$ and $j,k$ are both even.
  
  First assume $|j-k| > 1$ and $k$ is odd.
  The lines $\ell_{k-1}$ and $\ell_{k+2}$ don't intersect in $M$, so $\overline M \setminus (\ell_{k-1} \cup \ell_{k+2})$ has two connected components.
  The two components of $S^1 \setminus \{z_{k-1},z_{k+2}\}$ are mapped into these two different components by $\xi^{(1)}$.
  So in particular $\xi^{(1)}(z)$ and $\xi^{(1)}(z')$ are in opposite components of $\overline M \setminus (\ell_{k-1} \cup \ell_{k+2})$.
  If $\xi^{(1)}(z) \in \xi^{(2)}(z')$ then $\xi^{(2)}(z')$ is the line through $\xi^{(1)}(z)$ and $\xi^{(1)}(z')$.
  As $\overline M \cap \xi^{(2)}(z')$ is connected, this means that $\xi^{(2)}(z')$ intersects either $\ell_{k-1}$ or $\ell_{k+2}$ within $\overline M$.
  But this contradicts \autoref{lem:intersections_on_intervals}.

  If instead $|j-k| > 1$ and $j$ is odd, we repeat the same argument with $j$ in place of $k$.
  
  If $|j-k| > 2$ and $j,k$ are even we also use the same argument, but with $\ell_{k-2}$ and $\ell_{k+3}$ instead of $\ell_{k-1}$ and $\ell_{k+2}$.
\end{Prf}

  Recall that $I = [z_3,z_0] = [(s_2s_3s_1)_+,(s_1s_2s_3)_+]$, and let $J = [(s_2s_3s_1)_-,(s_1s_2s_3)_-]$, and $K = [z_1,z_2]$.
    We can see in \autoref{fig:hyperbolic_intervals} that $K \subset J$.

  \begin{Lem}\label{lem:transversality_Gamma_orbit}
    Let $A \subset S^1 \times S^1$ be a $\Gamma$--invariant subset which contains $I \times K$, $I' \times K'$, $I'' \times K''$, as well as $K \times I$, $K' \times I'$, and $K'' \times I''$. Then $A$ contains every pair $(x,y) \in S^1 \times S^1$ of distinct points which are not the two fixed points of a conjugate of $s_1s_2s_3$.
  \end{Lem}

  \begin{Prf}
    A pair $(x,y)$ of distinct points in $S^1$ defines an oriented geodesic $xy$ in the hyperbolic plane.
    As we noted in \autoref{sec:hyperbolic}, it intersects one of the ``altitude triangles'' bounded by Coxeter axes  (axes of elements conjugate to $s_1s_2s_3$).
    Hence some $\Gamma$--translate of either $(x,y)$ or $(y,x)$ is contained in $I \times J$, or $I' \times J'$, or $I'' \times J''$; see \autoref{fig:hyperbolic_intervals}.
    Due to the symmetry of the assumptions of the lemma, we can assume $(x,y) \in I \times J$, and in fact $(x,y) \in I^\circ \times J^\circ$ if $xy$ is not a Coxeter axis.

    It remains to show $I^\circ \times J^\circ \subset A$.
    To do this, we will decompose $J^\circ$ into a union of translates of $K$, $K'$ and $K''$.
    Let $\eta = (s_1s_3s_2)^2$, and note using \autoref{fig:conic_mappings} that $s_1s_3z_2' = z_1$ and analogously $s_2s_1z_2'' = z_1'$.
    So the intervals $K$ and $s_1s_3K'$ share an endpoint, as do $K'$ and $s_2s_1K''$.
    Therefore,
    \[[\eta z_2, z_2] = s_1s_3s_2s_1K'' \cup s_1s_3K' \cup K\]
    Now $I' = [z_2,z_3]  \subset [z_2,z_4] = s_2s_1[z_0,z_2] = s_2s_1I''$, so
    \[I \subset s_1s_3 I' \subset s_1s_3s_2s_1I'' \subset \eta I.\]
    So not only $I \times K \subset A$ but also $I \times s_1s_3K' \subset s_1s_3(I' \times K') \subset A$ and $I \times s_1s_3s_2s_1K'' \subset s_1s_3s_2s_1(I'' \times K'') \subset A$. Together, this gives $I \times [\eta z_2,z_2] \subset A$.
    Similarly, we find that $I \times [\eta^{k+1}z_2,\eta^kz_2] \subset \eta^k(I \times [\eta z_2,z_2]) \subset A$, for all $k \in \bN$.

    Let $L = (\eta_+,z_2]$ be the union of the sequence $[\eta^{k+1}z_2,\eta^kz_2]$ of adjacent intervals.
    Then $L \cup s_1L = (\eta_+,s_1\eta_+) = J^\circ$, so using $s_1I = I$ we obtain $I \times J^\circ \subset A$.
  \end{Prf}

  \begin{Thm}\label{thm:barbot_transverse}
    Let $\rho \colon \Gamma \to \SL(3,\bR)$ be a representation of type $(\frac{p_1-1}{2}, \frac{p_2-1}{2}, \frac{p_3-1}{2})$ with parameter $t_\rho \in (0,\tcrit^{-1}] \cup [\tcrit,\infty)$, and $\xi \colon S^1 \to \F$ the $\rho$--equivariant continuous boundary map from \autoref{prop:limit_curve_in_flags}.

      Then $\xi^{(1)}$ and $\xi^{(2)}$ are injective.
      If $t_\rho \not\in \{\tcrit^{-1},\tcrit\}$, then $\xi(x)$ and $\xi(y)$ are moreover transverse for every distinct pair $x,y \in S^1$, so $\rho$ is an Anosov representation.
  \end{Thm}

  \begin{Prf}
    We can assume that $\rho$ is irreducible, as we already know it is Anosov otherwise (see \autoref{sec:Anosov_representations}, in particular \autoref{fact:anosov_properties}(ii)).
    We can also assume that $t_\rho > 1$, otherwise we consider the representation $\rho \circ \psi$ instead, as in the proof of \autoref{lem:parameter_less_than_1}.

    We first consider the case $t_\rho > \tcrit$.
    If $x \in I = [z_3,z_0]$ and $y \in K = [z_1,z_2]$ then $\xi(x)$ and $\xi(y)$ are transverse by \autoref{lem:transversality_on_intervals}.
    The same is true if $(x,y)$ is in $I' \times K'$ or $I'' \times K''$.
    Transversality is symmetric, so we can apply \autoref{lem:transversality_Gamma_orbit} to extend this to all pairs $(x,y)$ with $x \neq y$, unless $x$ and $y$ are the fixed points of a conjugate $\gamma$ of the Coxeter element.
    But if $x = \gamma_+$ and $y = \gamma_-$, then $\xi(x)$ and $\xi(y)$ are transverse by \autoref{prop:limit_curve_in_flags} since $\gamma$ has distinct real eigenvalues by \autoref{lem:coxeter_eigenvalues}. \autoref{fact:irreducible_anosov} then shows that $\rho$ is Anosov.

    Now assume $t_\rho = \tcrit$.
    We know that $\xi^{(1)}(\gamma_+) \neq \xi^{(1)}(\gamma_-)$ for every conjugate $\gamma$ of the Coxeter element.
    So by \autoref{lem:transversality_Gamma_orbit}, to prove injectivity of $\xi^{(1)}$ it suffices to show $\xi^{(1)}(I)$ does not intersect $\xi^{(1)}(K)$.
    By \autoref{prop:limit_map_exists} $\xi^{(1)}(I) \subset \bx$, and $K = [z_0,z_2] \cap [z_1,z_3] = I'' \cap s_1I''$, so $\xi^{(1)}(K) \subset \rho(s_1)\bxpp \cap \bxpp$.
    But by \autoref{lem:box_in_ch} $\bxpp$ does not intersect $\ell_3$ and $\rho(s_1)\bxpp$ does not intersect $\ell_0$, so $\xi^{(1)}(K)$ is in one component of $\RP^2 \setminus (\ell_0 \cup \ell_3)$.
    Also by \autoref{lem:box_in_ch}, $\ell_0$ and $\ell_3$ intersect $\bx$ only on its boundary, so $\xi^{(1)}(I)$ is in the closure of a component of $\RP^2 \setminus (\ell_0 \cup \ell_3)$.
    These components cannot be the same (see e.g. \autoref{fig:transversality_proof}), hence $\xi^{(1)}$ is injective.
    Using the constructions in \autoref{sec:duality}, the result extends to $\xi^{(2)}$ and to $t_\rho = \tcrit^{-1}$.
\end{Prf}

\begin{Prf}[of \autoref{thm:main}]
  If $\rho \colon \Gamma \to \SL(3,\bR)$ is not a Coxeter representation, it has a finite image, so it cannot be Anosov.
  If $\rho$ is Anosov, it is in the Hitchin or Barbot component by \autoref{prop:non_anosov}, and by \autoref{fact:anosov_properties}(iv) $\rho(\gamma)$ has distinct real eigenvalues for every $\gamma \in \Gamma$ of infinite order, in particular for $\gamma = s_1s_2s_3$.

  Conversely, if $\rho$ is in the Hitchin component it is Anosov by \cite{ChoiGoldman, Labourie}.
  If $\rho$ is in the Barbot component and $\rho(s_1s_2s_3)$ has distinct real eigenvalues, then it has type $(\frac{p_1-1}{2},\frac{p_2-1}{2},\frac{p_3-1}{2})$ and $t_\rho \in (0,\tcrit^{-1}) \cup (\tcrit,\infty)$ by \autoref{lem:coxeter_eigenvalues}.
  So $\rho$ is Anosov by \autoref{thm:barbot_transverse}.
\end{Prf}






\KOMAoption{fontsize}{10pt}
\printbibliography

\vspace{1cm}

\textsc{Gye--Seon Lee} \\
Department of Mathematical Sciences and Research Institute of Mathematics, Seoul National University, Seoul 08826, Republic of Korea \\
\emph{Email:} \texttt{gyeseonlee@snu.ac.kr}

\textsc{Jaejeong Lee} \\
Research Institute of Mathematics, Seoul National University, Seoul 08826, Republic of Korea \\
\emph{Email:} \texttt{jaejeong@snu.ac.kr}

\textsc{Florian Stecker} \\
Department of Mathematics, The University of Texas at Austin, 2515 Speedway, Austin, TX 78712 \\
\emph{Email:} \texttt{math@florianstecker.net}

\end{document}